\documentclass[11pt]{amsart}

\title{The ghost stairs stabilize to sharp symplectic embedding obstructions}

\author{Dan Cristofaro-Gardiner}
\thanks{DCG partially supported by NSF grant DMS-1402200}
\address{Department of Mathematics, Harvard University and University of California, Santa Cruz}
\email{gardiner@math.harvard.edu}
\author{Richard Hind}
\thanks{RH partially supported by the Simons Foundation under grant \#317510}
\address{Department of Mathematics, Notre Dame University}
\email{hind.1@nd.edu}
 \author{Dusa McDuff}
  \thanks{DM partially supported by NSF grant DMS-1308669}
\address{Department of Mathematics,
 Barnard College, Columbia University}
\email{dusa@math.columbia.edu}
\date{February 10, 2017}

\usepackage{amssymb}
\usepackage{latexsym}
\usepackage{amsmath}
\usepackage{amsthm}
\usepackage{amscd}
\usepackage{mathabx}
\usepackage{graphicx}

\newtheorem{theorem}{Theorem}[subsection]
\newtheorem{question}[theorem]{Question}
\newtheorem{proposition}[theorem]{Proposition}
\newtheorem{prop}[theorem]{Proposition}
\newtheorem{corollary}[theorem]{Corollary}
\newtheorem{lemma}[theorem]{Lemma}

\newtheorem{lemma-definition}[theorem]{Lemma-Definition}

\numberwithin{figure}{section}
\numberwithin{equation}{subsection}

\theoremstyle{definition}
\newtheorem{definition}[theorem]{Definition}
\newtheorem{defn}[theorem]{Definition}
\newtheorem{remark}[theorem]{Remark}
\newtheorem{rmk}[theorem]{Remark}

\newtheorem{example}[theorem]{Example}

\newtheorem{cor}[theorem]{Corollary}

\newcommand{\eqdef}{\;{:=}\;}

\newcommand{\im}{{\rm im}}

\newcommand{\er}{{\Diamond}}
\newcommand{\less} {{\smallsetminus }}

\newcommand{\MS}{{\medskip}}
\newcommand{\NI}{{\noindent}}
\newcommand{\eps}{{\varepsilon}}
\newcommand{\ind}{{\rm ind}}
\newcommand{\gr}{{\rm gr}}

\renewcommand{\Tilde}{\widetilde}
\newcommand{\ov}{\overline}
\newcommand{\La}{{\Lambda}}
\newcommand{\Mm}{{\mathcal M}}
\newcommand{\Nn}{{\mathcal N}}
\newcommand{\Jj}{{\mathcal J}}
\newcommand{\Bb}{{\mathcal B}}
\newcommand{\Ii}{{\mathcal I}}
\newcommand{\Cc}{{\mathcal C}}
\newcommand{\Ee}{{\mathcal E}}
\newcommand{\Ll}{{\mathcal L}}
\newcommand{\p}{{\partial}}
\newcommand{\C}{{\mathbb C}}

\newcommand{\R}{{\mathbb R}}
\newcommand{\T}{{\mathbb T}}

\newcommand{\Z}{{\mathbb Z}}

\newcommand{\Aa}{\mathcal{A}}
\newcommand{\op}{\operatorname}

\newcommand{\Ker}{\op{Ker}}

\newcommand{\om}{\omega}
\newcommand{\al}{\alpha}
\newcommand{\be}{\beta}
\newcommand{\de}{\delta}
\newcommand{\De}{\Delta}
\newcommand{\ga}{\gamma}
\newcommand{\la}{\lambda}
\newcommand{\ka}{\kappa}

\newcommand{\Si}{\Sigma}
\newcommand{\si}{\sigma}

\newcommand{\se} {{\stackrel{s}\hookrightarrow}}
\renewcommand{\epsilon}{\varepsilon}

\begin{document}

\begin{abstract}

In determining when a four-dimensional ellipsoid can be symplectically embedded into a ball, McDuff and Schlenk found an infinite sequence of ``ghost" obstructions that generate an infinite ``ghost staircase" determined by the even index Fibonacci numbers.  The ghost obstructions are not visible for the four-dimensional embedding problem because strictly stronger obstructions also exist.  We show that in contrast, the embedding constraints associated to the ghost obstructions are sharp for the stabilized problem; moreover, the corresponding optimal embeddings are given by symplectic folding.  The proof introduces several ideas of independent interest, namely: (i) an improved version of the index inequality familiar from the theory of embedded contact homology (ECH), (ii) new applications of relative intersection theory in the context of neck stretching analysis, (iii) a new approach to estimating the ECH grading of multiply covered elliptic orbits in terms of areas and continued fractions, and (iv) a new technique for understanding the ECH of ellipsoids by constructing explicit bijections between certain sets of lattice points.

\end{abstract}

\maketitle

\tableofcontents

\section{Introduction}
\subsection{Background}

Let $(M_1,\omega_1), (M_2,\omega_2)$ be symplectic manifolds.  A {\em symplectic embedding}
\[ (M_1,\omega_1)\; \se \; (M_2,\omega_2)\]
is a smooth embedding $\Psi: M_1 \to M_2$ such that $\Psi^* \omega_2 = \omega_1$.  It can be a difficult problem to determine whether or not one symplectic manifold can be embedded into another; this is particularly true when the manifolds have the same dimension\footnote{If $\dim(M_2) \ge
\dim(M_1) + 2$ and $M_1$ is open, versions of Gromov's $h$-principle apply.}.

In fact, even deciding when one {\em symplectic ellipsoid}
\[ E(a_1,\ldots,a_n) = \left \lbrace \pi \frac{|z_1|^2}{a_1} + \ldots + \pi \frac{|z_n|^2}{a_n} < 1 \right \rbrace \subset \mathbb{C}^n,\]
can be embedded into another is largely open.  Hofer conjectured a purely combinatorial criteria for settling the $n=2$ case, and McDuff proved this in \cite{M1}.  In higher dimensions, Buse and Hind have constructed some optimal embeddings of ellipsoids into balls \cite{BH}, but a complete understanding of the problem seems far off.

In \cite{CGH}, the authors began the study of the function\footnote{In fact, an optimal embedding can be realized in all cases where the value of $c_k(x)$ is known. Optimal 4-dimensional embeddings exist by Corollary 1.6 in \cite{Mell} and this covers the cases when $c_k(x)=c_0(x)$. The folding maps which give 
 our other cases are quite explicit, 
and we can get optimal embeddings from Theorem 4.3 in \cite{pelayongoc}.}
\[ c_k(x) = \op{inf}  \lbrace \mu\hspace{1 mm} | \hspace{1 mm} E(1,x) \times \mathbb{C}^k \se B^4(\mu) \times \mathbb{C}^k \rbrace,\]
where $B^4(\mu) = E(\mu,\mu)$ is the $4$-ball of capacity $\mu$.
This is the stabilized version of the function $c_0(a)$, which was computed by McDuff and Schlenk in \cite{MS}.  It is a version of the ellipsoid embedding problem in which most of the arguments are infinite.

The main theorem of  \cite{CGH} states that
\[ c_k(x) = c_0(x), \quad 1 \le x \le \tau^4,\]
where $\tau \eqdef \frac{1+\sqrt{5}}{2}$ denotes the Golden Mean.  The function $c_k(x)$ is currently unknown for $x > \tau^4$; in fact, it is not even known whether or not this function depends on $k\ge 1$.  It {\em is} known, however, that
\begin{equation}
\label{eqn:foldingbound}
c_k(x) \le \frac{3x}{x+1},\qquad x>\tau^4,
\end{equation}
because of an explicit ``symplectic folding construction'' given by 
Hind in \cite{Hi}\footnote{
Actually the construction in \cite{Hi} only applied to compact subsets of the stabilized ellipsoid.
To embed the whole product we are appealing to  \cite{pelayongoc}.}.  In particular,  because
the volume bound gives $c_0(x) \ge \sqrt x$,  we have $c_k(x) < c_0(x)$ for $x > \tau^4$.
It is then natural to ask the following:
\begin{question}
\label{que:mainquestion}
Is it the case that
\begin{equation}
\label{eqn:conjecturedequation}
c_k(x) = \frac{3x}{x+1}
\end{equation}
for $x > \tau^4$ and $k \ge 1$?
\end{question}

An affirmative answer to Question~\ref{que:mainquestion} would imply that the stabilized embedding problem is quite rigid: all of the optimal embeddings would  be given
 either by stabilizing four-dimensional embeddings as in \cite{CGH}, or by Hind's generalization of symplectic folding.  On the other hand, an outcome in the negative would require the existence of as yet unknown embeddings.  In their proof \cite{HK2} that $c_k(x)$ is asymptotic to $3$ as $x \to \infty$, Hind and Kerman showed that \eqref{eqn:conjecturedequation} holds for
all  integers of the form $3g_n-1$, where $g_n$ is an odd index Fibonacci number.

\subsection{The ghost stairs}

One of the more mysterious aspects of McDuff and Schlenk's computation of $c_0(x)$ is the ``ghost stairs'', which we now review.

Recall first of all that the function $c_0(x)$ is particularly intricate for $1 \le x \le \tau^4$.  Here, it is given by an infinite staircase determined by the odd index Fibonacci numbers $g_\bullet = (1,2,5, 13, 34,\dots)$, called the ``Fibonacci staircase''.  For $\tau^4 \le x \le 7$, the function $c_0(x)$ is seemingly simpler ---  it turns out that
\begin{equation}
\label{eqn:c0a}
c_0(x) = \frac{x+1}{3}
\end{equation}
for $x$ in this range.  Nevertheless, McDuff and Schlenk show that there is a kind of analogue of the Fibonacci staircase underlying \eqref{eqn:c0a}, which they call the {\em ghost stairs}.

The idea is that the deviation of the Fibonacci staircase from the classical volume constraint when $x<\tau^4$ comes from a sequence of sharp obstructions, one for each of the embedding problems
\[ E\left(1,\frac{g_{n+2}}{g_n}\right) \to B^4(\mu).
\]
These obstructions imply that $$
c_0\left(\frac{g_{n+2}}{g_n}\right) = \frac{g_{n+2}}{g_{n+1}}   = \frac {3\frac{g_{n+2}}{g_n}}{\frac{g_{n+2}}{g_n}+1},
$$
where the second equality holds by the Fibonacci identity $3g_{n+1} = g_n + g_{n+2}$.
One can write down analogous obstructions for the problem
\begin{equation}
\label{eqn:ghostproblem}
E(1,b_n) \to B^4(\mu),
\end{equation}
where the $b_n, n\ge 0,$ are determined by the {\em even} index Fibonacci numbers
\begin{equation}\label{eq:evenFib}
h_\bullet: = (1,3,8,21,55,\dots)
\end{equation}
 via the formula $b_n = \frac{h_{2n+3}}{h_{2n+1}}$.\footnote
 {
The numerical properties of the ratios $b'_s = \frac{h_{2s+2}}{h_{2s}}$ are not the same --- for example, the even index terms in the sequence $h_\bullet$ are all divisible by $3$ --- and so their role in \cite{MS} is somewhat different.  However, the $b'_s$ do come up in our arguments in \S\ref{sec:heart},
since the class $z_M$ is determined by the ratio $\frac {\ell_n}{\ell_{n-1}} = b'_n$.
 }
   Thus $b_0 = 8, b_1 = \frac {55}8, b_2 = \frac{377}{55}$, and so on; the $b_n$ are decreasing, with limit $\tau^4$.  Since
$h_{n+3} = 3h_{n+2} - h_{n+1}$,
we again obtain the estimate $c_0(x)\ge \frac{3x}{x+1}$ for $x=b_n, n \ge 0$; moreover, as explained in \cite[\S 4.3]{MS}, the obstructions at the $b_n$ fit together to form an infinite staircase converging to $\tau^4$ from the right.   However, in dimension $4$ these obstructions are {\em not} sharp at $b_n$, since as mentioned above they are weaker than the volume obstruction,\footnote
{
In fact, the graphs of $y = \sqrt x$ and $y=\frac{3x}{x+1}$  cross at $x=\tau^4$.}
 and so do not influence $c_0(b_n)$ directly.
 It is for this reason that McDuff and Schlenk call them {\em ghost stairs}.

Our main result is that the embedding obstructions at the $b_n$
persist under stabilization.
Because of the symplectic folding bound \eqref{eqn:foldingbound} they are sharp, so that we obtain the following.
\begin{theorem}
\label{thm:main}
$c_k(b_n) = \frac{3 b_n} {b_n + 1}$ for all
$k \ge 1$ and $n\ge 0$.
\end{theorem}
Thus  $c_k(8) = \frac 8{21}$, $c_k(\frac{55}8) = \frac 8{21}$ and so on.

\subsection{Methods and relationship with embedded contact homology}

In view of the upper bound in \eqref{eqn:foldingbound},  Theorem~\ref{thm:main} will follow if we can establish a suitable lower bound for $c_k(x)$ at the given values of $x$.  In other words, we must find embedding obstructions for  these  $x$.  As in \cite{HK,CGH} this is
 accomplished by a two-step process:
 \begin{itemize}\item first, we find suitable $J$-holomorphic curves that
 obstruct the existence of a four-dimensional embedding $E(1,x)\;\se\; B^4(\mu)$ where $\mu< \frac{3x}{x+1}$, and
 \item second, we show that these obstructions persist for  stabilized embeddings $$
 E(1,x)\times \C^{k}\;\se\;  B^4(\mu)\times \C^{k}.
 $$
 \end{itemize}
 Although we are interested here in calculating $c_k(x)$ for the rational numbers $b_n$, it is convenient to increase $b_n$ slightly to  $x = b_n+\eps = \frac{p}{q} + \eps$,
 where $\eps>0$ is very small and irrational, so that there are only two periodic orbits on the boundary of the ellipsoid.
 In four dimensions it is also often convenient to compactify $B^4(\mu)$ to $\C P^2(\mu)$ by  adding the line at infinity.\footnote
 {
 In fact if the domain is an ellipsoid the two embedding problems are equivalent: see \cite{Mell}.  Here
 $\C P^2(\mu)$ denotes
 $\C P^2$ with symplectic form $\om$ scaled so that $\om$ takes the value $\mu$ on the line $L$.}
Thus,  for the first step we  consider the negative completion $\ov X:= \ov X_{\mu,x} $ of
 $\C P^2(\mu)\less {\rm im}\Phi$, where
 \begin{equation}\label{eq:phi}
  \Phi: E(1,x)\;\se\; \C P^2(\mu)
  \end{equation}
   is a symplectic embedding for some $\mu$, and  look for $J$-holomorphic  curves\footnote
 {
 Here we assume that $J$ is admissible, i.e. adapted to the negative end of $\ov X$: see \S\ref{sec:prelim} for more details.}
  $C$ in $\ov X$ of degree $d$ that are negatively asymptotic  to the short orbit $\be_1$ on $\Phi(\p E(1,x))$ with some multiplicity
 $m$.    If such a curve exists for generic $J$ and all sufficiently small $\eps>0$,
  then  the fact that it must have positive symplectic area gives the inequality
  \begin{equation}
  \label{eqn:areaineq}
  \mu> \frac m d.
  \end{equation}

When proving the existence of $C$ we will restrict to the case when $C$ has Fredholm index zero,
  since these are the curves that can potentially be counted, and will also work with a fixed value $\mu_\star$ of $\mu$.
  As we explain in more detail below, it turns out that the second stabilization step
  works for curves $C$ that have genus zero, Fredholm index zero and just one negative end of multiplicity $m= p$.

The second step is accomplished by the method of \cite{HK, CGH}, who prove a result that can  be stated as follows.   It will be convenient to denote  by
  $$
  \Mm\bigl(\ov X_{\mu,x}, dL,s, \{(\be_1,m_1),(\be_2,m_2)\}\bigr)
  $$
  the moduli space of  genus zero $J$-holomorphic  curves  with degree $d$ and $s$ negative ends, that cover the short orbit $\be_1$ a total of $m_1$ times and the longer orbit $\be_2$  a total of $m_2$ times.  Here we assume that $x = \frac pq + \eps$ as above.

It turns out (see \eqref{eqn:Xind}) that if $C$ has just one negative end of multiplicity $m_1 = m$ then its Fredholm index is
 \begin{equation}\label{eq:indC}
 \ind(C) = 2\bigr(3d - m - \lceil \frac mx \rceil\bigr).
 \end{equation}
 Hence, if $\ind(C) = 0$, we have $3d> m+\frac mx$.  If we now let $\eps\to 0$, we obtain $3d\ge  m+\frac {mq}p$  with equality exactly if $p|m$, since $\gcd(p,q) = 1$ by hypothesis.
Thus $$
\frac md \le \frac {3p}{p+q} = \frac{3b}{b+1},
$$
with equality exactly  if $p|m$.
In other words, the obstruction that index $0$ curves as above give through \eqref{eqn:areaineq} is never stronger than the folding bound, and so such curves could potentially persist for the stabilized embedding.  Our main stabilization result proves this when $m=p$.  (See Remark~\ref{rmk:counterex} for some  generalizations.)

 \begin{prop} \label{prop:goodC} Let $x:  = b + \eps$, where $b =  \frac pq$ with $\gcd(p,q) = 1$ and $\eps>0$ irrational and  very small, and fix $\mu_*>0$.  Suppose
that for all sufficiently small $\eps>0$ and generic  admissible $J$ there is a  genus zero curve $C$ in $\ov X_{\mu_\star,x}$ with degree $d$, Fredholm index zero, and one negative end on
$\{(\be_1, p)\}$,  where  $\gcd(d,  p) = 1$.
   Then, $3d=p+q$,
    and for all $k\ge 0$, we have $c_k(b)\ge \frac { p}d= \frac{3b}{b+1}$.
\end{prop}

The proof is given in \S\ref{sec:stablzn}: see  Proposition~\ref{prop:goodC1}.   The main point is that if $C$ has genus zero, Fredholm index zero, and one negative end, then its Fredholm index remains zero under stabilization.  Moreover the arguments in \cite{CGH}  guarantee that its contribution to the count of curves in the stabilization cannot be cancelled by some other curve even when one varies $\mu$ and the almost complex structure.

We end with some  comments on the proof of the first step.
 In \cite{CGH},
 the authors show  that when
 $x = \frac{g_{n+2}}{g_n} < \tau^4$ the embedding obstruction coming from  ``embedded contact homology'' (ECH)  is carried by a curve
with genus zero and one negative end as above.
  However, for $x =b_n> \tau^4$, it cannot be the case that the obstruction coming from ECH stabilizes.
Indeed,   \cite{CGHR} shows that embedded contact homology always at least recovers the volume obstruction $c_0(x) \ge \sqrt{x}$, which as already mentioned  is strictly above $\frac{3x}{x+1}$ for $x > \tau^4$.
Further,
  curves $C$ in four-dimensional cobordisms that are detected by the ECH cobordism map generally
have ECH index and Fredholm index equal to zero.
 However, in our case
 we will see that
 the relevant curves have ECH index {\em two} and Fredholm index zero, and hence cannot be
 expected to be
 embedded.\footnote
{
In fact, the ECH cobordism map detects buildings with ECH index zero that may (and often do) consist of curves with both positive and negative ECH index.
Further, it may not always be the case that curves with ECH index  two and Fredholm index zero must have double points, but as mentioned in Remark~\ref{rmk:writhexact}~(ii)
this is known in some situations.}

Thus, new methods are needed.  The basic idea in the present work is to stretch a collection of nodal curves that are modified forms of the McDuff--Schlenk obstructions that give the ghost stairs, and look at the top part of the resulting buildings.  Our aim is to show that at least one of the resulting buildings has a top level with Fredholm index zero and one negative end, and so by Proposition~\ref{prop:goodC} stabilizes to an index zero curve that gives an obstruction in higher dimensions.  We therefore need to
analyze the possible buildings that can arise when we stretch.  This analysis is complicated by the possible presence of negative index multiple covers --- configurations that most probably do occur,
see Remark~\ref{rmk:model} and Remark~\ref{rmk:8}.
To get around this difficulty, we use the fact that the modified
 McDuff-Schlenk obstructions lie in classes that have precisely $12$ genus zero representatives, and we show that at most
$9$ of these break in a problematic way.

To this end, we develop the tools used to analyze relative intersections and  ECH indices.
Specifically, we use a refined index inequality (Proposition~\ref{prop:indineq}), which is a reformulation of results in Hutchings \cite{hech}, together with a new approach to estimating the grading of elliptic orbits (which contributes to the ECH index in subtle ways) in terms of areas rather than lattice point counts, see Lemma~\ref{lem:fundest}.  We also use a technique pioneered by Hutchings and Nelson \cite{HN} that calculates writhes of curves that are close to breaking as the neck is stretched.  Situations requiring the analysis of potentially complicated holomorphic buildings are quite common in applications of holomorphic curve theory, see for example \cite{hlecture,HN,HT,N}, and so we expect our strategy to be potentially useful in other contexts.

This analysis of the limiting buildings forms the bulk of the paper.    It is described in \S\ref{sec:stab}, with the hardest computations deferred to \S\ref{sec:noconn} and \S\ref{sec:heart}.
 As we point out in Remark~\ref{rmk:n=0}~(ii), the same method works
 rather easily for the Fibonacci stairs $b = \frac{g_{n+2}}{g_n}$, while in the case of the ghost stairs it is  complicated by the presence of the obstruction curve that determines $c_0(x)$ for $\tau^4 < x < 7$.

Although in principle the methods we develop here could potentially be adapted to compute $c_k(x)$ for other values of $x$, as we explain in Remark~\ref{rmk:8} this is probably neither efficient nor the best approach for general $x$.  Indeed, many of the calculations here are simplified because of special properties of the Fibonacci numbers, and even with this the computations are quite involved.   We intend to explore other ways to construct suitable curves $C$ in a later paper.

\section{Preliminaries}
\label{sec:prelim}

This section reviews basic background material on continued fractions, Fibonacci identities, and the ECH index formulas.  

\subsection{Weight sequences and best approximations}
\label{sec:bestapprox}

 Beside the
 even index Fibonacci numbers $h_{\bullet}$ in \eqref{eq:evenFib}, the following auxiliary sequences will be useful, where 
 $Q_n = P_{n-1} = h_{2n+1}$:
\begin{align} \label{eq:ell}
& Q_0 = 1,\;\; Q_1= 8,\;\; Q_2 = 55, \;\; Q_3 = 377, \;\; Q_4 = 2584,\;\; Q_5 = 17711, \;\;   \ldots  \\  \notag
& \ell_{-1}: = 0, \;\;  \ell_0 = 1,\; \;\ell_1 = 7,\;\;\ell_2 = 48,\;\; \ell_3 = 329, \;\; \ell_4 = 2255, 
\;\dots\; \ell_n = \tfrac 13 h_{2n+2},\\ \notag
&  t_0 = 1,\; \;t_1 = 6,\;\; t_2 = 41,\;\; t_3 = 281, \, t_4 = 1926,\;\;  \quad    t_k: = \ell_k-\ell_{k-1} .
\end{align}
We also write
\[ b_n: = \frac{h_{2n+3}}{h_{2n+1}} = : \frac{P_n}{Q_n} = \frac{Q_{n+1}}{Q_n}.\]
We will use the following Fibonacci identities:
\begin{align}\label{eq:Fib1}
& 3\, \star_{n+1}  = \star_n + \star_{n+2},\quad \star = g,h;\\ \label{eq:Fibt}
& t_n - t_{n-1} = 5 \ell_{n-2},
\\ \label{eq:Fibq}
& Q_n = \ell_n + \ell_{n-1},
 \\ \label{eq:Fib2}
& h_{n+2}^2  - h_{n+1} h_{n+3} = 1; \\ \label{eq:Fibh}
& h_{2n+2}^2 - (3 h_{2n+2} -  h_{2n+3})\; =\; h_{2n+1} (h_{2n+3}- 1)+1.
\end{align}
Further, the $Q_n$, $\ell_n$ and $t_n$ are all  linear combinations of certain Fibonacci numbers,
and satisfy the recursion
\begin{equation}
\label{eqn:basicrecursion} 
\star_n = 7 \star_{n-1} - \star_{n-2},
\end{equation}
Hence their ratios $\frac{\star_n}{\star_{n-1}}$ converge to $\tau^4$.  
Moreover, the above identities 
may be proved  by checking them
 on two or three low values of $n$: because 
the Fibonacci numbers satisfy a two step linear recursion, one only needs to check   linear identities  for two values of $n$, and quadratic identities such as  \eqref{eq:Fib2} for three values of $n$: see \cite[Prop. 3.2.3]{MS}.

\begin{lemma}
\label{lem:fibid0} \begin{itemize}\item[{\rm (i)}] Let $x_n$ and $y_n$, $n\ge 0$, be two sequences that satisfy \eqref{eqn:basicrecursion}.  Then
the quantity $
x_n y_n - x_{n-1}y_{n+1} $ is independent of $n$.  Moreover the following identities hold:
\begin{align} \label{eq:recurrel}
&Q_{n+1} \ell_n =  Q_n \ell_{n+1} + 1, \\ \notag
& \ell_n\ell_n = \ell_{n-1}\ell_{n+1} + 1, \\ \notag
&\ell_nP_n = \ell_{n-1}P_{n+1} + 8.
\end{align}
\item[{\rm (ii)}]
The following sequences (and their products) are increasing with $n$:  
$$
 \frac {Q_n}{P_{n}}, \quad \frac {\ell_n}{\ell_{n+1}}, \quad \frac {t_{n+1}}{t_n},\quad    \frac {\ell_n}{P_{n}},\quad     \frac {\ell_n}{P_{n+1}}, \quad   \frac {Q_n}{\ell_{n}}, \quad 
 \frac {Q_{n+k}}{t_{n}}, \quad 0\le k \le 3. \;\;  
 $$

\end{itemize}
\end{lemma}
\begin{proof}  (i) holds because 

\[ x_n y_n - x_{n-1}y_{n+1} = (7x_{n-1} - x_{n-2})y_n - x_{n-1}(7 y_n - y_{n-1}) = x_{n-1}y_{n-1} - x_{n-2}y_n.\]
Hence $x_n y_n - x_{n-1}y_{n+1} = x_1 y_1 - x_{0}y_{2} =: \ka$ is constant.
Thus 
the quotient $\frac{x_n}{y_{n+1}}$  increases or  decreases according to whether the constant is positive or negative.  
Alternatively, (i) implies that to check whether one of these sequences is increasing or decreasing, one just has to look at the first two terms.
The rest of the lemma now holds by direct calculation.  Note that it suffices to check that $ \frac {Q_{n+3}}{t_{n}}$ increases because, if $0\le k < 3$, then
 $ \frac {Q_{n+k}}{t_{n}} =   \frac {Q_{n+3}}{t_{n}}\cdot  \frac {Q_{n+k}}{Q_{n+3}}$  is the product of two increasing sequences.
\end{proof}

Because the sequence $\frac{P_n}{Q_n}$ converges to $\tau^4 = \frac 12(7 + 3\sqrt 5)$  which is a solution of the equation $\tau^4 +  \tau^{-4} = 7$,
one can check that

\begin{align}
\label{eqn:lnpnlimit}
\lim_{n \to \infty} \frac{\ell_n}{P_n}  &= \lim_{n \to \infty} \left(\frac{P_n}{9P_n} + \frac{Q_n}{9P_n}\right) = \sigma =: \frac 16(3-\sqrt 5) < 0.128, \\ \notag
\lim_{n \to \infty} \frac{\ell_n}{Q_n} & = \lim_{n \to \infty} \frac{\ell_n}{P_n}\;\frac {P_n}{Q_n} =  1-\si,\\ \notag
\lim_{n \to \infty} \frac{t_n}{Q_n} & = \lim_{n \to \infty} \frac{\ell_n}{Q_n} - \lim_{n \to \infty} \frac{\ell_{n-1}}{P_{n-1}} = 1-2\si > 0.745.
\end{align}

\MS

\NI{\bf Weight sequences:}
As explained in \cite[Lem.1.2.6]{MS} for example,  the weight sequence for $b = \frac pq$ is a nonincreasing finite sequence of positive numbers in $\frac 1q\Z$ such that
\begin{align}\label{eq:weight}
w(\tfrac pq) = (w_1,\dots w_m),  &\quad  W(\tfrac pq): = q\, w(\tfrac pq) = (W_1,\dots,W_m), \;\;\mbox{ where }\\ \notag
W_m = 1,& \quad \sum _i W_i^2 = pq, \quad \sum_i W_i = p + q -1.
\end{align}
If $b$ has continued fraction expansion $[a_0,a_1,\dots ,a_k],$ then the weights $W(b)$ occur in blocks of lengths $a_0,a_1,\dots,a_k$.  Hence $m = \sum a_i$
and we may write
\begin{align}\label{eq:weight0}
W(b) = (X_0^{\times a_0}, X_1^{\times a_1}, \dots, X_k^{\times a_k}), & \quad X^{\times a}: = \underbrace{X,\dots,X}_a.
\end{align}
In this notation, given $b:=\frac pq$ with $\gcd(p,q)=1$, the corresponding  $X_i$ and $a_i$ are determined for increasing $i$  by  the recursion 
$$
X_{-1} = p, \quad X_0 = q, \quad X_{i+1} = X_{i-1}-a_iX_i \;\; i<k,
$$
where the $a_i>0$ are chosen so that $0\le X_{i+1}<X_i$.   On the other hand, given a continued fraction
$[a_0,a_1,\dots,a_k]$, we can calculate which number $b = \frac pq$ it represents by using the same recursion but starting at the end with $X_k =1$.
This recursion implies that  $p = a_0 X_0 + X_1$ and hence that 
\begin{align}\label{eq:weight1}
W(b) &= \left(X_0^{\times a_0}, \dots, X_r^{\times a_r}, W\left(\frac{X_r}{X_{r+1}}\right)\right), \quad \forall r\ge 0.
\end{align}
The relevance of weight expansions to our embedding problem is this result from \cite{Mell}. 

\begin{prop}\label{prop:ell}   Let $w(b)$ be the weight expansion of $b = \frac pq$.  Then for any $\eps>0$ one can embed $m$ disjoint balls of capacities 
$(1-\eps)w(b)$ into ${\rm int\,} E(1,b)$, and hence remove 
 almost all of the interior of an ellipsoid $E(1,b)$ by blowing  it up $m$ times with weights $(1-\eps)w_i$.  
\end{prop}

The following lemma is helpful when finding continued fraction expansions.

\begin{lemma}\label{le:weight} Let $S_0, S_1,\dots$ be a (strictly) increasing sequence of positive integers with $7 > \frac {S_1}{S_0}> \tau^4$, that satisfy 
the recursion \eqref{eqn:basicrecursion}.  Then 
 there are positive integers $a_1,\dots, a_k$ for some $k\ge 0$ such that
 $$
 \frac{S_{n+1}}{S_n} = [6; (1,5)^{\times n},a_1,\dots,a_k],\quad  \forall n\ge 0.
 $$ 
\end{lemma}

\begin{proof}  
We give an inductive argument.  The case $n=0$   
holds because $6< \frac{S_1}{S_0} < 7$.

We next claim that the sequence $ \frac{S_{n+1}}{S_n}$ is decreasing with limit $\tau^4$.  This holds by applying part (i) of Lemma~\ref{lem:fibid0} with $x_n= y_n = S_n$, and using the fact that 
$\frac {S_1}{S_0} : = x> \tau^4$ so that $x^2 - 7x +1>0$.  In particular, for all $n$, $6 < \frac{S_{n+1}}{S_n} < 7$, so the first entry of its continued fraction expansion is $6$, and 
$\frac{S_{n+1}}{S_n}\ge \frac{13}{2}$, hence $2(S_{n+1}-6S_n) > S_{n}$ so that
 the continued fraction expansion of $\frac{S_{n+1}}{S_n}$ has the form  
$
[6,1, \dots]$. Thus by \eqref{eq:weight1} we have
\begin{align}\label{eqn:WS}
W\Bigl(\frac {S_{n+1}}{S_n}\Bigr)  &= \Bigl( S_n^{\times 6}, S_{n+1}-6S_n = S_n-S_{n-1}, W(\frac{S_n-S_{n-1}}{S_{n-1}})\Bigr) .
\end{align}
But, by induction, we may assume that $\frac {S_{n}}{S_{n-1}} = [6,(1,5)^{\times(n-1)},a_1,\dots, a_k]$.  Therefore,
$\frac {S_{n}-S_{n-1}}{S_{n-1}} = [5,(1,5)^{\times(n-2)},a_1,\dots, a_k]$.  
Hence, because the continued fraction for $\frac {S_{n+1}}{S_n} $ is given by the length of the blocks in
its weight expansion, we find that
$$
\frac {S_{n+1}}{S_n}= [6,(1,5)^{\times n}, a_1,\dots, a_k],
$$
as claimed.
\end{proof}

\begin{corollary} \label{cor:weight}   For $n \ge 1$,  we have the following weight expansions.
\begin{align*} 
{\rm (a)}\qquad &b_n  = \frac{Q_{n+1}}{Q_n} = [6; (1,5)^{\times (n-1)}, 1, 7],\\ 
{\rm(b)}\qquad& \frac{\ell_{n+1}}{\ell_{n}}  = [6; (1,5)^{\times n},1] = [6; (1,5)^{\times (n-1)}, 1, 6],\\
{\rm (c)}\qquad& \frac{t_{n+1}}{t_{n}}  = 
[6; (1,5)^{\times n}].
\end{align*}
\end{corollary}
\begin{proof}  Since the sequences  $Q_n, \ell_n, t_n$ satisfy \eqref{eqn:basicrecursion}, this an immediate consequence of
Lemma~\ref{le:weight}.  
\end{proof}

In \S\ref{sec:heart} we will need the following variants of the quadratic formula  $\sum W_i^2 = pq$ in \eqref{eq:weight}.  The first involves a vector $z_M(n)$ that is part of the data of a ``model curve" that we will study.

\begin{lemma}\label{lem:WcM} For $n\ge 1$, define $z_M(n)$ to be the vector  $6W(\frac{\ell_n}{\ell_{n-1}})$ with $7$ ones appended at the end.  Thus $z_M(n)$ has the same length as $W(b_n)$, and has the following expansion
$$
z_M(n) = \bigl( (6\ell_{n-1})^{\times 6},  6 t_{n-1}, (6\ell_{n-2})^{\times 5},\dots, 6 t_0, 1^{\times 7}\bigr).
$$
Then
\begin{align}\label{eqn:WcM}  z_M(n)\cdot W(b_n) = \ell_{n-1}Q_n + t_n Q_{n+1} - 1.
\end{align}
\end{lemma}
\begin{proof}  When $n=1$, $z_M(1) = (6^{\times 6}, 6, 1^{\times 7})$ and the claim is that
$$
z_M(1)\cdot W(\frac{55}{8}) = 
8 + 6\times 55 - 1 = 337.
$$
    But
$$
z_M(1)\cdot W(\frac{55}{8})  = (6^{\times 6}, 6, 1^{\times 7})\cdot (8^{\times 6}, 7, 1^{\times 7}) = 36\times 8 + 42 + 7 = 337.
$$
Thus we may assume inductively that the result is known for $n-1\ge 1$ and consider the case $n$.  As in \eqref{eqn:WS}, we may write
$$
 z_M(n) =  \bigl( (6\ell_{n-1})^{\times 6},  6 t_{n-1},  c'_M(n-1)\bigr),
 $$
 where $ c'_M(n-1)$ is the truncated version of $ z_M(n-1)$ in which the first entry $6\ell_{n-2}$ is removed.
Since $W(b_n)$ has an analogous expression, we find that
\begin{align*}
 z_M(n)\cdot W(b_n) &= 36 \ell_{n-1} Q_n + 6 t_{n-1} (Q_n - Q_{n-1})  +  c'_M(n-1)\cdot W\bigl(\frac{P_{n-1} - Q_{n-1}}{Q_{n-1} }\bigr)\\ 
& = 36 \ell_{n-1} Q_n + 6 t_{n-1} (Q_n - Q_{n-1}) + ( \ell_{n-2} Q_{n-1} + t_{n-1} Q_n - 1)\\
&\hspace{2in} - 6\ell_{n-2}Q_{n-1} \\
& = Q_n ( 7 t_{n-1} + 36 \ell_{n-1} ) - Q_{n-1} ( 5 \ell_{n-2} - 6t_{n-1} ),
\end{align*}
where the second equality is obtained using the inductive hypothesis.
Hence we must show that the right hand side of the last equation equals $ \ell_{n-1}Q_n + t_n Q_{n+1} - 1$.
If we write $Q_{n+1} = 7Q_n - Q_{n-1}$ and gather the terms in $Q_n, Q_{n-1}$ on different sides of the equation,
 we find that it suffices to show
$$
Q_n\bigl(7 t_n - 35 \ell_{n-1} -  7t_{n-1}   \bigr) = Q_{n-1}\big( t_n - 6t_{n-1} -5 \ell_{n-2}  \bigr).
$$
But the coefficient on the left vanishes because  $t_n - t_{n-1} = 5 \ell_{n-1}$ by \eqref{eq:Fibt}, while the same identity shows that the coefficient on the right
also vanishes because  $ t_n - 6t_{n-1} = t_{n-1} - t_{n-2} - 5\ell_{n-2} = 0$.
\end{proof}  

\begin{remark}
We can instead write
\begin{equation}
\label{eqn:cmdotW}
z_M(n) \cdot W(b_n) = \ell_n^2 + 41\ell_n\ell_{n-1} - 5 \ell_{n-1}^2 + 6.
\end{equation}
For our purposes, the identity \eqref{eqn:WcM} is more geometrically natural --- later, we will see that it directly implies that the model curve has the area we expect.  However, we will need  \eqref{eqn:cmdotW} as well.  To prove \eqref{eqn:cmdotW}, it is equivalent by Lemma~\ref{lem:WcM} to show that the right hand sides of \eqref{eqn:cmdotW} and \eqref{eqn:WcM} are equal.  We can rewrite the right hand side of \eqref{eqn:WcM}:
\begin{align*}
\ell_{n-1}Q_n + t_n Q_{n+1} - 1 & = (\ell_{n-1} + t_{n+1}) Q_n + 1  \\
& = (\ell_{n+1} - \ell_{n}+\ell_{n-1}) Q_n + 1\\
& = 6 \ell_n Q_n + 1\\
& = 6 \ell_n^2 + 6 \ell_n\ell_{n-1} + 1,
\end{align*}
where in the first line we have used Lemma~\ref{lem:fibid0}, and in the last we have used \eqref{eq:Fibq}.  So, it is equivalent to show 
\[  (6 \ell_n^2 + 6 \ell_n\ell_{n-1}) - (\ell_n^2 + 41\ell_n\ell_{n-1} - 5 \ell_{n-1}^2 ) =  5,\]
or equivalently that
\begin{equation}
\label{eqn:fineq}
5( \ell_n^2 - 7\ell_n\ell_{n-1} +\ell_{n-1}^2) = 5.
\end{equation}
Since
\[ \ell_n^2 - 7\ell_n\ell_{n-1} + \ell_{n-1}^2  = \ell_{n-1}^2 -\ell_n\ell_{n-2}= 1,\]
where the last equality follows by \eqref{eq:recurrel}, equation \eqref{eqn:fineq} holds. \hfill$\er$
\end{remark}

The other identity that we will use gives a convenient way for studying sequences satisfying a certain recursion closely related to \eqref{eqn:basicrecursion}.

\begin{lemma}\label{lem:R}  Given positive integers $A, B>0$, define 
$$
R(A,B): = \Bigr(R_0^{\times 6}, R_1, R_2^{\times 5}, \dots, R_{2n-3}, R_{2n-2}^{\times 5}, R_{2n-1}\Bigr)
$$
 by the recursion $R_0 = A, \hspace{1 mm} R_1 = B,$
\begin{align}\label{eqn:R}
& R_{2k} = R_{2k-2}-R_{2k-1},\;\; R_{2k+1} = R_{2k-1} - 5 R_{2k},\; k<n.
\end{align}
Then we have:
\begin{itemize}
\item[{\rm (i)}]
$R(A,B) = A \cdot R(1,0) + B \cdot R(0,1)$.
\item[{\rm (ii)}]
$R(0,1) = \Bigl(0^{\times 6},1,(-1)^{\times 5},6,(-7)^{\times 5},\dots,(-\ell_{k-1})^{\times 5},t_{k},\dots, (- \ell_{n-2})^{\times 5}, t_{n-1}\Bigr).$
 \item[{\rm (iii)}]  Let $\widetilde W$ be the vector obtained from  $W(b_n)$ by deleting the last block of length $7$.  
 Then
 $$
 \widetilde W = \Bigl(Q_n^{\times 6}, P_n-6Q_n, (7Q_n-P_n)^{\times 5},\dots\Bigr) \equiv -Q_{n-1}\ R(0,1) \pmod {Q_n}.
 $$
\item[{\rm (iv)}]  If $\De = R(x_0,x_1) = (x_0^{\times 6}, x_1, x_2^{\times 5}, \dots, x_{2n-1})$ for some $n\ge 1$, then $$
\De\cdot R(0,1) = \ell_{n-1}x_{2n-1}.
$$
\end{itemize}
\end{lemma}

\begin{proof}  
The sequences $R(A,B)$ and $A \cdot R(1,0) + B \cdot R(0,1)$ both have the same initial conditions, and any linear combination of sequences satisfying \eqref{eqn:R} also satisfies this recursion.  This proves (i).  To prove (ii) one just has has to check that the recursion is satisfied, and this follows because $\ell_n$ and $t_n: = \ell_n - \ell_{n-1}$ both satisfy \eqref{eqn:basicrecursion}.  To prove (iii), notice first that  $ \widetilde W$ does satisfy the recursion \eqref{eqn:R} by part~(a)  of Corollary~\ref{cor:weight}. Further $P_n-6Q_n = Q_{n+1} - 6Q_n = Q_n - Q_{n-1} $ by \eqref{eqn:basicrecursion}.  Hence, (iii) follows from (i).  We  prove (iv) by induction on $n$.  It is clear when $n=1$, and the inductive step holds because by (ii) and \eqref{eqn:R}, we have
\begin{align*}
\ell_{n-1}x_{2n-1} - 5 \ell_{n-1} x_{2n} + t_{n} x_{2n+1}& = \ell_{n-1}(x_{2n-1} - 5x_{2n}) +  t_{n} x_{2n+1} 
\\ &
= \ell_{n-1}x_{2n+1} + t_{n} x_{2n+1} =  \ell_n x_{2n+1},
\end{align*}
as required. 
\end{proof}

\NI {\bf Best approximations:} 

Let $\theta$ be any irrational number. We will  need to use some facts about rational numbers $p/q$ that best approximate $\theta$ from below. 

 Recall that a rational number $p/q$ 
 in lowest terms is a {\bf best rational approximation} to $\theta$ if 
$ | \theta - p/q | < |\theta - m/n|$
for all $n < q$, while it is a
{\bf best rational approximation from below}
if 
\[ 0 < \theta - p/q< \theta - m/n,\quad \forall n<q, m < \theta n.
\] 
 To elaborate, let $\theta = a_0 + \frac 1{a_1 + \dots}$ have continued fraction expansion
\[ \theta = [a_0,a_1,\ldots].\]
The {\bf convergents} of $\theta$ are the rational numbers
\begin{align}\label{eq:convgt}
c_k & \eqdef p_k/q_k \eqdef [a_0,a_1,\ldots,a_k]. 
\end{align}
For any $k$, they satisfy
\[ c_0 < c_2 < \ldots < c_{2k} < \theta < c_{2k+1} < \ldots < c_3 < c_1.\]
Any convergent is a best approximation to $\theta$.  To get all possible best approximations, we must also consider the {\bf semiconvergents} of $\theta$.  A semiconvergent is a fraction of the form $c_{k-2} \oplus {\bf r} \cdot c_{k-1}$, where 
\begin{itemize}\item $0 < r < a_k$, 
\item the operation $\oplus$ is defined by the rule
\[ \frac{p}{q} \oplus \frac{p'}{q'} = \frac{p+p'}{q+q'},\]
and 
\item the multiplication by ${\bf r}$ denotes repeated addition with the $\oplus$ operation.
\end{itemize}
  For motivation, note that the convergents satisfy 
\begin{equation}
\label{eqn:fundamentalrelation}
c_k = c_{k-2} \oplus {\bf a_k} \cdot c_{k-1}.
\end{equation} 
If $k$ is even, then 
the fractions $c_{k-2} \oplus {\bf r} \cdot c_{k-1}$ increase with $r\ge 0$ and for $r\le a_k$ are all smaller than $\theta$, while 
if $k$ is odd these fractions are bigger than $\theta$.  Another useful fact is that the convergents satisfy
\begin{equation}
\label{eqn:anotherrelation}
q_np_{n-1} - q_{n-1}p_n = (-1)^n.
\end{equation}

The following well known fact
 will be very useful; for a proof see \cite{HW} or the proof of \cite[Lem.~3.3]{HT2}.
 
\begin{lemma}~\label{le:fact}  Suppose that $\theta>0$ is an irrational number.
Then the  rational numbers that best approximate $\theta$ from below are the even convergents $c_{2k}$, and the semiconvergents $c_{2k-2} \oplus {\bf r} \cdot c_{2k-1}$, with $k\ge 1$ and $1\le r< a_{2k}$.  
\end{lemma}

The following two examples are key.
 
\begin{example}
\label{ex:bigmono}
Let $\theta = \theta_n: = b_n + \eps$ for some very small irrational $\eps>0$ to be chosen later. 
 We will want to know those best approximations from below with denominator no more than $\ell_n$.  By Corollary~\ref{cor:weight} we have
  \[b_n = P_n/Q_n = [6; (1,5)^{\times (n-1)}, 1, 7]\]
We know that $Q_n > \ell_n$.  If $\epsilon > 0$ is sufficiently small, then $b_n$ is an even convergent of $\theta$, and the even convergents with denominator less than $Q_n$ have continued fraction expansion
\[ c_{2k} \eqdef [6; (1,5)^{\times k}] = \frac{t_{k+1}}{t_{k}},\qquad 0 \le k < n,
\]
 while the odd convergents with denominator less than $Q_n$ have continued fraction expansion
\begin{equation}\label{eqn:bestapp}
c_{2k+1} \eqdef [6; (1,5)^{\times k},1] = [6; (1,5)^{\times (k-1)}, 1, 6] = \frac{\ell_{k+1}}{\ell_{k}},\qquad 0 \le k < n.
\end{equation}
Further, there are $6$ semiconvergents of the form
\[ \frac{t_n + r\cdot \ell_n}{t_{n-1} + r\cdot \ell_{n-1}}, \quad 1 \le r \le 6\]
that are smaller than $\theta$, and for each $k < n$, there are $4$ semiconvergents of the form
\[ \frac{t_k + r\cdot \ell_k}{t_{k-1} + r \cdot \ell_{k-1}}, \quad 1 \le r \le 4\]      
that are smaller than $\theta$. 

The even convergents, and the semiconvergents mentioned above are all of the best possible approximations to $\theta$ from below with denominator no more than $Q_n$. \hfill$\er$
\end{example}

\begin{example}
\label{ex:smallmono}
Now let $\theta = \tilde{\theta}_n := \frac 1{\theta_n}$.  We want to know those best approximations from below with denominator no more than $\ell_n$.  We know
\[ 1/b_n = [0; 6, (1,5)^{n-1}, 1, 7].\]
If $\epsilon > 0$ is sufficiently small, then all of the even convergents to $\theta$ are of the form
\[c_{2k} \eqdef [0; 6, (1,5)^{\times (k-1)}, 1] = [0;6, (1,5)^{\times (k-2)}, 1, 6] = \frac{\ell_{k-1}}{\ell_k}\] 
for $1 \le k \le n$, except for the convergent $c_0 \eqdef [0] = \frac{\ell_{-1}}{\ell_0}.$
Thus, in this case the best  rational approximations from below all have denominators $\ell_k$ for $0 \le k \le n$.    \hfill$\er$           
\end{example}

\subsection{Basics of embedded contact homology}\label{sec:ECH}

Let $J$ be an almost complex structure on a completed symplectic cobordism $\overline{X}$.  We will assume throughout the paper that $J$ is {\bf admissible}.  This means that on any symplectization end 
$\bigl(Y\times I, d(e^s\la)\bigr)$
of $\overline{X}$ (where $I = (-\infty, -N)$ or $(N,\infty)$ and $s$ denotes the coordinate on $\mathbb{R}$),  $J$ is translation invariant, rotates the contact structure $\ker(\la)$ 
positively with respect to $d\lambda$, and sends $\partial_s$ to the Reeb vector field $R$.  We will want to consider $J$-holomorphic curves with disconnected domain. 
So in the following
 we will call  a curve with  connected domain  {\bf irreducible}, and call it {\bf reducible} otherwise.  
 Further a curve is called {\bf somewhere injective} if each of its irreducible components is somewhere injective and no two have the same image.  All of the curves throughout the paper will have punctured domain, and are asymptotic to closed Reeb orbits near the punctures, see for example \cite[\S 3.1]{hlecture}.  We consider curves up to the usual equivalence relation, namely reparametrization of the domain.

\MS

\NI
{\bf Relative intersection theory:}  Consider two 
 distinct\footnote{This means in particular that $C$ and $C'$ have no irreducible components in common.} somewhere injective, $J$-holomor\-phic curves  $C, C'$ in a four-dimensional completed symplectic cobordism $\overline{X}$. In our proof, we will frequently want to compute
\[ C \cdot C.' \]
This is an algebraic count of intersection points of $C$ with $C'$.  By positivity of intersections, each point counts positively.  

Because $\overline{X}$ is noncompact, the quantity $C \cdot C'$ is not purely homological.  Rather, we have
\begin{equation}
\label{eqn:interstheq}
C \cdot C' = Q_{\tau}([C],[C']) + L_{\tau}(C,C').      
\end{equation}
Here, $\tau$ denotes a trivialization of $\xi = \Ker(\lambda)$ over all embedded Reeb orbits, and $[C]$ denotes the {\bf relative homology class} of $C$.  This is defined regardless of whether or not $C$ is somewhere injective, and takes values in $H_2(\overline{X},\alpha,\beta)$, where $\alpha$ and $\beta$ are {\bf orbit sets}, namely finite sets $\lbrace (\gamma_i,m_i) \rbrace$, where the $\gamma_i$ are embedded Reeb orbits, and the $m_i$ are positive integers. The orbit set $\alpha$ is given by the positive asymptotics of $C$.   We say that $C$ is {\bf asymptotic} to an orbit set $\Theta = \lbrace (\gamma_i,m_i) \rbrace$ at $+\infty$ if, for each $i$, the sum of the multiplicities of the positive ends of $C$ at $\gamma_i$ is exactly $m_i$, and $C$ has no positive ends at any other orbit other than the $\gamma_i$.   The orbit set $\beta$ is given by the negative asymptotics.     The fact that $Q_{\tau}([C],[C'])$ is homological is proved in \cite{hech}.  

Equations \eqref{eqn:cp2calc} and \eqref{eqn:blowupcalc} show
 how to compute it in the situations relevant to us. 
 
If $C$ partitions $m_i$ as $m_{i1},\dots, m_{in_i}$, then we also denote the orbit set as
\begin{equation}\label{eqn:orbset}
\lbrace (\gamma_i,m_i) \rbrace = \{(\ga_i^{m_{i1}},\dots, \ga_i^{m_{i n_i}})\},\qquad m_i = \sum_{j=1}^{n_i} m_{ij},
\end{equation}
where $\ga_i^r$ denotes a single end on $\ga_i$ of multiplicity $r$.

For future use, we define
\[ \mathcal{M}(\overline{X},J, \alpha,\beta)= \mathcal{M}(\alpha,\beta)\]
to be the moduli space of $J$-holomorphic curves in $\overline{X}$ that are asymptotic to the orbit set $\alpha$ at $+\infty$ and asymptotic to the orbit set $\beta$ at $-\infty$.

The term $L_{\tau}(C,C')$ is the {\bf asymptotic linking number} of $C$ and $C'$.  To define it, first fix an embedded orbit $\gamma_i$ at which both $C$ and $C'$ have positive ends.  By intersecting $C$ with an $s=R$ slice in the positive end of $\overline{X}$ for sufficiently large $R$, the positive ends of $C$ at $\gamma_i$ form a link $\zeta^+_{i,C}$, which we can regard as a link in $\mathbb{R}^3$ via the trivialization $\tau$ as in \cite[\S 3.3]{hlecture}.  We can define a link $\zeta^+_{i,C'}$ similarly, and we can define the linking number $L_{\tau}(\zeta^+_{i,C}, \zeta^+_{i,C'})$ of these two links to be their linking number in $\mathbb{R}^3$, using the identification $\tau$.  If $R$ is sufficiently large, then this number does not depend on the choice of $R$.  We can define links $\zeta^-_{i,C}, \zeta^-_{i,C'}$ and linking numbers for orbits at which $C$ and $C'$ both have negative ends analogously.  

We now define 
\begin{equation}
\label{eq:linking}
 L_{\tau}(C,C') = \sum_{i=1}^{n} L_{\tau}(\zeta_{i,C}^+,\zeta_{i,C'}^+) - \sum_{j=1}^{m} L(\zeta_{j,C}^-,\zeta_{j,C'}^-)
 \end{equation}
where the first sum is over the embedded orbits at which both $C$ and $C'$ have positive ends, and the second sum is over the embedded orbits at which both $C$ and $C'$ have negative ends.

\MS

\NI {\bf The ECH index and the partition conditions:}
Let $C \in \mathcal{M}(\alpha,\beta)$ be a somewhere injective curve in $\overline{X}$.   Part of our proof will involve estimating the ECH index of such a curve.  We now review what we need to know about the ECH index.

Recall first the {\bf Fredholm index} for curves in $4$-dimensions\footnote
{
See Lemma~\ref{lem:largeS} for higher dimensions.} takes the form
\begin{equation}
\label{eqn:freddefn}
\ind(C) = -\chi(C) + 2c_{\tau}(C)+CZ_{\tau}^{ind}(C).
\end{equation}
Here, $c_{\tau}(C)$ denotes the {\bf relative first Chern class} of $C$ (see \cite[\S 3]{hlecture}), and $CZ_{\tau}^{ind}(C)$ denotes the Conley--Zehnder index
\begin{equation}\label{eq:Frind1}
 CZ^{ind}_{\tau}(C) = \sum_i CZ_{\tau}(\gamma_i) - \sum_j CZ_{\tau}(\gamma_j),
 \end{equation}
where the first sum is over the (possibly multiply covered) orbits given by the positive ends of $C$, the second sum is over the (possibly multiply covered) orbits given by the negative ends of $C$, and $CZ_{\tau}$ of a Reeb orbit $\gamma$ denotes its Conley--Zehnder index: see \eqref{eq:Frind2} below for the elliptic case.  

If $C$ is somewhere injective, we can bound $\ind(C)$ from above by the {\bf ECH index} of $C$.  The ECH index depends only on the relative homology class of $C$, and is defined for mutiply covered curves as well by the formula 
\begin{equation}
\label{eqn:echdefn}
I([C]) = c_{\tau}([C]) + Q_{\tau}([C]) + CZ_{\tau}^I([C]),
\end{equation}
where $Q_{\tau}$ denotes the relative intersection pairing from \eqref{eqn:interstheq}, and $CZ^I_{\tau}$ is the {\bf total Conley-Zehnder index}
\[CZ^I_{\tau}([C]) = \sum_i \sum^{m_i}_{k=1} CZ_{\tau}(\alpha_i^{k}) - \sum_j \sum^{n_j}_{k=1} CZ_{\tau}(\beta_j^k),\]
where $\alpha = \lbrace (\alpha_i,m_i) \rbrace,$ $\beta = \lbrace (\beta_j,n_j) \rbrace$, and $\gamma^x$ denotes the $x$-fold cover of $\gamma$. The precise statement of this bound is the {\bf index inequality}        
\begin{equation}
\label{eqn:indexinequality}
\ind(C) \le I([C]) -2 \delta(C)
\end{equation}
for somewhere injective curves, proved in \cite{hech}; see also Proposition~\ref{prop:indineq} below.   Here, $\delta(C) \ge 0$ is an algebraic count of the singularities of $C$.

When equality holds in \eqref{eqn:indexinequality}, for example if $\ind(C)=I(C)$, then 
we can say much more about the asymptotics of $C$. 
Indeed, if such a curve $C$ has ends at an embedded orbit $\alpha_i$ with total multiplicity $m_i$,  then the multiplicities of the ends of $C$ at $\alpha_i$ give a partition of $m_i$ that 
is called the {\bf  ECH partition}.    This partition depends only on whether $\alpha_i$ is at the positive or negative end of $C$, and can be computed purely combinatorially, as is shown in \cite{hech} and reviewed in the proof of Proposition~\ref{prop:indineq} below.  For positive and negative ends, it is denoted respectively as
\[ p^+_{\alpha_i}(m_i), \qquad  p^-_{\alpha_i}(m_i).\]
The next remark explains what we will need.

\begin{remark}\label{rmk:p+}{\bf (Computation of  $p^{\pm}_{\alpha}(m)$ for elliptic ends)} \rm

\NI (i)  Consider a positive end along an elliptic orbit $\alpha$ with 
  mod $1$ monodromy angle of $\theta\in (0,1)$. 
  Let $\Lambda$ be the maximal concave piecewise linear path in the first quadrant that starts at $(0,0)$, ends at $(m, \lfloor m \theta \rfloor)$, has vertices at lattice points, and stays below the line $y = \theta x$.  It is shown in \cite{hech} that $p^+_{\alpha}(m)$ is given by the horizontal displacements of this path.
Here the word \lq\lq maximal" includes the assumption that the edges of $\La$ have no interior lattice points, in other words that for each segment of the path the horizontal and vertical displacements are  mutually prime.  Thus instead of a single segment labelled by $(4,6)$, for example, we have two segments each with labels $(2,3)$. 
\MS

\NI (ii) 
For a negative end the procedure is analogous, except that $\La$ is now the minimal convex lattice path that lies above the line $y = \theta x$.
For example, if $
\al$ is the long orbit of $E(1,x)$ where $x = \frac PQ + \eps$, then the monodromy angle of $\al$ is $x$,
so $p_\al^-(Q) = (Q)$ only if $\frac{P+1}Q$ is the best approximation to $x$ from above.
In the case $\frac PQ = b_n$, it follows from \eqref{eqn:bestapp}  that this best approximation is $\frac{\ell_{n}}{\ell_{n-1}}$. 
 Since $7\ell_{n-1}< Q = Q_n$, the path $\La$ starts with $7$ segments along the line of slope $\frac{\ell_{n}}{\ell_{n-1}}$; 
 see Lemma~\ref{lem:partcond} below.
 Thus  in this case the partition $p_\al^-(Q)$ must have at least eight terms.  
  \hfill$\er$
\end{remark}

Since $p^+_{\alpha}(m)$ only depends on the mod $1$ monodromy angle $\theta$
of $\alpha$, we will sometimes write $p^+_{\theta}(m)$ instead.  A useful fact about the positive partition is if $p^+_{\theta}(m)=(a_1,\ldots,a_s)$, then
\begin{equation}
\label{eqn:partitionequation}
\lfloor (a_i + a_j) \theta \rfloor = \lfloor a_i \theta \rfloor + \lfloor a_j \theta \rfloor
\end{equation}
for any $1 \le i \ne j \le s$, see \cite[Ex. 3.13.]{hlecture}. 
Similarly, if $p^-_{\theta}(m)=(b_1,\ldots,b_s)$ then 
$\lceil (\sum_i b_i) \theta \rceil =\sum_i \lceil b_i \theta  \rceil$.  For example, if $\theta \equiv \frac PQ + \eps \pmod 1$
and $p^-_\theta(Q) = (b_1,\ldots,b_s)$ then 
\begin{equation}
\label{eqn:CZQ}
\sum_i \lfloor b_i \theta \rfloor = \sum_i (\lceil b_i \theta  \rceil -1) = P +1 -s.
\end{equation}

\MS

\NI
{\bf The relative adjunction formula:}
The index inequality \eqref{eqn:indexinequality} is related to an adjunction formula that we will also need.  Namely, recall the {\bf relative adjunction formula} from \cite{hlecture}.  This says that if $C$ is somewhere injective then
\begin{equation}
\label{eqn:adjunction}
c_{\tau}([C]) = \chi(C) + Q_{\tau}([C]) +w_{\tau}([C]) - 2 \delta(C).
\end{equation}
The term here that has not already been introduced, $w_{\tau}(C),$ is called the {\bf asymptotic writhe} of $C$.  Its definition is similar to the definition of the asymptotic linking number in \eqref{eq:linking}.  Namely, fix an embedded orbit $\gamma_i$ at which $C$ has positive ends, and 
regard  the links $\zeta^+_{i,C}$ and $\zeta^-_{i,C}$ as links in $\mathbb{R}^3$ via the trivialization $\tau$.  Let $w_{\tau}(\zeta^+_{i,C})$ and $w_{\tau}(\zeta^-_{i,C})$ denote the writhes of these links.  If $R$ is sufficiently large, then this does not depend on the precise choice of $R$. We can define writhes associated to negative ends analogously.

We can now define
\[w_{\tau}(C) \eqdef \sum_{i=1}^n w_{\tau}(\zeta^+_{i,C}) -  \sum_{j=1}^m w_{\tau}(\zeta^-_{j,C}),\] 
where the first sum is over the orbits at which $C$ has positive ends, and the second sum is over the orbits at which $C$ has negative ends.    
        \MS
        
\NI {\bf An improved index inequality:}
There is a refined version of \eqref{eqn:indexinequality} that will be relevant to what follows.  To state it, suppose first that  
$\gamma$ is an elliptic orbit at which $C$ has 
positive
ends of total multiplicity $m>0$, and let $\theta$ be the mod $1$ monodromy angle of $\gamma$, normalized to be in $(0,1)$. 
 The ends of $C$ give a partition $(a_1,\ldots,a_n)$ of $m$.  Order the numbers $a_1, \ldots, a_n$ so that
\begin{equation}\label{eq:ai}
\theta>  \frac{\lfloor a_1 \theta \rfloor} {a_1} \ge \ldots \ge \frac {\lfloor a_n \theta \rfloor} {a_n},
\end{equation}
and let $\Lambda_C$ be the concave lattice path in the first quadrant  that starts at $(0,0)$, ends at $(m, \sum_{i=1}^n \lfloor a_i \theta \rfloor),$ and has edge vectors $(a_i, \lfloor a_i \theta \rfloor)$, appearing in the same order as the $a_i$. 
Define
\begin{equation}\label{eq:Adef}  A_C(\gamma,m) = \mathcal{L}(\Lambda_C) + \tfrac 12 b(\Lambda_C),
\end{equation}
where:
\begin{itemize}
\item $\mathcal{L}(\Lambda_C)$ is  
the number of lattice points in the region bounded by the line $y = \theta x$ and the vertical line from $(m, \sum_{i=1}^n \lfloor a_i \theta \rfloor)$ to $(m, m \theta)$
that lie strictly above the path $\Lambda_C$;
\item $b(\La_C)$ is the sum over all edges of $\Lambda_C$ of the number of {\bf interior} lattice points in each edge.
\end{itemize}

Now define
\begin{equation}\label{eq:Adef1} 
A(C) = \sum_{(\gamma,m)} A_C(\gamma,m),
\end{equation}
where the sum is over pairs $(\gamma,m)$ for which $\gamma$ is elliptic and $C$ has at least one positive end. 
There is a similar definition for negative ends, that we do not give since we do not need it.
 Note that if $C$ has ECH partitions, then $A(C) = 0$.  This holds by the maximality condition in Remark~\ref{rmk:p+}~(i) together with 
\eqref{eqn:partitionequation}.
  
 Here is the 
refined index inequality. 
\begin{proposition}
\label{prop:indineq}  Let $C$ be a somewhere injective curve.  Then
\begin{align}\label{eqn:indineq}
I(C) - \op{ind}(C) \ge 2 \delta(C) + 2A(C).
\end{align}   
\end{proposition}

\begin{proof}
This is 
 implicit in the work of Hutchings, but for completeness we give the proof.  In this proof we will assume  that $C$  has no negative ends, since that is the case we need.

By combining the definition of the ECH index \eqref{eqn:echdefn}, the definition of the Fredholm index \eqref{eqn:freddefn}, and the relative adjunction formula \eqref{eqn:adjunction}, we get
\begin{equation}
\label{eqn:withadjunction}
I(C) - \op{ind}(C) = CZ_{\tau}^I(C) - CZ_{\tau}^{ind}(C) - w_{\tau}(C) + 2 \delta(C).
\end{equation}
The terms $w_{\tau}(C), CZ_{\tau}^I(C)$ and $CZ_{\tau}^{ind}(C)$ are all sums over terms corresponding to each orbit at which $C$ has ends.  

So, let $(\gamma,m)$ be a pair corresponding to an orbit at which $C$ has positive ends, assume that $\gamma$ is elliptic, and let $\zeta$ be the braid coming from the ends of $C$ at $\gamma$.  By \cite[Eq. 5.4]{hlecture} and \cite[Lem. 5.5]{hlecture}, we have\footnote
{
There is a similar lower estimate for the writhe of a negative end; see Remark~\ref{rmk:writhexact}~(ii).}
\begin{equation}
\label{eqn:writhebound2} 
w_{\tau}(\zeta) \le  \sum_{i,j = 1}^n \op{max}(p_ia_j,p_ja_i) - \sum_{i=1}^n p_i,
\end{equation}
where $p_i =  \lfloor a_i \theta \rfloor$  in the notation of \eqref{eq:ai}.
We also know that
\begin{equation}
\label{eqn:czdefn}
CZ_{\tau}^I( (\gamma,m) ) - CZ_{\tau}^{ind}( (\gamma,m)) = \sum^n_{i=1} (2 \lfloor i \theta \rfloor + 1) - \sum_{i=1}^n (2 p_i + 1),
\end{equation}
where $CZ_{\tau}^I( (\gamma,m))$ denotes the contribution of the pair $(\gamma,m)$ to $CZ_{\tau}^I$, and similarly for $CZ_{\tau}^{ind}( (\gamma,m)).$     
 
 Consider
\[ 2A \eqdef \sum_{i,j = 1}^n \op{max}( p_i a_j, p_j a_i).\]
This is twice the area of the region $P$ bounded by the path $\Lambda_C$ defined above, the vertical line from $(m,0)$ to 
$(m, \sum_{i=1}^n p_i)$, and the $x$-axis.  Pick's theorem gives
\begin{equation}
\label{eqn:picks}
2 A = 2T - B - 2,
\end{equation}
where $T$ is the total number of lattice points in $P$, and $B$ is the number of boundary lattice points.  We have
\begin{equation}
\label{eqn:Tequation}
T = m + 1 + \sum_{i=1}^n \lfloor i \theta \rfloor - \mathcal{L}(\Lambda_C),
\end{equation}
and
\begin{equation}
\label{eqn:Bequation}
B = m + n + \sum_{i=1}^n p_i + b(\Lambda_C).
\end{equation}

Combining \eqref{eqn:writhebound2} with \eqref{eqn:picks}, \eqref{eqn:Tequation}, and \eqref{eqn:Bequation} gives 
\[ 
w_{\tau}( (\gamma,m) ) \le m - k + 2\sum_{i=1}^n \lfloor i \theta \rfloor - 2 \mathcal{L}(\Lambda_C) - b(\Lambda_C) - 2 \sum_{i=1}^n p_i.
\]
Combining this inequality with \eqref{eqn:czdefn} gives
\[ w_{\tau}((\gamma,m)) \le CZ_{\tau}^I( (\gamma,m)) - CZ_{\tau}^{ind}((\gamma,m)) - 2 \mathcal{L}(\Lambda_C) - b(\Lambda_C).\]
Now  sum this final equation over all elliptic orbit sets, use the bound \cite[Lem. 5.1]{hlecture} for the hyperbolic orbit sets, and combine the resulting equation with \eqref{eqn:withadjunction}.
\end{proof}      

\begin{rmk}\label{rmk:writhexact}\rm (i)  If $\ga$ is a negative end of $C$ of total multiplicity $m$, then the analog of \eqref{eqn:writhebound2} is the lower bound
\begin{equation}
\label{eqn:writheb3} 
w_{\tau}^-(\zeta) \ge  \sum_{i,j = 1}^n \op{min}(p_ia_j,p_ja_i) - \sum_{i=1}^n p_i,
\end{equation}
where  $C$ has ends of multiplicities $(a_1,\dots,a_n)$ on $\ga$ and $p_i =  \lceil a_i \theta \rceil$,
see \cite[\S 5]{hlecture}.  
\MS

\NI (ii)  The inequality for the writhe given in \eqref{eqn:writhebound2}  is proved by considering the asymptotic behavior of $C$ near the limiting orbit $\ga$.
As pointed out to us by Hutchings, if $C$ has only one positive end on $\ga$ with multiplicity $m$ and if $p^+_{\theta}(m) = (m)$ so that $C$ has the ECH partition at this end, then this estimate is in fact an equality.  To see this, note that
by \cite[Lemma~6.4]{HUT01} this is equivalent to claiming that  the asymptotic expansion  of the trajectory has a term corresponding to the smallest possible elgenvalue. But this holds by the argument outlined in \cite[Remark~3.3]{HT}.  
   A similar statement holds for negative ends.

 Note also that  exactness of the writhe  bounds at both ends of a curve implies equality in \eqref{eqn:indineq}.  
 Therefore, if a curve has ECH partitions (so that $A(C) = 0$), exact writhe bounds, and also has $I(C) -\ind(C) = 2$,
then it must have a double point.
\MS

\NI (iii) The proof of Proposition~\ref{prop:indineq} above shows that if $C$ is a simple curve such that $I(C) = \op{ind}(C) - 2\de(C)$, then the path $\Lambda_C$ must be maximal, and so $C$ must have the ECH partitions.  This proves \eqref{eqn:indexinequality}.  
 \hfill$\er$
\end{rmk}

\NI {\bf ECH index computations:}
For our purposes, $\overline{X}$ will always be the completion of a symplectic cobordism $X$ 
with two boundary components $\p ^{\pm}{X}$ that are either empty or are ellipses $Y=\partial E(a,b)$; later we will refer to the ends that we add to complete $X$ as {\bf symplectization-like ends}.    We now review the relevant formulas for $c_{\tau}, Q_{\tau},$ and $CZ$ for these $\overline{X}$.  Recall that if $b/a$ is irrational, then the Reeb vector field for $\partial E(a,b)$ has exactly two embedded orbits, $\gamma_1 = \lbrace z_2 = 0 \rbrace$ and $\gamma_2 = \lbrace z_1 = 0 \rbrace$.  They are both elliptic.  It is convenient to keep track of their {\bf action} 
defined by 
\begin{align}\label{eqn:action}
\mathcal{A}(\gamma) &= \int_{\gamma} \lambda.
\end{align}
We have $\mathcal{A}(\gamma_1) = a$ and $\mathcal{A}(\gamma_2)=b$.    
 
First assume that $\overline{X} = Y \times \R$.  Then, as explained in eg \cite[\S 3.7]{hlecture}, for any curve $C$, both $c_{\tau}([C])$ and $Q_{\tau}([C])$ depend only on the asymptotics of $C$.  Assume, then, that 
\[C \in \mathcal{M}(Y \times \R,\alpha,\beta),\]
where $\alpha = \lbrace (\gamma_1, m_1), (\gamma_2, m_2) \rbrace$ and $\beta = \lbrace (\gamma_1, n_1), (\gamma_2,n_2) \rbrace$.  

Then we define the {\bf action}  of $C$ to be
\begin{align}\label{eqn:action2}
\mathcal{A}(C) &= \Aa(\al) - \Aa(\be) = \int_{\al} \lambda - \int_\be \la.
\end{align}
It is  shown in \cite[\S 3.7]{hlecture} that we can choose the trivialization $\tau$ so that:
\begin{itemize}
\item The monodromy angle of $\gamma_1$ is $a/b$ and the monodromy angle of $\gamma_2$ is $b/a$.  
\item $c_{\tau}(C) = (m_1 + m_2) - (n_1 + n_2).$
\item $Q_{\tau}(C) = 2(m_1m_2-n_1n_2)$. 
\end{itemize}    
To compute the relevant $CZ$ terms, recall that if $\gamma$ is an elliptic orbit, with monodromy angle $\theta$ with respect to $\tau$, then 
\begin{equation}\label{eq:Frind2} CZ_{\tau}(\gamma) = 2 \lfloor \theta \rfloor + 1.
 \end{equation}

Since in the current situation  the terms in the ECH index for $C$ only depend on the asymptotics of $C$, it is convenient to define a {\bf grading} 
\begin{equation}\label{eq:gr10}
\op{gr}\left(\lbrace (\gamma_1,m_1), (\gamma_2,m_2) \rbrace\right) = m_1 + m_2 + 2m_1m_2 + \sum_{i=1}^{m_1} CZ_{\tau}(\gamma_1^i) + \sum_{i=1}^{m_2} CZ_{\tau}(\gamma_2^i)
\end{equation} 
associated to any orbit set 
on the ellipsoid $\p E(a,b)$, 
so that if $C \in 
\mathcal{M}(Y\times \R, \alpha,\beta)$ then
\begin{equation}\label{eq:gr00}
I(C) = \op{gr}(\alpha) - \op{gr}(\beta).
\end{equation} 
One can check that when $b/a$ is irrational,
\begin{equation}\label{eq:gr0}
\op{gr}\left(\lbrace (\gamma_1,m_1), (\gamma_2,m_2) \rbrace\right)  = 2(\#N(m_1,m_2) - 1),
\end{equation}
where $N(m_1,m_2)$ is the number of integral points in the first quadrant triangle with slant edge  $x + \frac ba y = m_1 + \frac ba m_2$.  
 
There are two other $4$-dimensional cobordisms $\overline{X}$ for which we will want to understand these calculations.

The first comes from removing the interior of an irrational ellipsoid $E(1,x)$ from $\mathbb{C}P^2(\mu)$ and completing 
at the negative end, where the symplectic form on $\mathbb{C}P^2(\mu)$ is the Fubini--Study form scaled so that the line has size $\mu$.  In this case, any relative homology class is determined by its coefficient along the {\bf line class} $L$ and its negative asymptotics in $H_2(\overline{X},\emptyset,\beta).$   Continuing with the trivialization from above, if $C \in \mathcal{M}(\overline{X},\emptyset,\beta),$ $[C] = dL$, and $\beta = \lbrace (\gamma_1, n_1), (\gamma_2,n_2) \rbrace$ then
\begin{equation}
\label{eqn:cp2calc}
c_{\tau}([C]) = 3d - (n_1 + n_2), \quad Q_{\tau}([C]) = d^2 - 2n_1n_2.     
\end{equation}               
Further,  if $n_1$ is partitioned as $(a_1,\dots,a_r) $ while $n_2$ is partitioned as $(b_1,\dots,b_s) $,  then equations  \eqref{eqn:freddefn}, 
 \eqref{eq:Frind1}, and  \eqref{eq:Frind2} imply that if $C$ has $k$ connected components then
\begin{equation} \label{eqn:Xind}
\tfrac12 \ind (C) =  -k + 3d -   \sum_i \bigl(a_i + \lfloor \frac{a_i}x\rfloor \bigr)   -  \sum_j \bigl(b_j+   \lfloor b_jx\rfloor\bigr).
\end{equation}               
 Finally we define the {\bf action } (or {\bf $\om$-energy}) of $C$ to be 
\begin{align}\label{eqn:action3}
\mathcal{A}(C) &= d\,\mu - \Aa(\be) = d\, \mu -  n_1 - n_2x. 
\end{align}

The second cobordism 
 comes from performing a sequence of blowups in the interior of an irrational ellipsoid and then completing 
at the positive end; call it $\widehat{\Ee}$.
  Let $E_1, \ldots, E_k$ denote the {\bf exceptional classes} associated to these blowups.  In this case, any relative homology class is determined by its coefficients along these classes.  Using the same trivialization as above, if $C \in \mathcal{M}(\widehat{\Ee},\alpha,\emptyset)$, 
  $[C] = - (m_1 E_1 + \ldots m_k E_k)$, and $\alpha = \lbrace (\gamma_1, n_1), (\gamma_2,n_2) \rbrace$, then
\begin{equation}
\label{eqn:blowupcalc}
c_{\tau}([C]) =  (n_1+n_2) - (m_1 + \ldots + m_k), \quad Q_{\tau}([C]) = 2n_1n_2 - (dm_1^2 + \ldots m_k^2), 
\end{equation}      
so that by \eqref{eqn:echdefn}   we have
\begin{equation} \label{eqn:Eei}
 I (C) =  \gr(\al) -  \sum_i (m_i + m_i^2). 
\end{equation}  
Further, if $C$ is connected and its positive end is partitioned as above then   \eqref{eqn:freddefn}       implies that
\begin{equation} \label{eqn:indEE}
 \tfrac 12 \ind(C)  = -1 + r + s + \sum_i \bigl(a_i + \lfloor \frac{a_i}x\rfloor \bigr)   +  \sum_j \bigl(b_j+   \lfloor b_jx\rfloor\bigr) - \sum m_i.
\end{equation}    
Notice that in this formula the $+1$ terms in the index formula \eqref{eq:Frind2}  are cancelled by the contribution of each end to 
the Euler characteristic, while the multiplicities $n_1,n_2$ in $c_{\tau}([C])$ have been rewritten as sums 
$ \sum_i a_i ,  \sum_j b_j$.
Finally, if  the symplectic area of $E_i$   
  is $w_i$, we define the {\bf action} of $C$ to be
\begin{align}\label{eqn:action4}
\mathcal{A}(C) &= \Aa(\al) - \sum m_iw_i  = n_1 + n_2 x - \sum m_iw_i.
\end{align}

Note that in all these situations the action is nonnegative.  This is clear in the case of a symplectization  since the condition that
 $J$ is admissible implies that $d\lambda$ is pointwise nonnegative on $C$,   with equality at $p \in C$ if and only if the tangent space to $C$ at $p$ is the span of the Reeb vector field and $\partial_s$.    A similar argument works for any exact cobordism.  It remains to note that the cobordisms in the second two examples can be made exact by removing the line from $\mathbb{C}P^2(\mu)$ and the exceptional divisors from
$\widehat{\Ee}$, in which case the contributions $d\, \mu$ and $\sum d_iw_i$ to $\Aa(C)$ can be interpreted as actions of the corresponding Reeb orbits.

In \S\ref{sec:stablzn}  we will also use the Fredholm index formula in higher dimensions.  If the dimension is $2N$ the analog of \eqref{eqn:freddefn} is
\begin{equation}
\label{eqn:freddefnN}
\ind(C) = (N-3)\chi(C) + 2c_{\tau}(C)+CZ_{\tau}^{ind}(C),
\end{equation}
 where $CZ_{\tau}^{ind}(C)$ is now calculated as follows.  As before $CZ_{\tau}^{ind}(C)$  is a sum of contributions from each end of $C$, where  a single end of multiplicity $a_k$ on the $k^{th}$ orbit $\ga_k$ of a generic\footnote
 {
 i.e. the ratios  $b_i/b_j, i\ne j,$ are irrational}
  ellipsoid $E(b_1,\dots,b_N)$ contributes a sum of $N-1$ terms of the form $2 \lfloor \theta \rfloor + 1$ (see \eqref{eq:Frind2}) with monodromy angles $\theta =  \frac{b_i}{b_k}$ for $1\le i\le N, i\ne k$.   
If some ratios $b_i/b_j, i\ne j,$ are rational,
then we are in a Morse-Bott situation and the index depends on whether the end is positive or negative.
In particular, if  $E(1,x,S,\dots,S)\subset \C^{2 + k}$ where $1<x<S$ so that $N = 2+k$, then the contribution to the index of a
 negative end of multiplicity $a$ on an orbit of action $S$ is the same 
 as that of the third orbit $\ga_3$ on the ellipsoid $E(1,x,S_1,\dots,S_k)$ where $S_1<S_2<\dots<S_k$ are slight perturbations of $S$.  Thus the monodromy angles are
 $ \frac{1}{S_1}, \frac{x}{S_1}, \frac{S_2}{S_1},\dots,\frac{S_k}{S_1}$, where all but the first two terms are slightly $>1$.

\section{Stabilizable curves}\label{sec:stab}

We now explain how to establish the existence of a curve $C$ satisfying the conditions in Proposition~\ref{prop:goodC} for suitable values of $d,p$. The argument is slightly different for the two cases $n=0$ (with $b_0 = 8$) and $n>0$.

\subsection{The setup}  \label{sec:stablesetup}
Recall from [MS] that for $n > 0$  there is an embedding
\begin{equation}\label{eq:mu0}
\Phi:  E(1,b_{n}+\epsilon)\; \se\; {\rm int} (B^4(\mu_n+\epsilon')), \qquad
\mu_n: = \frac {1 + b_n}3 = \frac{h_{2n+2}}{h_{2n+1}},\quad n>0,
\end{equation}
where  the numbers $\epsilon, \epsilon'>0$ are very small and irrational.\footnote
{
See Remark~\ref{rmk:eps} for a more precise description.  For now, we work with these small perturbations by simplifying via approximate identities such as
$\mu_n+\epsilon'\approx \mu_n$.}

 Let $\Ee: = \Ee_n$ be the image of this embedding, and write
 $\beta_1$ for the short orbit on its boundary $\p \Ee_n$ and $\beta_2$ for the long one.  We complete ${\rm int} (B^4(\mu_n+\epsilon'))$ to
 $\C P^2(\mu_n + \eps')$ and then
  define
  \begin{equation}\label{eq:X}
  \ov X: = \mbox{negative completion of }\bigl(\C P^2(\mu_n + \eps') \less \Ee_n\bigr).
 \end{equation}
When $n=0$, we make similar definitions with $\mu_0 = \frac {17}6 = c_0(8)$.

This section and the next explain the proof of the following result.

\begin{proposition}
\label{prop:goodcurves}
For $n\ge 0$ and generic admissible $J$ on $\overline{X}$, the moduli space $\mathcal{M}(\overline{X},h_{2n+2}L,1,\beta_1^{h_{2n+3}})$, of genus zero curves $C$ in class $h_{2n+2}L$ with  a single negative end on $\be_1^{h_{2n+3}}$, is nonempty.
\end{proposition}

This immediately implies our main result.

\begin{corollary}  $c_k(b_n) = \frac{3b_n}{b_n+1}$ for all $n\ge 0$.
\end{corollary}
\begin{proof}   Since $3h_{2n+2} = h_{2n+1}  + h_{2n+3} $ by \eqref{eq:Fib1}, it follows from \eqref{eq:indC} that  the curves $C$ in Proposition~\ref{prop:goodcurves}
have index $0$.  Hence we may apply
 Proposition~\ref{prop:goodC}, which gives  $c_k(b_n) \ge \frac{h_{2n+3}}{h_{2n+2}} = \frac{3b_n}{b_n+1}$.  The result follows from this together with the folding bound \eqref{eqn:foldingbound}.
\end{proof}

To prove Proposition~\ref{prop:goodcurves}, we first blow up $\C P^2$ in the
interior of the ellipsoid $\Ee_n$,
denoting the blown-up manifold by $\widehat{\C P}\!\,^2$.
For each $n$ we consider a class $B$ (described below) that is represented in $\widehat{\C P}\!\,^2$ by a finite number of genus zero curves with one double point, and then consider what happens to these representatives when we stretch the neck along the boundary of the ellipsoid $\p \Ee_n$. When we do this, we get a sequence of curves that converge in a suitable sense to a limiting building, with top level in $\ov X$, bottom level in $\widehat{\Ee}: =
 \widehat {\Ee}_n$, the positive completion  of the blown up ellipsoid, and perhaps also some intermediate levels in the symplectization $\p\Ee_n\times \R$ (usually called the \lq \lq neck").    Our aim is to show that at least one of the resulting  top level curves lies in
$\mathcal{M}(\overline{X},h_{2n+2}L,1, \beta_1^{h_{2n+3}})$.  We will assume that the reader
is familiar with this stretching process; for details see  for example \cite[\S2.3]{HK}.  For convenience, we will sometimes say that a building obtained in this way
is a {\bf breaking} of the $B$-curve.

Here are more details.  For each $n$, consider the
 weight sequence
 $$
 w(b_{n}) \eqdef (w_1,\ldots,w_m)
 $$
  defined in \eqref{eq:weight}; it satisfies  $\sum w_i^2 = b_n$.
  Also
  recall the normalized weight
   sequence
$W(b_{n})=(W_1,\ldots, W_m)
\eqdef h_{2n+1} w(b_{n}).$
By Proposition~\ref{prop:ell}, it is possible to  remove almost all of the interior of the ellipsoid $\Ee_n = \Phi(E(1,b_n+\eps))$ by a sequence of blowups of weights almost equal\footnote
{
The actual weights of the blowup are $(1-\eps'')w_i$ where $\eps''>0$ is very small: see Remark~\ref{rmk:eps}.}
 to $w_1,\dots, w_m$, to obtain a manifold $\widehat{\C P}\!\,^2$ that contains the boundary $\p \Ee_n$ and has
 symplectic form $\Tilde\om$  such that $\Tilde\om(E_i) \approx w_i$.
The elements of the normalized weight sequence are integers, so we can consider the homology class
\begin{align}\label{eq:B}
B = h_{2n+2}L - W_1 E_1 - \ldots -
W_m E_m=: 3\ell_{n+1} L - E(b_{n}) \in H_2(\widehat{\C P}\,\!^2).
\end{align}
We will want to record some information about the class $B$.

Using \eqref{eq:weight} and \eqref{eq:Fib2}, we find that
\begin{align}
\label{eqn:selfintersection}
B \cdot B &= h_{2n+2}^2 - h_{2n+1}h_{2n+3}=1,\\ \notag
c_1(B)& = 3h_{2n+2}-h_{2n+1}-h_{2n+3}+1 = 1,
\end{align}
where the last equality holds by  \eqref{eq:Fib1}.
  Thus, spheres in class $B$ have Fredholm index zero.

Finally, note that for each $n>0$ we have
\begin{equation}
\label{eqn:area}
\omega(B)\;=\;  h_{2n+2} \om(L) - \sum_i W_i\,\om(E_i) \; \approx \frac{h_{2n+2}^2-h_{2n+1}h_{2n+3}}{h_{2n+1}}\;=\;\frac{1}{h_{2n+1}}.
\end{equation}
We will frequently use the fact that when we stretch, the symplectic area $\omega(B)$ is the sum of the {\bf action} of each curve in any level of the resulting building.
Here we define the action for each part of the building using the formulas \eqref{eqn:action2}, \eqref{eqn:action3}, and \eqref{eqn:action4}; the claim about the action of the limit follows  immediately from the fact that  contributions to the action of the building from matching pairs of ends cancel.

We now claim that there is a sequence of Cremona transforms taking the class $B$ to the class $3L - E_1 - \ldots - E_8$.  This essentially follows from \cite[Prop. 4.2.7]{MS}.  To elaborate, there the authors consider a vector $v$ which is given by modifying $\widetilde{v}:=(h_{2n+2}; W_1,\ldots,W_{\ell})$ by replacing two of its   entries of $1$
by a single entry with value $2$ (note that by \cite[Eq. 4.12]{MS}, there are $7$ ones at the end of $\widetilde{v}$; in this regard, it is helpful to note that the $b_{2n+1}$ in our notation correspond to the ``$v_n(7)$" in the notation used there.) They then show that there is a sequence of Cremona transforms taking $v$ to the vector $(1;1,1)$.  The sequence of moves they describe first transforms this vector to the vector $(3;2,1,1,1,1,1,1)=: (3;2,1^{\times 6})$, and then one reduces further to $(1;1,1)$;  see \cite[Lem. 4.2.9]{MS}. Since the first set of moves does not affect any of the last $7$ entries in $v$, when we apply these moves to
   $\widetilde{v}$ we obtain $(3;1^{\times 8})$, as required.

Since Cremona transforms preserve the deformation class of the symplectic form on a blow-up $\widehat{\C P}\!\,^2$, the classes $B$ and  $3L - E_1 - \ldots- E_8$ have the same (genus $0$) Gromov-Witten invariant.  Thus, the Gromov-Witten invariant of the class $B$ is $12$.  For a generic choice of compatible $J$, the relative adjunction formula then implies
that the class $B$ is represented by $12$ immersed spheres, each with one nodal point.  The idea is now to stretch
these curves, and show that some of them must break in such a way that the moduli space $\mathcal{M}$ in Proposition~\ref{prop:goodcurves} is nonempty.

\begin{defn} \label{def:CU}{\it
We denote by $C_U$  the {\bf top level} of the building that arises when we stretch, and by $C_L$ its {\bf lower part}, i..e the union of all the other levels of the limiting building.  Further we denote by $C_{LL}$ its {\it lowest level}.  Thus $C_{LL}\subset \widehat\Ee$. }
\end{defn}
 Since the blowing up operations all take place inside the ellipsoid $\Ee_n$,  the curve $C_U$ lies in the negative completion
$\overline{X}$
of $\C P^2(\mu_n + \eps')\less \Ee_n$, while the lowest  level $C_{LL}$ lies in the positive completion ${\widehat \Ee}_n$ of the blown up ellipsoid.  The building $C_L$ consists of $C_{LL}$ together (possibly) with some curves
 lying in the {\it neck}, i.e. in the symplectization of $\p \Ee_n$.  Those of our arguments that involve $C_L$ will only consider its  topological properties.  Hence later we will consider it to be a union of {\it matched components}: see Definition~\ref{def:conn}.

  \begin{lemma}\label{le:bottomend} When $n>0$ and $\eps,\eps'>0$ are sufficiently small,
there are only three possibilities for the lower end of $C_U$, namely  the orbit sets $\lbrace (\beta_1,h_{2n+3} ) \rbrace,$ $\lbrace (\beta_2,h_{2n+1}) \rbrace,$ and $\lbrace (\beta_1, \ell_{n}), (\beta_2, \ell_{n}) \rbrace$.
\end{lemma}
\begin{proof}
Note that by \eqref{eq:ell}, \eqref{eq:mu0} and \eqref{eq:Fib2},
  the maximal action
(i.e. symplectic area)
of the lower end of $C_U$ is
$$
\ell_n(\mu_n+ \eps) \approx
\frac{h_{2n+2} h_{2n+2}}{h_{2n+1}} =h_{2n+3}+\frac{1}{h_{2n+1}}.
$$
Now the action of $\be_1$ is $1$, while that of $\be_2$ is $b_n + \eps'\approx \frac{h_{2n+3}}{h_{2n+1}}$.
Because $b_n$ is rational, we may choose $\eps, \eps'$ so small that the orbit set
$\lbrace (\beta_1, s), (\beta_2, t) \rbrace$ at the bottom of $C_U$ satisfies $$
s + t b_n = s + \frac {t h_{2n+3}}{h_{2n+1}} \le h_{2n+3}+\frac{1}{h_{2n+1}}.
$$
On the other hand, the estimate for $\om(B)$ in \eqref{eqn:area} implies that (modulo $\eps,\eps'$) the action of the bottom of $C_U$ must be at least $h_{2n+3}$.  Thus the proof of the lemma is completed by Lemma~\ref{lem:actconsid} below.\end{proof}

\begin{lemma}
\label{lem:actconsid}
\begin{itemize}
 \item[(i)]   There are precisely two orbit sets of action $\approx h_{2n+3}$, namely $\{(\be_1, h_{2n+3})\}$ and  $\{(\be_2, h_{2n+1})\}$.
 \item[(ii)]   There is a  unique orbit set of action $\approx h_{2n+3}+\frac{1}{h_{2n+1}}$, namely
 $\lbrace(\beta_1, \ell_{n}),(\beta_2, \ell_{n}) \rbrace$.
 \item[(iii)] For any $0 \le x < h_{2n+3}$, there is at most one orbit set of action $x$.
\end{itemize}
\end{lemma}

\begin{proof}
Assume that we have two distinct orbit sets of the same action, and write
\[ a + b \frac{h_{2n+3}}{h_{2n+1}} = a' + b' \frac{h_{2n+3}}{h_{2n+1}},\]
for $a, a', b, b'$ nonnegative integers with $b\ge b'$.  We can assume without loss of generality that $b > b'$, else $b = b'$, then $a = a'$.  We know from above that
\[(b - b') \frac{h_{2n+3}}{h_{2n+1}} \]
is an integer.  We also know that $h_{2n+1}$ and $h_{2n+3}$ are relatively prime.  Hence, $(b-b')$ must be divisible by  $h_{2n+1}$, and so $a + b \frac{h_{2n+3}}{h_{2n+1}}$ must be at least $h_{2n+3}$.  This proves (iii).
Moreover,
the equation $k + \ell \frac{h_{2n+3}}{h_{2n+1}} = h_{2n+3}$  does have precisely two solutions,  namely
$(h_{2n+3},0)$
 and $(0,1)$, which proves (i).

This argument also shows that the orbit set of action $h_{2n+3} + \frac{1}{h_{2n+1}}$ must be unique.  Otherwise, there would be a solution to $a + b \frac{h_{2n+3}}{h_{2n+1}} = h_{2n+3} + \frac{1}{h_{2n+3}}$ with $b \ge h_{2n+1}$, which is impossible.
This
proves (ii).
\end{proof}

To compute the gradings of these three orbit sets as in \eqref{eq:gr0}, note that by Pick's Theorem,\footnote
{
This says that the area of a lattice triangle is $i+\frac b2 -1$ where $i$ is the number of interior lattice points and $b$ is the number of lattice points on the boundary.} the number of lattice points in the triangle with vertices $(0,0), (h_{2n+1},0)$ and $(0,h_{2n+3})$ is $\frac12 (h_{2n+1}+1)(h_{2n+3}+1)+1$.  This implies that we have
\begin{align}\label{eq:gr}
& \op{gr}(\lbrace(\beta_1,h_{2n+3})\rbrace) =(h_{2n+1}+1)(h_{2n+3}+1)-2, \\ \notag
& \op{gr}(\lbrace(\beta_2,h_{2n+1})\rbrace) = (h_{2n+1}+1)(h_{2n+3}+1), \\ \notag
&  \op{gr}(\lbrace(\beta_1, \ell_{n}),(\beta_2, \ell_{n}) \rbrace) = (h_{2n+1}+1)(h_{2n+3}+1)+2.
\end{align}
To prove Proposition~\ref{prop:goodcurves} in the case $n\ge 1$, we want to show that when we stretch we get at least one building such that $C_U$ has negative asymptotics $\lbrace (\beta_1,h_{2n+3}) \rbrace$.  To  this end, we prove the following:

\begin{proposition}
\label{prop:boundcurv}  If $n\ge 1$, then
when we stretch the curves in class $B$, there are at most $9$ such that $C_U$ has negative asymptotics $\lbrace (\beta_1,\ell_n), (\beta_2,\ell_n) \rbrace$ or $\lbrace (\beta_2,h_{2n+1} ) \rbrace$.
\end{proposition}

To be precise, for a sequence of almost-complex structures stretched to length $R_i$ we can label the holomorphic curves in class $B$ by $C^i_k$ for $k=1, \dots ,12$. By choosing a subsequence of $i \to \infty$ we may assume that each of the $C^k_i$ converge to a holomorphic building. The proposition claims that at most $9$ of these $12$ buildings have $C_U$ with asymptotics $\lbrace (\beta_1,\ell_n), (\beta_2,\ell_n) \rbrace$ or $\lbrace (\beta_2,h_{2n+1} ) \rbrace$.

\begin{cor} \label{cor:main}  {\it Proposition~\ref{prop:goodcurves}  holds when $n>0$.}
\end{cor}
\begin{proof}[Proof of Corollary]
Since there are $12$ curves in class $B$, Proposition~\ref{prop:boundcurv}  implies that there are at least three curves with negative asymptotics $\lbrace (\beta_1,h_{2n+3}) \rbrace.$  We  show in Proposition~\ref{prop:oneend} that these curves must have exactly one end.
Thus Proposition~\ref{prop:goodcurves} holds.
\end{proof}

The proof of Proposition~\ref{prop:boundcurv}  is  complicated and occupies most of the rest of this paper.  This section  considers the easier parts of the proof, that investigate what happens when $C_U$ has ends either just on $\be_1$ or just on $\be_2$.  Also, we  show in \S\ref{sec:stablzn}  how to deduce
the stabilization result Proposition~\ref{prop:goodC} using the arguments in  \cite{HK,CGH}.
Note however, that when $\tau^4 < x < 7$ the $4$-dimensional embedding obstruction for $E(1,x)\;\se\; B^4(\mu)$ is a curve $C_0$ in $\ov X$ of degree $3$ with two ends of multiplicity $1$, one on $\be_1$ and the other on $\be_2$; see Remark~\ref{rmk:C0}.  This has essentially zero action (i.e. it is a low action curve in the sense of \S\ref{sec:lowen} below) and,  as we saw in Lemma~\ref{le:bottomend},  $C_U$ might well be
one of its  multiple covers.  Understanding the structure of limiting buildings whose top is a multiple cover of $C_0$ requires much of the ECH machinery explained in \S\ref{sec:ECH}, and takes up \S\ref{sec:noconn} and \S\ref{sec:heart}.

\begin{rmk}\label{rmk:n=0} (i)  {\bf (The case $n=0$)}
The  formulas \eqref{eq:gr} hold for all $n\ge 0$.  Further if $n=0$ we can still use the same formula for $B$, i.e. we
have  $B: = 3L-E_1-\dots- E_8$.  However, the calculation for the action (or symplectic area) no longer works because we now
 can only embed $E(1,8+\eps)$  into $B^4(\frac{17}6+ \eps')$.  Thus  the class $B: = 3L-E_1-\dots - E_8$ has area $\approx \frac 12$.
More significantly, the curve $C_0$ no longer exists generically (since it has negative Fredholm index; see Remark~\ref{rmk:C0}), and the proof of Proposition~\ref{prop:oneend}  fails: indeed there can now be curves $C_U$ with more than one end on $\be_1$.
 For further details of this case see \S\ref{sec:n=0}.\MS

 \NI (ii) {\bf (The Fibonacci stairs)}   The proof outlined above is markedly easier at the points $a_n: = \frac{g_{n+2}}{g_n}$ considered in \cite{CGH}, where $g_\bullet = (1,2,5,13,34,\dots)$ is the sequence of odd-placed Fibonacci numbers.  In this case, the  curve to be stretched lies in class $B = g_{n+1}L - W(\frac{g_{n+2}}{g_n})$.  Since this is the class of an exceptional curve, it is represented by an embedded sphere, with ECH index equal to zero.
 Hence  ECH theory applies  to show that the top curve $C_U$ has just one end on $\be_1^{g_{n+2}}$, since the partition conditions imply that
 $p_{\be_1}^-(g_{n+2}) = (g_{n+2})$.
 Further details are left to the interested reader.   All the complications in our argument are caused by the fact that we start with a curve with one double point which may well disappear into the neck or lower part of the curve when we stretch, yielding limits whose top is a multiple of $C_0$. Notice also that when $x<\tau^4$, the calculation of the embedding function $c_0$ in \cite{MS}  shows that we can choose $\mu$ so small that a curve in class $C_0$ would have negative action, and so could not exist.
 \hfill$\er$
\end{rmk}

\subsection{The low action curves}\label{sec:lowen}

We begin by classifying the top level curves with action on the order of $\epsilon$, which we will call {\bf low action curves}.
 (In \S\ref{sec:heart} we  sometimes call these {\bf light } curves.)  The discussion after \eqref{eqn:area} implies that in any building in class $B$ all but one curve has low action in this sense.  This holds
 because $\om(B)\approx \frac 1{Q_{n}}$ (where $Q_n: = h_{2n+1}$) and the symplectic areas of the classes $L, E_i$ as well as the actions of the orbits on $E(1,x)$ also are approximately equal to multiples of $\frac 1{Q_n}$.

\begin{rmk}\label{rmk:eps}\rm
 For each $n$, the proof of Theorem~\ref{thm:main} given below involves a finite number of strict inequalities.
Throughout we assume that $\eps,\eps'$ (and any other similar constant) are so small
that each of these finite number of inequalities that holds when $\eps,\eps'=0$ also holds when they are  positive.
We will indicate that an irrational quantity $\theta$ is an approximation to some  (usually rational) quantity $b$ by writing
$\theta =_\eps b$ rather than $\theta\approx b$ as in \S\ref{sec:stablesetup}. \hfill$\er$
\end{rmk}

Throughout  we assume that $J$ is admissible and generic as explained at the beginning of \S\ref{sec:ECH}.

\begin{proposition}
\label{prop:lowenergy}
When $n=0$ there are no low action curves in class $dL$ with $d < h_{2n+2}$.  If $n>0$ the only such curves
occur when $d= 3m$ for some positive integer $m$ and have negative end on the orbit set $\lbrace(\beta_1,m),(\beta_2,m)\rbrace.$
\end{proposition}

\begin{proof}    If $n=0$, we have $h_{2n+2} = 3$ and  we use the embedding $E(1,8+\eps')\se B^4(17/6 + \eps)$ as in Remark~\ref{rmk:n=0}.
Hence $d<3$ so that the action of the class $dL$ is not an integer, while the action of the negative end is (approximately) an integer.  Hence there can be no low action curves in this case.

From now on we suppose $n>0$.
It suffices to prove
the proposition  when  $C$ is somewhere injective and hence has nonnegative ECH index.  If $C$ is in class $dL$, asymptotic to $\beta=\lbrace (\beta_1,\ell), (\beta_2, m) \rbrace$ then, by  \eqref{eq:gr00},  the ECH index $I(C)$ and action  of $C$ are given by
\begin{equation}\label{eq:iact}
I(C) = d^2+3d-\op{gr}(\beta), \qquad  {\rm action}(C)  =_\eps d\, \frac{h_{2n+2}}{h_{2n+1}}-\ell- m \frac{h_{2n+3}}{h_{2n+1}}.
\end{equation}
 To better understand the relationship between the grading and the action
 it is convenient to introduce the $\Nn(a,b)$ sequence,\footnote
 {
 This is the sequence  obtained by listing the numbers $\ell a + m b, \ell, m\ge 0$ in increasing order with repetitions. Note that if $\frac ba$ is irrational then the numbers $\ell a + m b$ are all distinct.}
that has  the following important property:
\begin{itemize}\item[$(\star)$]  if $\Nn_s(a,b) = \ell a + m b$ where $\frac ba$ is irrational,  then the orbit set $\{(\be_1,\ell),(\be_2,m)\}$ on $\p E(a,b)$ has grading $2s$.
 \end{itemize}
  The following combinatorial lemma is now key:
\MS

\NI
\begin{lemma}\label{le:comb}  Let $n>0$, $d < h_{2n+2}$, and $k=\frac 12 (d^2+3d)$.  Then $$
\Nn_k\bigl(1,\frac{h_{2n+3}}{h_{2n+1}}\bigr) = \frac{h_{2n+2}}{h_{2n+1}} \; \Nn_k(1,1)
$$
 only if
 $n>0$ and
 $\Nn_k(1, \frac{h_{2n+3}}{h_{2n+1}}) = m  + m \frac{h_{2n+3}}{h_{2n+1}}$ for some integer $m>0$.
 \end{lemma}
\begin{remark}
\label{rmk:uniqueness}
The following observation is helpful for understanding this lemma.   If $n>0$ and $k, k' \le \frac 12(d^2 + 3d)$ where $d = h_{2n+1}-1$, then
$$
\Nn_k\bigl(1,\frac{h_{2n+3}}{h_{2n+1}}\bigr)=\Nn_{k'}\bigl(1, \frac{h_{2n+3}}{h_{2n+1}}\bigr)\;  \Longrightarrow\; k=k'.
$$
In other words there are no repeated  entries in this sequence until we get to the two points on the line
$x + \frac{h_{2n+3}}{h_{2n+1}} y = h_{2n+3}$ that each give entries $\Nn_\bullet(1,\frac{h_{2n+3}}{h_{2n+1}}) =  h_{2n+3}$ and these occur at  places $k>
\frac 12 (d^2 + 3d)$.
We saw in Lemma~\ref{lem:actconsid} that there are no repeated entries in $\Nn(1, \frac{h_{2n+3}}{h_{2n+1}})$ until we get to the two entries of $h_{2n+3}$.  Hence all we need to check is that these points occur for  sufficiently large $k$.

 To see this, note that $\Nn(1,1)=(0,1,1,2,2,2,\dots, \ell,\dots,\ell,\dots)$, with $\ell+1$ entries of $\ell$ for each $\ell$.  Therefore
 for each $s$ there are $\sum_{\ell=1}^s (\ell+1) = \frac 12 (s^2+3s)+1$ entries that are $\le s$.
Hence, for any $k \le \frac 12(s^2+3s),$ we have $ \Nn_k(1,1) \le s$ (recall that our convention is that the sequence $\Nn_k(a,b)$ is indexed starting at $k=0$).  Applying this with $s=h_{2n+2}-1$ together with
 the Monotonicity Axiom for ECH capacities from \cite{qech},\footnote
{
This says that if ${\rm int }\, (E(a,b))\se E(c,d)$ then $\Nn(a,b)$ is termwise no larger than $\Nn(c,d)$; we apply this to the embedding
${\rm int }\, (E(1, b_n))\se B^4( \frac{h_{2n+2}}{h_{2n+1}})$ from \eqref{eq:mu0}.
}
 we find
that  for $k$ in this range
\begin{equation}\label{eq:MONO}
\Nn_k\bigl(1, \frac{h_{2n+3}}{h_{2n+1}}\bigr)\;\le\; \frac{h_{2n+3}}{h_{2n+2}}\, \Nn_k(1,1) \; \le\;   \frac{h_{2n+3}}{h_{2n+2}} ( h_{2n+2} - 1) < h_{2n+3}
\end{equation}
as required.
  \hfill$\er$
\end{remark}

\NI {\bf Proof of Proposition~\ref{prop:lowenergy} for the case $n>0$, assuming  Lemma~\ref{le:comb}.}

  If $C$ has low action and bottom asymptotic to the orbit set $\be: = \{(\be_1,\ell),(\be_2,m)\}$,  then by \eqref{eq:iact}
we must have $d\frac{h_{2n+3}}{h_{2n+2}} = \ell + m \frac{h_{2n+3}}{h_{2n+1}}$.  The
term on the right is an entry in  $ \Nn(1, \frac{h_{2n+3}}{h_{2n+1}})$, while the
term on the left is an entry in
the sequence  $\frac{h_{2n+3}}{h_{2n+2}}\Nn(1,1)$ which as explained in Remark~\ref{rmk:uniqueness} we may take to occur
at the place $k=\frac 12(d^2 + 3d)$.  Therefore,
with $k =\frac 12(d^2 + 3d)$, there is $k'$ such that
$$
\frac{h_{2n+2}}{h_{2n+1}}\ \Nn_k(1,1) = \Nn_{k'}\bigl(1, \frac{h_{2n+3}}{h_{2n+1}}\bigr).
$$
But, if $n>0$  the ECH Monotonicity Axiom  implies that as in \eqref{eq:MONO}
$$
\Nn_{\bullet}\bigl(1, \frac{h_{2n+3}}{h_{2n+1}}\bigr) \; \le\; \frac{h_{2n+2}}{h_{2n+1}} \Nn_{\bullet}(1,1),
$$
so that $ \Nn_{k}(1,1) \le  \Nn_{k'}(1,1)$.
Because the sequence $\Nn(1,1)$ has $d+1$ entries of $d$ with the last one at place $k = \frac 12(d^2+3d)$, we could have $k'< k$, but if we do
we find that
$$
\frac{h_{2n+2}}{h_{2n+1}}\ \Nn_{k'+1}(1,1)\; =\; \frac{h_{2n+2}}{h_{2n+1}}\ \Nn_k(1,1)  \;=\;
 \Nn_{k'}\bigl(1, \frac{h_{2n+3}}{h_{2n+1}}\bigr) \; < \;  \Nn_{k'+1}\bigl(1, \frac{h_{2n+3}}{h_{2n+1}}\bigr),
 $$
 where at the last step we use the result in Remark~\ref{rmk:uniqueness}.  But this
 contradicts the Monotonicity Axiom.   Hence $k'\ge k$.    On the other hand, property
  $(\star)$ for $\Nn$ implies that
$\op{gr}(\be) = 2 k'$ and we know $I(C)= d^2 + 3d -2k'  = 2k-2k'\ge 0$.  Hence $k=k'$.
We now apply Lemma~\ref{le:comb}  to deduce that
$$
\Nn_k\bigl(1,\frac{h_{2n+3}}{h_{2n+1}}\bigr) = m + m\frac{h_{2n+3}}{h_{2n+1}}.
$$
Hence the negative asymptotics of $C$ are $\lbrace (\beta_1,m),(\beta_2,m) \rbrace$. It follows that  its action is
$=_\eps  d(\frac{h_{2n+2}}{h_{2n+1}}) - m(1+\frac{h_{2n+3}}{h_{2n+1}})$.  But by  \eqref{eq:Fib1} this is zero  only if $m=\frac d3\in \Z$.
\end{proof}

It remains to prove Lemma~\ref{le:comb}.

\begin{figure}
\label{fig:fig1}
\centering
\includegraphics[scale=.5]{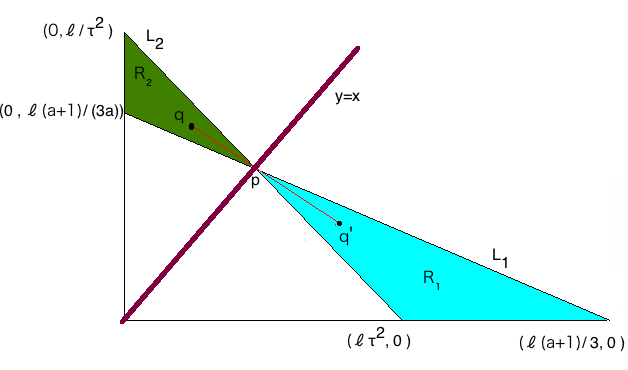}
\caption{The injection of lattice points used to prove Lemma~\ref{le:comb}.  We illustrate the case where $\ell$ is divisible by $3$; the other cases are similar.}
\end{figure}

\begin{proof}[Proof of Lemma~\ref{le:comb}]
The assumption on $k$ implies that $N(1,1)_{k+1} > N(1,1)_k$.  Define $\ell = N(1,1)_k$.  Then $k$ is equal to the number of lattice points in the triangle with vertices $(0,0),(\ell,0)$ and $(0,\ell)$, minus $1$.
Since $n>0$ we have   $ \tau^4< b_n < 7$, so that we  may apply \cite[Prop~2.4]{CGLS}  which states that
 $k$ is exactly the number of lattice points in the triangle $\mathcal{T}_1$ with vertices $(0,0), (0,\ell/\tau^2), (\ell \tau^2,0)$, where $\tau=\frac 12 (1+\sqrt{5})$.\footnote
 {
 To see that this fits into the discussion in \cite{CGLS} note that $\tau^2 + \frac 1{\tau^2} = 3$ so that $\mathcal{T}_1$ is a rescaling of the triangle with $\al,\be = (3,3)$ considered in  \cite[Thm~1.1]{CGLS}.
 }
Denote by $\mathcal{T}_2$ the triangle with vertices $(0,0), (0,\frac{ \ell(a+1)}{3a}), (\ell \frac{a+1}3,0)$, thus with slant edge of slope $-\frac 1a$, where $a: = b_n$.
\MS

 \NI {\bf Claim:}  {\it The number of lattice points in the triangle $\mathcal{T}_1$ is less than or equal to the number of lattice points that are both within $\mathcal{T}_2$ and strictly underneath the part of the upper boundary of $\mathcal{T}_2$ that is to the right of the line $y=x$ (this is the line $L_1$ in Fig.~1).}

\begin{proof}[Proof of Claim]     The line $y=x$ intersects the line $\frac{3}{a+1}x+\frac{3a}{a+1}y = \ell$ at $P=(\ell/3,\ell/3).$ Note that the point $P$ is independent of $a> \tau^4$; hence, the lines $L_1$ and $L_2$ given by $\frac{3}{a+1}x+\frac{3a}{a+1}y = \ell$ and $(1/\tau^2)x + \tau^2y = \ell$ respectively intersect at the point $P$.  We therefore have
\begin{equation}
\label{eqn:countingequation}
\# \lbrace \mathbb{Z}^2 \cap \mathcal{T}_a(\ell) \rbrace =  \# \lbrace \mathbb{Z}^2 \cap \mathcal{T}_{\tau^4}(\ell) \rbrace + D - U,
\end{equation}
where $D$ is the number of points in the region $R_1$ bounded by the lines $L_1$, $L_2$ and the $x$-axis (the blue region), not including lattice points on the left boundary, and $U$ is the number of lattice points in the region $R_2$ bounded by the lines $L_1, L_2$, and the $y$-axis (the green region), not including lattice points on the lower boundary, see Figure~\ref{fig:fig1}.

We must prove that   $U \le D$.
We now show that
there is a simple explicit injection of the lattice points counted by $U$ into the lattice points counted by $D$, as illustrated in the figure.
This is easiest to see in:

\MS

\NI
{\bf Case 1.} {\it  $\ell \equiv 0$, mod $3$. }

In this case, the point $P$ is a lattice point.  So, if $Q$ is a lattice point counted by $U$, let $V = P-Q$ and consider $Q' = P+V$. This is also a lattice point, see Fig.~\ref{fig:fig1}.  It lies on the line of slope $V$ passing through $P$, so to see that $Q'$ is counted by $D$, we need to show that the $y$-coordinate of $Q'$ is nonnegative, and $Q'$ is not on the left most boundary of $R_1$.  The second condition is immediately verified, since the line $L_2$ has irrational slope, hence $P$ is the unique lattice point on it.  For the first condition, observe that if $Q$ is in $R_2$, then the $y$-coordinate of $Q$ is no larger than $\ell/\tau^2$, hence the $y$-coordinate of $Q'$ is bounded from below by $\ell (2/3 - 1/\tau^2) > 0$.

Since the assignation $Q \to Q'$ is an injection, the claim holds in this case.
\MS

\NI
{\bf Case 2.}  $\ell \in \lbrace 1, 2 \rbrace$, mod $3$.

In this case, the point $P=(x,x)$ for $x$ satisfying ${\rm frac(x)} \in \lbrace 1/3, 2/3 \rbrace$ (here, ${\rm frac(x)}$ denotes the fractional part).  We now argue similarly as in the previous case.  Let $Q$ be a lattice point counted by $U$, let $V=P-Q$, and consider $Q'=P + 2V$.  This is a lattice point, on the line of slope $V$ passing through $P$; it is not on the leftmost boundary of $R_1$, because the line $L_2$ has no lattice points on it at all, being a line of irrational slope passing through a nonintegral rational point.  To see that the $y$-coordinate of $Q'$ is nonnegative, we observe similarly to above that if $Q$ is in $R_2$, then the $y$-coordinate of $Q$ is no larger than $\ell/\tau^2$, hence the $y$-coordinate of $Q'$ is bounded from below by $\ell (1 - 2/\tau^2) > 0$.
Since the assignation $Q \to Q'$ is an injection as above, the claim holds in this second case as well.

This completes the proof of the Claim.
\end{proof}

Now if $\Nn_k(1,\frac{h_{2n+3}}{h_{2n+1}})=\frac{h_{2n+2}}{h_{2n+1}}\Nn_k(1,1)=\frac{h_{2n+2}}{h_{2n+1}}\ell$, then by Remark~\ref{rmk:uniqueness}, we can find a unique lattice point $(m,n)$ satisfying $$
m \frac{h_{2n+1}}{h_{2n+2}}+ n \frac{h_{2n+3}}{h_{2n+2}} =\ell.
$$
Thus $(m,n)$ is the unique lattice point on the upper boundary of the triangle $\mathcal{T}_2$.  To prove the lemma we must show  that $(m,n)$ is on the line $y=x$.

To see this, assume that $(m,n)$ is strictly to the right of the line $y=x$.  Then by what was said previously, there are strictly more lattice points in the triangle $\mathcal{T}_2$ than in the triangle $\mathcal{T}_1$.  This implies that
$$
\Nn_{k'}\bigl(1,\frac{h_{2n+3}}{h_{2n+1}}\bigr)=\frac{h_{2n+2}}{h_{2n+1}}\; \Nn_k(1,1)=
\Nn_k\bigl(1,\frac{h_{2n+3}}{h_{2n+1}}\bigr)\;\; \mbox { for some }k' > k,
$$
 a contradiction.

Now assume that $(m,n)$ is strictly to the left of the line $y=x$.  In fact, there are no such lattice points.  To see this, observe that the intersection of the line $y=x$ with the slant edge of the triangle $\mathcal{T}_2$ occurs at the point $(\frac \ell3,\frac\ell 3)$.  Since the slant edge has slope $-\frac{h_{2n+1}}{h_{2n+3}}$, if there is such a point $(m,n)$, it has integral nonnegative  coordinates of the form
$$
\bigl(\frac \ell 3-\delta, \frac\ell 3+\delta\frac{h_{2n+1}}{h_{2n+3}}\bigr) \quad\mbox{ for some } \de>0.
$$
But then  $3 \delta\frac{h_{2n+1}}{h_{2n+3}}\in \Z$, which, because $\gcd(h_{2n+1}, h_{2n+3}) = 1$, implies $\delta \ge \frac{h_{2n+3}}3 >\frac \ell 3$.  Thus
 $\frac \ell 3-\delta$ must be negative, a contradiction.

It follows that $(m,n)$ must be on the line $y=x$, which completes the proof of  Lemma~\ref{le:comb} and hence also of
Proposition~\ref{prop:lowenergy}.
\end{proof}

\begin{rmk}\label{rmk:C0}\rm
The numerical arguments in the above proof of Proposition~\ref{prop:lowenergy} leave open the possibility that $C_U$ could be a multiple cover of a curve $C_0$ in class $3L$ with two ends, one on $\be_1$ and the other on $\be_2$, both of multiplicity one.
When $\tau^4 < x < 7$, such
 a curve has index $2\bigl(-1 + 9 - 1 - (1 + \lfloor x\rfloor)\bigr) = 0$ if $6<x<7$, and one can check that its ECH index is also $0$ for these $x$.
Further, one can show that $C_0$ exists either by considering what happens to an
 exceptional sphere in the class $A=3L-2E_1-E_2 - \dots E_8$ when  the neck  is stretched, or
  by considering the ECH cobordism map from the boundary of the (distorted) ball $\p B^4(\mu)$ to the ellipsoid $\p E(1,b_n)$ as in \cite{CGH}.   The corresponding embedding obstruction  is $ 3\mu > 1+x$, which gives $c_0(x) \ge \frac{1+x}3$.  Hence by
  \cite{MS}, this obstruction is sharp in dimension $4$.

  But because $C_0$ has two negative ends, the stabilization arguments in Proposition~\ref{prop:goodC1} do not apply; and indeed this obstruction cannot persist in higher dimensions because we know from \cite{Hi} that $c_k(x) \le \frac {3x}{1+x}$ for $k\ge 1$. \hfill$\er$
\end{rmk}

\begin{remark}
In the proof of the Claim above, the only relevant property needed of the $b_n$ is the inequality $b_n > \tau^4$.  Hence, the argument also shows that $\mathcal{N}(1,a) \le \frac{a+1}{3} \mathcal{N}(1,1)$,
which implies that
there is a symplectic embedding $E(1,a) \se B^4(\frac{a+1}{3})$ by McDuff's
proof of the Hofer conjecture in~\cite[Thm. 1.1]{M1}.  This gives a new and probably simpler proof of the computation of $c_0(x)$ for $\tau^4 \le x \le 7$ in \cite[\S 4]{MS}.  (The reverse inequality needed for this computation, i.e. that $\mathcal{N}(1,a) \le \lambda \cdot \mathcal{N}(1,1)$ implies  $\lambda \ge \frac{a+1}{3}$, follows directly by looking at the tenth term in each sequence.)
\end{remark}

\subsection{Analysing $C_U$ when $n>0$}

This section contains the proof of Proposition~\ref{prop:boundcurv} when $n>0$, modulo some results that are deferred to \S\ref{sec:noconn}
and \S\ref{sec:heart}.
By Lemma~\ref{le:bottomend} there are three possibilities for the negative asymptotics of $C_U$, namely
 $\lbrace (\beta_2,h_{2n+1}) \rbrace,$   $\lbrace (\beta_1, \ell_{n}), (\beta_2, \ell_{n}) \rbrace$ and $\lbrace (\beta_1,h_{2n+3} ) \rbrace$.
We discuss these cases in turn.

\begin{lemma}
\label{lem:nozero}
The negative end of $C_U$ cannot be $\lbrace (\beta_2,h_{2n+1}) \rbrace$.
\end{lemma}

\begin{proof}
If $C_U$ were
reducible, then one of its irreducible components would have to have low action so that by Proposition~\ref{prop:lowenergy}
its lower end would have to involve the orbit $\be_1$ as well as $\be_2$.  So this cannot happen.  Hence, since  $C_U$ cannot be
irreducible and multiply covered because $\gcd(h_{2n+1},h_{2n+2}) = 1$,
we conclude that
$C_U$ must be irreducible and somewhere injective.  By the grading calculations in \eqref{eq:gr},
we also must have $I(C_U)=0$, and thus $\op{ind}(C_U)=0;$ so the negative ends of $C_U$ must satisfy the partition conditions.
These are described in Remark~\ref{rmk:p+}~(ii).  By  \eqref{eqn:CZQ},
 when we compute $\op{ind}(C_U)$ using the index formula \eqref{eqn:Xind},
 we find
 $$
\tfrac 12 \ind(C_U) = - 1  + 3\ell_{2n+2}  - (h_{2n+1} + h_{2n+3} +1 - s) = s-2,
 $$
where $s$ denotes the number of
negative ends and we have applied \eqref{eq:Fib1}.
But  Remark~\ref{rmk:p+}~(ii) shows that $s\ge 8$, contradicting the fact that   $\ind(C_U) =0$.
\end{proof}

\NI {\bf The case with negative end $\lbrace (\beta_1, \ell_{n}), (\beta_2, \ell_{n}) \rbrace$:}

  In this case $C_U$
has ECH index $-2$ by \eqref{eq:gr}  and hence exists for generic $J$ only if it is a multiple cover.  Since it could be a multiple cover of the curve $C_0$ mentioned in Remark~\ref{rmk:C0}, we cannot ignore this possibility.    This case is analyzed by looking at the structure of the building $C_L$ formed by all but the top level of the limiting building (see Definition~\ref{def:CU}).  Thus $C_L$ consists  of some curves
 in the neck  together with some curves  $C_{LL}$  in the completion $\widehat\Ee_n$ of the blown up ellipsoid.  Note that, for generic $J$,  $C_L$ intersects the exceptional divisors transversally in $\sum_i W_i =  h_{2n+1} + h_{2n+3} - 1$ points; we think of each such intersection point as a {\bf constraint}.

 Many of our arguments are essentially topological rather than analytical  in nature, and it is convenient to introduce the following terminology.
 By a {\bf curve} or {\bf trajectory}  we mean the image of some $J$-holomorphic map $u$; thus curves lie in a single level of the building and  need not be connected.
  A curve with a connected domain is called {\bf irreducible}.
  We may join these curves along pairs of  positive and negative ends that match in the sense that each limits on the same multiple $\be_i^r$ of the same orbit, thereby decomposing the building into a union of connected pieces that we call {\bf matched components}, or {\bf components} for short.

Here is the key definition.

\begin{defn}\label{def:conn}  {\it A matched component in  $C_L$ is called a {\bf connector} if its top has ends both on $\be_1$ and on $\be_2$.}
\end{defn}

Thus a  connector might have two levels, the top (lying in the neck) consisting of a trivial cylinder over $\be_1^r$ together with a cylinder from $\be_2^s$ to $\be_1^m$, joined by an irreducible curve in $\widehat{\Ee}_n$ with two top ends, one on $\be_1^r$  and the other on $\be_1^m$.
\MS

\NI{\bf Proof of Proposition~\ref{prop:boundcurv}:}
By Lemma~\ref{lem:nozero}, Proposition~\ref{prop:boundcurv} will hold if we show that there are at most $8$
 $B$-curves that limit on a building such that $C_U$ has negative end $\lbrace (\beta_1, \ell_{n}), (\beta_2, \ell_{n}) \rbrace$ and $C_L$ has no connectors, and at most one more
 in which $C_L$ has a connector.  The first claim is proved in \S\ref{sec:noconn},  and the second (which is considerably harder) is proved in Proposition~\ref{prop:up} in \S\ref{sec:heart}.
\hfill$\Box$
\MS

\begin{example}\label{ex:558}\rm {\bf (The case $n=1$, i.e. with  $b_1 = \frac {55}8$)}  We illustrate what can happen  in the first nontrivial case.
Here $\ell_1 = 7$ and we have constraints
$$
E(\tfrac {55}8) =  8 (E_1 + \dots + E_6) + 7E_7 + E_8 + \dots + E_{14} =: 8E_{1\dots6} + 7E_7 + E_{8\dots (14)}.
$$
If $C_L$ has no connectors, we may divide its matched components into two groups  $D_1, D_2$, where, for $i=1,2$,  $D_i$ is a union of planes with top on some multiple of $\be_i$.
We will see in \S\ref{sec:noconn} that in this case there are eight possibilities for $C_L$, one  for each $k = 7,\dots,14$, that may be distinguished by the distribution of the constraints.
When $k=7$ then $D_1$ goes through $E_1,\dots, E_7$ (and has action $=_\eps \frac 18$), while $D_2$ goes through all the others and has low action.  On the other hand, if $k>7$ then
$D_1$ goes through the $8$ constraints $E_1,\dots,E_7, E_k$ and has low action, while $D_2$ goes through all the others.  Notice that in the first case in order for $D_2$ to have nonnegative index  it must be connected with a single top end on $\be_2^7$, while in the second case $D_1$ must be connected with one top end on $\be_1^7$.
Thus in both cases $C_U$ is a connected $7$-fold cover of $C_0$, with at least one end of multiplicity $7$.

If $C_L$ has a connector, then we show in \S\ref{sec:heart}  that $C_U$ is the union of $C_0$ with a $6$-fold cover of $C_0$,  and
that there is a unique connector $D_{12}$, with  two top ends on
 $\be_1, \be_2^6$, and action $=_\eps \frac 18$.  In this case $D_1$ consists of the $6$ planes  with top $\be_1$, each through a single constraint $E_i, 1\le i\le 6$, while $D_2$ has top $\be_2$ and goes through $E_1,\dots, E_7$.  All the other constraints  $6E_{1\dots 7} + E_{8\dots(14)}$ lie on $D_{12}$.

 As we will see, there are analogous decompositions for all $n>1$.  In particular, if there is a connector it is unique and has essentially all the action.   It has two top ends, one on $\be_1^{\ell_{n-1}}$ and the other on $\be_2^{\ell_n-\ell_{n-1}}$, and there is a  formula for its constraints in terms of weight expansions; see
 Proposition~\ref{prop:improva}.   As in the proof of Proposition~\ref{prop:n=0}  below, we use the fact that $B\cdot B = 1$ to show that for each such distribution of constraints  there is at most one $B$-curve that breaks this way.\hfill$\er$
  \end{example}

\NI {\bf The case with negative end $\lbrace (\beta_1,h_{2n+3}) \rbrace$:}

\begin{proposition}
\label{prop:oneend}  Let $n>0$.   Then, if $C_U$
 has negative ends on $\lbrace (\beta_1,h_{2n+3}) \rbrace$,  it must have a single negative end.
\end{proposition}

\begin{proof}
First note that
because the bottom constraint involves only $\be_1$  it follows as
in the proof of Lemma~\ref{lem:nozero} that
$C_U$  is
irreducible
and hence somewhere injective since $3\ell_{n+1} = h_{2n+2}$ and $h_{2n+3}$ are mutually prime.
By the relative adjunction formula~\eqref{eqn:adjunction} and the formulas in \eqref{eqn:cp2calc},
we must have
\begin{align}
\label{eqn:writheJ0}
w_{\tau}(C_U)  = - w_{\tau}^-(C_U)&  = 2 - s+ h_{2n+2}^2 - (3 h_{2n+2} -  h_{2n+3}) -2\delta(C),\\ \notag
& = 2-s + h_{2n+1} (h_{2n+3}- 1)+1 - 2\de(C),
\end{align}
where $s$ denotes the number of negative ends, and we have used the Fibonacci identity~\eqref{eq:Fibh}.
On the other hand,  if $(a_1,\ldots,a_s)$ is the partition given by the negative ends of $C_U$, then by Remark~\ref{rmk:writhexact}~(i) we have
\begin{equation}
\label{eqn:writheb}
w_{\tau}^-(C_U) \ge \sum_{i \ne j} \min(a_ip_j,a_jp_i) + \sum_{i=1}^s (a_i-1)p_i
\end{equation}
where
 $p_i = \lceil \frac{a_i h_{2n+1}}{h_{2n+3}} \rceil$.
We can therefore bound the right hand side of \eqref{eqn:writheb} from below by
\begin{align}
\label{eqn:lowerb}
\frac{h_{2n+1}}{h_{2n+3}}\, \bigl(\sum_{i \ne j} a_ia_j + \sum_{i=1}^s (a_i^2-a_i)\bigr) &= \frac{h_{2n+1}}{h_{2n+3}}\left( (\sum_i a_i)^2 - (\sum_i a_i)\right)\\ \notag
& =
h_{2n+1} (h_{2n+3}- 1),
\end{align}
since $\sum a_i = h_{2n+3}$.

  So regardless of the $a_i$, the right hand side of \eqref{eqn:writheb} is some integer bounded from below by $h_{2n+1} (h_{2n+3}- 1)$.   In fact, \eqref{eqn:writheb} must be strictly greater than $h_{2n+1} (h_{2n+3}- 1)$ as long as $s \ge 2$, since the bound \eqref{eqn:lowerb} comes from throwing away some fractional parts, which must be positive.

So, assume for the sake of contradiction that $s \ge 2$.  Then the right hand side of \eqref{eqn:writheb} must be at least $h_{2n+1} (h_{2n+3}- 1)+1$, so
 \eqref{eqn:writheJ0} implies that
 we must have
$s=2$ and $\de(C) = 0$.

We show below that  when $s=2$
\begin{equation}
\label{eqn:s=2}
\sum_{i=1}^2 (a_i-1)p_i\ge \frac{h_{2n+1} }{h_{2n+3}}\ \sum_{i=1}^2 (a_i^2-a_i) + 1.
\end{equation}
By~\eqref{eqn:writheb} this implies that
\begin{align*}
w_{\tau}^-(C_U)  & \ge \sum_{i \ne j} \min(a_ip_j,a_jp_i) + \frac{h_{2n+1} }{h_{2n+3}}\ \sum_{i=1}^2 (a_i^2-a_i) + 1\\
& > h_{2n+1}(h_{2n+3} -1) + 1.
\end{align*}
Hence, because $w_{\tau}^-(C_U)$ is an integer, we must have $w_{\tau}^-(C_U)\ge  h_{2n+1}(h_{2n+3} -1) + 2$, which contradicts \eqref{eqn:writheJ0}.

It remains to prove~\eqref{eqn:s=2}.  To this end, note first that because
 $a_1$ and $a_2$ are additive inverses modulo $h_{2n+3}$, without loss of generality we  have $\op{frac}(\frac{a_1 h_{2n+1} }{h_{2n+3}}) < \frac 12$.\footnote
 {
 To get strict inequality, we also use the fact that $\Tilde\theta $  is slightly larger than $\frac{h_{2n+1} }{h_{2n+3}}$.}
 This implies
  that
 $p_1 - \frac{h_{2n+1} }{h_{2n+3}} a_1 >\frac 12$ so that
  $$
  (a_1-1)p_1 - \frac{h_{2n+1} }{h_{2n+3}}(a_1^2-a_1) \;>\; \tfrac 12(a_1-1)\; \ge\; 1
  $$
 unless $a_1 \le 2$.  If $a_1\le 2$, then $$
 h_{2n+3}> a_2 \ge h_{2n+3}-2, \quad\mbox{ and }
1-\op{frac}\bigl( \frac{a_2h_{2n+1} }{h_{2n+3}}\bigr)\; \ge\; \frac{h_{2n+1} }{h_{2n+3}},
$$
 which implies
 that
 \begin{align*}
 (a_2-1)p_2 - \frac{h_{2n+1}}{h_{2n+3}}\bigl(a_2^2-a_2\bigr) &=  (a_2-1)\bigl(p_2-  \frac{a_2 h_{2n+1} }{h_{2n+3}}\bigr)\\
 & = (a_2-1)\bigl(1-\op{frac}(\frac{a_2  h_{2n+1} }{h_{2n+3}})\bigr) \\
  & \ge (a_2-1)\, \frac{h_{2n+1}}{h_{2n+3}}  > 1.
\end{align*}
   In either case, then, the claim is true, hence the proposition.
\end{proof}

\subsection{The case when $C_L$ has no connectors} \label{sec:noconn}

We now consider what happens when $n>0$, the negative end of $C_U$  is $\lbrace (\beta_1,\ell_{n}), (\beta_2,\ell_{n}) \rbrace$, and there are no connectors in the sense of Definition~\ref{def:conn}.
 We write $C_L = D_1 \cup D_2$ where $D_1$ denotes the union of matched components
  with ends only at $\beta_1$, and $D_2$ denotes the union of components  with ends only at $\beta_2$.
Thus a component of $D_2$ might consist of a cylinder in the neck with top end on some multiple of $\be_1$ and bottom on a multiple of $\be_2$,
completed by a union of planes in the blown up ellipsoid.
 Notice that if a $B$-curve is close to breaking into a building whose top is an $\ell_n$-fold cover of $C_0$ and whose bottom has no connectors, then we may cut the $B$-curve just {\it above} the neck into three pieces that approximate $\ell_n C_0, D_1,$ and $D_2$.

We record the homology classes of
$D_1$ and $D_2$  via integer vectors.  Here, as in \eqref{eq:B}, we order the exceptional classes in decreasing order of size, and we write
\[ (z_1, \ldots, z_k)\]
for the homology class
\[ -(z_1E_1 + \ldots + z_k E_k).\]
So, with this notation the homology class of $C_L$ is given by the normalized weight expansion
$$
W(b_{n}) = \bigl( W_1^{\times 6},  W_2, W_3^{\times 5} ,\dots, 7, 1^{\times 7}\bigr),
 $$
which by~\eqref{eq:weight} is a vector of total length $6n+8$ whose entries occur in blocks $\Bb = (\Bb_0, \Bb_1,\dots)$ whose lengths are given by the entries in the continued fraction expansion $[6, (1,5)^{\times (n-1)}, 1, 7]$ of $b_{n}=\frac{P_n}{Q_n}$.  The
symplectic areas (or actions)
of the exceptional classes are given (modulo $\eps$) by the vector $w(b_n)$.

Consider some representative of $B$, and let $z$ denote the homology class of $D_1$ and $y$ the homology class of $D_2$.  Thus, since we are assuming that there are no connectors, $z + y = W(b_{n})$.

\begin{proposition}
\label{prop:bound}
In any neck stretching, at most $8$ representatives of $B$ limit on a building with no connectors and
where $C_U$ has
negative end $\lbrace (\beta_1,\ell_{n}), (\beta_2,\ell_{n}) \rbrace$.
\end{proposition}

Proposition~\ref{prop:bound} will follow easily from the
 next two lemmas. Before reading the proof of Lemma~\ref{lem:keylemma},  it might be useful to refer back to Example~\ref{ex:558}.

\begin{lemma}
\label{lem:keylemma}  There are at most $8$ elements in  the collection  $\Cc$ of vectors $z$ that occur as possible homology classes for $D_1$.
\end{lemma}

\begin{proof}
{\bf Step 1.}  {\it Two elements $z, z'\in \Cc$ differ in at most two places.  Moreover any two entries differ by at most one.}

 Let $z$ and $z'$ be two homology classes with representatives $D_1, D_1'$, and consider $(z-z') \cdot (y' - y);$ let $D_2$ and $D_2'$ denote the representatives for $y$ and $y'$.  We have
\begin{equation}
\label{eqn:dotproductequation1}
(z-z') \cdot (y' - y) = z\cdot y' + z'\cdot y - z'\cdot y' -z\cdot y .
\end{equation}

We can understand the right hand side of \eqref{eqn:dotproductequation1} by observing that,
 because the orbits $\be_1, \be_2$ on $\p \Ee$ can be filled by discs that intersect once,
  the number of intersections between $D_1$ and $D_2'$ is given by $\ell_{n}^2-z\cdot y'$, and similarly for the other relevant pairings of the $D_i'$.  On the other hand,
  as we pointed out above there are $J$-holomorphic $B$-curves whose lower parts approximate $D_1\cup D_2$ and $D_1'\cup D_2'$ arbitrarily closely.  Hence all these intersection numbers are nonnegative, and bounded from above by $1$ because $B\cdot B = 1$.  Hence $|(z-z') \cdot (y' - y)|\le 2$.

In addition, because $z+y=z'+y'=W$
the vectors $z-z', y'-y$ are equal. Thus, if $z - z': = (\eps_i)$ we have
\begin{equation}
\label{eqn:dotequation}
0\le  \sum \eps_i^2 = (z-z') \cdot (y'-y)  \le 2.
\end{equation}
Step 1 follows readily.

\vspace{2 mm}

\NI
{\bf Step 2.}  {\it Completion of the proof.}

 Now recall that each of the classes $E_i$ have a definite area, and these areas come in blocks; no block has length greater than $7$.  We also know that the total action of
each $D_i$ must be close to either $0$ or $\frac 1{Q_n}$.

Also recall that only the last block  of the weight vector $w$ has
entries of
 size $\frac 1{Q_n}$; further, the only two blocks
 whose entries
 differ by $\frac 1{Q_n}$ are the second and third to last.
 We now prove that $\Cc$ has at most $8$ elements by a case by case analysis.

\vspace{2 mm}

\NI
{\bf Claim.} {\it  If any two $z$ and $z'$ in $\mathcal{C}$ differ on any block $\mathcal{B}$ other than the last three, then all the representatives in $\mathcal{C}$ differ
on this block.}

\begin{proof}[Proof of Claim]
Because the entries in $\Bb$ have size $> \frac 1{Q_n}$ and the total area of $D_1$ is at most $\frac 1{Q_n}$ it follows from Step 1 that the difference vector $z - z'$ must have exactly two entries, one $+1$ and one $-1$.  Further these must occur on the same block because $\Bb$ is not one of the last three blocks.
So, if $z''$ is any other representative in $\mathcal{C}$, $z''$ must have the same entries as $z$ and $z'$ on all other blocks: otherwise, either $z''$ would differ from one of $z$ and $z'$ in more than $2$ places, violating Step $1$, or it would have the wrong area.
\end{proof}

 \NI
{\em Case 1:}  With this claim in mind, consider the case where all of the elements in $\mathcal{C}$ differ on $\mathcal{B}$,
 which by assumption has length $\ell$, where $\ell = 5$ or $6$.
  Label the places in $\mathcal{B}$ (i.e. the entries in any vector in $\mathcal{C}$ corresponding to the block $\mathcal{B}$) $p_1, \ldots, p_\ell$, and fix an element $z$, with entries $a_1,\dots, a_\ell$ on this block.  Now consider another element $z' \in \mathcal{C}$.  Then $z'$ and $z$ must differ in two places, and without loss of generality we may assume that
$z' = (a_1+\eps_1, a_2-\eps_1, \dots,a_\ell)$, where $|\eps_1|=1$.  A third element $z''$ can differ from both $z,z'$ in the first place by at most $1$.  Hence the first place of $z''$ must be either $a_1$ or $a_1+\eps_1$; and similarly, its second place is either $a_2$ or $a_2-\eps_1$.  Further, if $z''$ is different from both $z,z'$ we may assume it differs from $z$ (and hence also $z'$) in place $3$, and hence is either
$$
(a_1+\eps_1, a_2, a_3-\eps_1, a_4,\dots, a_\ell) \quad\mbox{ or} \quad (a_1, a_2-\eps_1, a_3+\eps_1, a_4,\dots,a_\ell).
$$
Since the two third entries above differ by $2$, only one of these possibilities can occur.
Thus there are at most three elements in $\Cc$
whose entries differ only in  the places $1,2,3$ of the block.
Moreover, by interchanging $z,z'$ we may assume that $z''= z$ in its first place,
so that $z,z'$ differs only in the first two places, while $z',z''$ differ only in the second and third places.
If there is another element $z'''\in \Cc$ we may assume that it differs from $z,z',z''$ in the fourth place.  Again there are two possibilities for this element, only one of which occurs.  Finally there might be one or two more  elements in $\Cc$ that differs from the previously found elements in places $5$ and or $6$.  Thus $|\Cc|\le 6$, so that  Lemma~\ref{lem:keylemma} holds in this case.
\MS

\NI
{\em Case 2:}  The other case to consider, then, is the case where all the representatives in $\mathcal{C}$ differ only on the last three blocks.

\vspace{2 mm}

Note first of all that in this case, by area considerations, all of the elements in $\mathcal{C}$ either differ on the second and third to last blocks, or on just the last block.  In the case where the elements differ on the second and third to last blocks, we can repeat the argument from Case $1$, to conclude that there are no more representatives in $\mathcal{C}$ than the sum of the lengths of these two blocks.  Since the sum of these lengths is no more than $7$, this proves Lemma~\ref{lem:keylemma} in this case.

We can assume, then, that all of the entries in $\mathcal{C}$ differ on the last block.  With this in mind, label the places in this block
 $p_1, \ldots, p_7$, as above.  As above, first fix an element $z \in \mathcal{C}$, and consider another element $z' \in \mathcal{C}$.  Note that because $z + y = W$ has entries $1$ on the last block, all entries of $z$
are $0$ or $1$.  Since this is true for all elements in $\Cc$, any other element $z'\in \Cc$ is determined by
the places at which it differs from $z$.  We choose $z$   so that it has the smaller of the two areas of the elements in $\Cc$, and then write any $z'\in \Cc$ as $z+\eps$, where $\eps$ is a vector of length $7$ with entries $0$ or $1$ and we add modulo $2$.  Thus $\eps$ belongs to a collection $\Ee\subset \Z_2^7$ of vectors that
each have at most two nonzero entries.  Let $\Ee_1\subset \Ee$ be the subset of vectors of length $1$.
If $\Ee_1=\emptyset$, then the argument in Step 1 shows that $|\Ee| = |\Cc| \le 7$.   If $ \Ee_1\ne \emptyset $,
then our choice of $z$ implies that $area(z+\eps)>area(z)$ for $\eps\in \Ee_1$.
Therefore the places where $\eps\ne 0$ form a subset of the places where $z= 0$.
Therefore if $z$ has $k$ zero places, assumed w.l.o.g. to be the first $k$, there is a subset
$\Ii\subset \{1,\dots, k\}$ such that  for all $i\in \Ii$ there is an element
$z_i$ in $\Cc$ that agrees with $z$ except at the $i$th place where
$z$ has zero while $z_i$ has one.   Any another element $z$ in $\Cc$  must have a zero in some place $k$ where $z$ is one.  Hence because $z$ has minimal area, $z$ also has to have a one in some place $j$  that $z$ is zero.
But if  $i\in \Ii\less \{j\}$, then $z_i$ differs from $z$ in all three places $ i,j, k$.
Therefore, if $|\Ii|>1$ there are no such elements $z$, which implies that $|\Cc| = |\Ii| + 1 \le 8$. On the other hand if $|\Ii|=1$,  there are at most $7 - |\Ii|$ elements $z$, one for each place where $z$ has entry $1$.  Thus in all cases there are at most $8$ elements in $\Cc$.
This completes the proof of Lemma~\ref{lem:keylemma}. \end{proof}
\vspace{2 mm}

\begin{lemma}\label{lem:keylemma2} $D_1$ and $D_2$ must intersect.
\end{lemma}
\begin{proof}
We know that $z+y=W$,
 the normalized weight  vector; below, we also make use of the unnormalized weight  vector  $w$. Observe  that
\begin{equation}
\label{eqn:weightsequenceequation}
 P_nQ_n=W \cdot W = (z+y)\cdot(z+y) = z\cdot z+ y \cdot y + 2 z \cdot y .
\end{equation}
\NI
{\bf Step 1.}   {\it
The following key estimate holds:}
\begin{equation}
\label{eqn:estimate}
z\cdot z + y\cdot y \; \ge \;
\ell_{n}^2\, \bigl(\frac{Q_n}{P_n}+\frac{P_n}{Q_n}-1\bigr).
\end{equation}
We prove this using
an optimization argument.  Recall that one of the $D_i$ has area close to $0$, and the other has area close to $1/h_{2n+1}$.  If $D_1$ has area close to $0$, then $z$ is subject to the area constraint
\begin{equation}
\label{eqn:constraint}
z \cdot w = \ell_{n}.
\end{equation}
Hence $z \cdot z$ is minimized when $z=\lambda w$, where $\lambda$ is a scalar satisfying
\begin{equation}
\label{eqn:lambdaeq}
\lambda =
\frac{\ell_{n}}{w\cdot w} = \ell_{n}\frac{Q_n}{P_n}
\end{equation}
so that
\begin{equation}
\label{eqn:xoptimizationeq}
z \cdot z =
\ell_{n}^2\,  \frac{Q_n}{P_n}
\end{equation}
Similarly, the same argument gives that if $D_2$ has area close to $0$, then $y \cdot y$ is minimized when
\begin{equation}
\label{eqn:yoptimizationeq}
y \cdot y =
\ell_{n}^2\, \frac{P_n}{Q_n}.
\end{equation}
Thus, if $D_1$ and $D_2$ both had area close to $0$, then combining \eqref{eqn:xoptimizationeq} and \eqref{eqn:yoptimizationeq} would give \eqref{eqn:estimate}.   The only difference when $D_1$ has area close to
$\frac{1}{Q_n}$
is that the right hand side of \eqref{eqn:constraint} is smaller by
$\frac{1}{Q_n}$
and analogously for $D_2$.  This will weaken our estimate for $z \cdot z + y \cdot y$.  The difference is greatest when $D_2$ has area close to $\frac{1}{Q_n}$
 In this case, the right hand side of the analog of \eqref{eqn:constraint} is smaller by $\frac{1}{Q_n}$
so that  the right hand side of the analog of \eqref{eqn:lambdaeq} is smaller by $\frac{1}{P_n}$.
Hence, the right hand side of the analog of \eqref{eqn:yoptimizationeq} is smaller by an amount that is bounded from below by
$$
2 \ell_{n}\,  \frac{P_n}{Q_n} \cdot \frac1{P_n} = 2 \cdot \frac{\ell_{n}}{Q_n} < 2.
$$
Since  $z\cdot z + y\cdot y$ is an integer, we obtain  \eqref{eqn:estimate}.

\vspace{2 mm}

\NI
{\bf Step 2.}  {\it Completion of the proof.}

Given \eqref{eqn:estimate}, we can now prove the claim by using \eqref{eqn:weightsequenceequation}.
We know that $D_1$ and $D_2$ must intersect
\[ \ell_{n}^2 - z \cdot y \]
times.  By using \eqref{eqn:estimate} and \eqref{eqn:weightsequenceequation}, we get that $z\cdot y$ is bounded from above by
\[ \ka: = \tfrac 12 \left( P_nQ_n - \ell_n^2 \bigl(\frac{Q_n}{P_n} + \frac{P_n}{Q_n}\bigr)\right) + \tfrac 12.\]
Thus, the number of intersection points is bounded from below by $\ell_{n}^2 - \ka$.
But  $P_n^2 + Q_n^2 - 7P_nQ_n = 9$, which implies $\frac{Q_n}{P_n} + \frac{P_n}{Q_n}>7$.  Thus
$$
2(\ell_{n}^2 - \ka) > 9\ell_n^2 -  P_nQ_n = \frac 19((P_n + Q_n)^2 - 9P_nQ_n) = 1.
$$
Thus the intersection number is positive, as claimed.  \end{proof}

\NI \begin{proof}[Proof of  Proposition~\ref{prop:bound}]   For generic $J$ on the blow up $\widehat{\C P^2 }$, the class $B$ is represented by precisely $12$ disjoint embedded curves $C_\al, 1\le \al\le 12$.
 When we stretch the neck via a generic  family $J^R, R\to \infty,$ of almost complex structures,  we may choose an increasing unbounded sequence $R_i$  and choose labels for these $B$-curves so that, for each $\al$,  the sequence $(C_{i\al})_{i\ge 1}$ converges to some limiting building as $i\to \infty$.
 In view of Lemma~\ref{lem:keylemma},
it suffices to show that for each splitting $z,y$ of the constraints  there is at most one such sequence $(C_{i\al})_{i\ge 1}$  whose limiting building
has these constraints.

Suppose to the contrary that there were two such sequences.  Then as we remarked earlier for very large $i$ we may cut
the spheres $C_{i}$ and $C_{i}'$ just above the neck in such a way that the lower parts of these curves are unions $D_{1i}\cup D_{2i}$ and $D_{1i}'\cup D_{2i}'$,
where  $D_{1i}$ and $D_{1i}'$ are compact curves with   constraints $z$ and boundaries that are both are very close to $\be_1$, while $D_{2i}$ and $D_{2i}'$   have constraints $y$ and  boundaries  very close to $\be_2$.   But
 Lemma~\ref{lem:keylemma2} implies that  for large $i$ both  intersections $D_{1i}\cap D_{2i}' $ and $D_{1i}'\cap D_{2i}$ are nonempty.  Since these  are intersections of $J$-holomorphic curves, both intersections are positively oriented.  It follows that $C_i\cdot C_i'\ge 2$.  Since  these curves both represent the class $B$ which has $B\cdot B = 1$, this is impossible.
\end{proof}

\subsection{The case $n=0$}\label{sec:n=0}

As in Remark~\ref{rmk:n=0}, we take
\begin{equation}\label{eq:mu00}
\mu_* = \frac{17}6 + \eps',
\end{equation}
and start from the embedding
$$
\Phi: E(1,8+\eps)\;\se\; \C P^2(\frac {17}6 + \eps'),
$$
defining $\ov X$ to be the completed complement of its image. The class $B$ is now $3L-E_1 -\dots - E_8$,
 and we are interested in the structure of   the top part $C_U$ of the  limits of genus $0$ representatives of $B$ as the neck $\Phi(\p E(1,8+\eps))$ is stretched.  One main difference from the case $n>0$ is that now
 $C_U$ must limit on the orbit set  $\{(\be_1,8)\}$.  Indeed,
 since $C_U$ must have positive action it cannot be asymptotic to $\{(\be_1,1),(\be_2,1)\}$; because its action is $<1$ it cannot
 end on $\{(\be_1,r)\}, r<8$;
  while if it were asymptotic to $(\be_2,1) $ it would by \eqref{eqn:Xind} have index $2(-1 + 9 - 1-8)<0$ which is also impossible.

On the other hand,  it
is now possible for $C_U$ to have two negative ends on the orbit set
$\lbrace (\beta_1,8) \rbrace$.
  To see this, note that  if $C_U$ has a single negative end on $\lbrace (\beta_1,8) \rbrace$ then, by Remark~\ref{rmk:writhexact}~(ii) the writhe bound in \eqref{eqn:writhebound2} is an equality, and one can use it together with Proposition~\ref{prop:indineq} to calculate that
$2\de(C_U) = 2$, i.e. $C_U$ must have a double point.   On the other hand, the top part $C_U^A$ of the building obtained from   the exceptional sphere in class $A = 3L-2E_1-E_2-\dots E_7$  by stretching the neck  $\p E(1,8+\eps)$  must be embedded.  Therefore, $C_U^A$  must have more than one end on
$\lbrace (\beta_1,8) \rbrace$;  cf. Remark~\ref{rmk:8} below.

Here is our main result.

\begin{prop} \label{prop:n=0} When $n=0$ all representatives of $B$ break into a building whose top $C_U$ is connected and has negative end on
the orbit set
$\lbrace (\beta_1,8) \rbrace$. There are at most $8$ representatives with more than one end, and hence at least $4$ with just one end.
\end{prop}

The proof below shows that if $C_U$ has more than one end, it must have two ends of multiplicities $1,7$, and that there are at most $8$ such possibilities; cf. Remark~\ref{rmk:8}.  Our main tool is a writhe calculation for neck components.

  \begin{lemma}\label{le:CL}  When $n=0$, $C_U$ is connected and simple with negative end on
  $\lbrace (\beta_1,8)\rbrace$.  Moreover, the bottom level $C_{LL}$ consists of  eight disjoint and embedded components, each with top $(\be_1,1)$ and going through one constraint.
  \end{lemma}
\begin{proof}  We saw above that $C_U$ must have negative end on    $\lbrace (\beta_1,8)\rbrace$.
A curve of nonnegative index in $\ov X$  of degree $1$  has bottom end on $\lbrace (\beta_1,m)\rbrace$ for $m\le 2$ while
a similar curve of degree two has   bottom end on $\lbrace (\beta_1,m)\rbrace$ for $m\le 5$.  Hence $C_U$ must be connected,
and hence somewhere injective because $\gcd(3,8) = 1$.

The bottom level $C_{LL} $ of the limiting building lies in the completed blown up ellipsoid,  has top on $\lbrace(\beta_1,8)\rbrace$ with grading $16$ and goes through the constraints $E_1,\dots, E_8$.  Therefore,
 \eqref{eqn:Eei}  implies that $I(C_{LL}) = 0 = \ind(C_{LL})$.
 Therefore by  Proposition~\ref{prop:indineq},  $C_{LL}$ is embedded with ECH partitions.
Since $p^+\bigl((\be_1,8)\bigr) = (1^{\times 8})$,    $C_{LL}$ must have eight positive ends.
 But no component of  $C_{LL}$ can have more than one positive end because $C_U$ is connected and the original curve in class $B$ has genus zero.  Hence $C_{LL}$ must have $8$ components, which are disjoint, since their union has no double points.
 \end{proof}

\begin{lemma}\label{le:CL1}  When $n=0$, $C_U$ has either one negative end or two negative ends of multiplicities $1,7$.
Moreover, in the latter case there is one double point in the neck just before breaking.
 \end{lemma}
\begin{proof}
Consider the part  in the neck region of a stretched $B$-curve that is very close to breaking into a curve with top $C_U$.  This neck part is a union of
connected curve pieces that for short we call components,
one for each negative end of  $C_U$.  By Lemma~\ref{le:CL}, its bottom ends all have multiplicity one,
i.e. they approximate the simple orbit $\be_1$.
We denote by $C_s^{neck}$ a component  with one positive end
very close to $\be_1^s$,
and therefore $s$ negative ends each of multiplicity $1$.
Consider the union $C^{neck}: = C_{s_1}^{neck}\cup C_{s_2}^{neck}$ of two such neck components where $s_1\le s_2$.
Because $C_U$ and $C_{LL}$ are simple, we may estimate the writhe $w^+(C^{neck})$ at the top of $C^{neck}$ by using the writhe estimate for the appropriate {\it negative} ends of $C_U$ (which  is a lower bound), and the writhe $w^-(C^{neck})$ at its bottom by the formula for the {\it positive} end of $C_{LL}$ (which is an upper bound).  Thus  at the positive end we have $a_1=s_1, a_2 = s_2$ and $p_1=p_2=1$ so that
 $w^+(C^{neck}) \ge (s_1-1) + (s_2-1) + 2s_1$, while $w^-(C^{neck})\le 0$ since each end has $a_i=1,p_i=0$.  Thus the adjunction formula gives
\begin{equation}
\label{eq:neck}
2 \de(C_{s_1}^{neck}\cup C_{s_2}^{neck}) \ge 4-2-(s_1+s_2) + (s_1-1) + (s_2-1) + 2s_1 = 2s_1,
\end{equation}
i.e. there are at least  $s_1$ double  points in the neck.
More generally,  if
 the bottom of $C_U$ has $r$ ends with multiplicities $a_1\le a_2\le \dots \le a_r$,  then we have
\begin{equation}\label{eq:neck2}
n \ge  (r-1)a_1 + (r-2)a_2  + \dots +a_{r-1}
\end{equation}
double points in the neck.   Since we must have $n\le 1$ we find that $r=2$ and $a_1=1$ as claimed.
\end{proof}

\begin{proof}[Proof of Proposition~\ref{prop:n=0}]  Let us suppose that $C_U$ has two negative ends. By Lemma~\ref{le:CL} these have multiplicities $1,7$, and the curve $C_{1}^{neck}\cup C_{7}^{neck}$ has one double point.  We claim that each neck component is embedded. (Recall these are components of our curves mapping to the neck region just before breaking.) To see this, notice that the right hand expression  in formula~\eqref{eq:neck} decomposes as a sum of three terms,
the  terms $2-1-s_i + (s_i-1)$ that give  lower bounds for $2\de(C_{s_i}^{neck})$ and the term $2s_1$ that bounds twice the intersection number of the components.  Since there is at most one double point,  we must have $2\de(C_{s_i}^{neck})=0$ for each $i$.

To complete the argument, observe now that because
 there is a distinguished constraint that is attached to $C_1^{neck}$, there are eight ways to assign the constraints.
Moreover, as in the proof of Proposition~\ref{prop:bound} given at the end of \S\ref{sec:noconn}, for each such assignment   there is at most one  $B$-curve that is close to splitting in this way.
To check the latter statement,  note that  any  two such distinct curves $C,C'$ which are close to splitting would have
to intersect in the neck region  in at least two points, namely  $(C')_1^{neck}\cap C_7^{neck}$ and $C_1^{neck}\cap (C')_7^{neck}$ (which both consist of a single point by the calculation above).  But this is impossible
because all intersections count positively (since the curves are $J$-holomorphic) and $B\cdot B = 1$.
\end{proof}

\begin{rmk}\label{rmk:8}  (i) One can similarly study the possible splittings of the exceptional sphere $C^A$ in class $A =
3L-2E_1-E_2-\dots- E_7$. Again it must split along the orbit set $\lbrace (\be_1,8)\rbrace$, but now the grading of this orbit set is smaller than the contribution to the ECH index
of the  bottom constraint $\sum m_i E_i= 2E_1+E_2+\dots+ E_7$; indeed,  by \eqref{eqn:Eei}, this   is $2\sum_i m_i^2 + m_i$.

Hence $C_{LL}$ cannot be simple: indeed it consists of a two-fold cover of a cover of a component
$C_{LL}^1$ through $E_1$ together with six other components through $E_i, 2\le i\le 7$.
It is  plausible that $C^A$ splits into a building whose top level has two negative ends of multiplicities $1,7$, and that the component of $C_{LL}$ attached to $C_1^{neck}$  is $C_{LL}^1$ and so goes though $\al_{11}$, while the $7$ components attached to the $7$ negative ends of $C_7^{neck}$ go through $\al_{11},\dots,\al_{17}$.
For in this case one can imagine that the bottoms of  $C_1^{neck}$  and $C_7^{neck}$ are distorted by the attached copies of  $C_{LL}^1$  in such a way that the corresponding  neck components of an approximating $A$-curve no longer intersect.  Although we do not attempt to prove here that this  is what happens,  the fact that $C_{LL}$ is not simple does mean that the writhe calculation in Lemma~\ref{le:CL1}  is no longer valid.  This situation was misunderstood in \cite{HK}; see \cite{HK2}.
\MS

 \NI (ii)
 One can try to generalize this argument to ellipsoids of the form $E(1,3k-1), k>3.$  One would now start from  a curve in class $B = 3k - E_1-\cdots - E_{3k-1}$, with genus zero and $\de(B) = \frac 12(k-1)(k-2)$ double points,  stretch the neck, and then hope to find among the resulting buildings  at least one whose top level is a genus zero trajectory $C_U$ with just one negative end on $\be_1^{3k-1}$.  One can argue as above that the top level of each such building must have negative ends on the orbit set
  $\lbrace (\be_1,{3k-1})\rbrace$.  However, because $B$-curves have more double points, there are now more possibilities for the partitions of $3k-1$ that occur.  In particular the double points may not all lie in the neck and so be detectable by looking at the distribution of the constraints; instead one has to count the maximum number of curves $C_U$ for each possible partition of $3k-1$.
   One can attempt to do this by considering analogs for the exceptional class $A$ in (i) above.  For example, if $k=4$ the following classes (all with $c_1=1$ and $B\cdot A_i = 5$) are
  relevant:
  \[
    \begin{array}{lll}
   A_1 :  = 4L-2E_{1}-E_{2\dots 10}& A_2 :  = 4L-2E_{12}-E_{3\dots 9}&\mbox{ with } A_1^2 = 3, \; A_2^2 = 1\\
   A_3:  = 4L-3E_{1}-E_{2\dots 9} & A_4 :  = 4L-2E_{123}-E_{4\dots 8} &\mbox{ with } A_3^2=A_4^2 = -1.
   \end{array}
  \]
  Here $A_3,A_4$ are classes of exceptional spheres, while $A_2$ has one double point and $A_3$ has three.
 This approach is barely possible when $k=4$ (though to make a proof one would have to substantiate the claims in (i) above), and seems to show that there are some trajectories $C_U$ with just one negative end.  The argument uses the fact that there
 are $620$  $B$-curves, $12$ $A_2$-curves and   $96$  $A_1$-curves through a generic set of points.
Since  the genus zero Gromov--Witten invariant of the class $B$ and the relevant classes $A$ grows very rapidly with $k$,  this method does not seem feasible for large $k$. \hfill$\er$
 \end{rmk}

 \subsection{The stabilization process}\label{sec:stablzn}

We conclude  this section with a proof of the following sharpened version of Proposition~\ref{prop:goodC},  which gives conditions under which embedding obstructions in dimension $4$ persist under stabilization.

\begin{prop}\label{prop:goodC1}  Let $\ov X_{\mu_*,x}$ be the completion of $\C P^2(\mu_*) \less \Phi(E(1, x))$, where
 $x =b + \eps$, where $b=\frac pq$ for some relatively prime integers $p,q$ and   very small and  irrational $\eps>0$.
Suppose that for all sufficiently small $\eps>0$ and generic  admissible $J$ there is a  genus zero curve $C$ in $\ov X_{\mu_*,x}$ with degree $d$, Fredholm index zero, and one negative end on $(\be_1, p)$, where $gcd(d,p)=1$.  Then there is a constant $S(d,x)$ such that for any $S\ge S(d,x), k\ge 0$, and $\mu> 0$,
the existence of a symplectic embedding
\begin{equation}
\label{eqn:potentialemb}
E(1,x, \underbrace{S,\dots, S}_k) \;\se \; \C P^2(\mu) \times \R^{2k}
\end{equation}
implies that $\mu \ge \frac d{ p} = \frac{3b}{b+1}$
\end{prop}
\begin{proof}
 This is proved by slightly extending the arguments given in \cite[\S3,4,5]{CGH} that establish the corresponding result for the Fibonacci staircase $b= \frac{g_{n+2}}{g_n}$.
We start from the  stabilized  embedding
 $$
 \Tilde{\Phi}_0: = \Phi\times \iota:  E(1,x,S,\dots,S)\;\se\; \C P^2(\mu_*)  \times  \R^{2k}
 $$
where $\Phi$ is as in \eqref{eq:phi}, $\mu_*$ is given by \eqref
{eq:mu0} or \eqref{eq:mu00},
and $\iota$ is the obvious inclusion on the last $2k$-dimensions.  Consider the space $\ov{M}\,\!'$ obtained by removing the ellipsoid ${\rm im\,}
  (\Tilde{\Phi}_0)$ and completing at the negative end.
 By the argument in \cite[Lem. 3.1]{HK}, any embedding $\widetilde\Phi_1$ of
 $r \cdot E(1,x,S,\dots,S)$
 into $\C P^2(\mu_*)  \times  \R^{2k}$
  may be connected to $\Tilde{\Phi}_0$ by a $1$-parameter family $$
 \Tilde{\Phi}_t:\la(t)E(1,x,S,\dots,S)\to \C P^2(\mu_*)  \times  \R^{2k},\qquad t\in [0,1],
 $$
  of embeddings, for a suitable function $\la(t) \in (0,\op{max}(r,1)]$ with $\la(t) = 1$ for  $t$ near $0$ and $\la(t) = r$ for $t$ near $1$.  In particular, we may apply this to the embedding \eqref{eqn:potentialemb}, rescaled by $\frac{\mu_*}{\mu}$.  The idea is now to translate this family of embeddings into a family of compatible almost complex structures on a fixed manifold, in order to find a $J$-holomorphic curve that gives a constraint on $\mu$.

To elaborate, choose an identification of the completed spaces obtained by removing $\im (\Tilde{\Phi}_t)$ with $(\ov{M}\,\!', \om_t')$, where $\om_t'$ is a suitable family of symplectic forms; we assume that the identification is the  identity outside  $\C P^2(\mu_*)  \times  B^{2k}(T/2)$ for a suitably large $T$.
 We consider the space $\Jj'(T)$ of
 almost complex structures $J$ on $\ov{M}\,\!'$ that are admissible for some $ \om_t'$, and have product form  outside $\C P^2(\mu_*)  \times  B^{2k}(T/2)$ with projection  to $\R^{2k}$ equal to the standard complex structure.
  Our aim is to show that for generic $J\in \Jj'(T)$ that is admissible with respect to $\om_1'$
  there is at least one degree $d$ curve in $\ov{M}\,\!'$ with a single end on
 $\be_1^{p}$,
  since then the required bound comes from  the positivity of  its action.

  We may enlarge $T$ so that  the monotonicity theory  implies that every degree $d$ $J$-holomorphic curve
  lies inside $\C P^2(\mu_*)  \times  B^{2k}(T)$, and then compactify $\ov{M}\,\!'$ at its positive end, obtaining a family of symplectic manifolds $(\ov{M},\om_t)$ that are identified with completions of  subsets of $\C P^2(\mu_*)  \times  \C P^{k}(2T)$: see \cite[Lemma~3.3]{HK}.\footnote
 {In \cite{CGH}, the divisor $({\rm line}) \times  \C P^{k}(2T)$ is removed from the target,  one slightly distorts $\C P^2(\mu_*)\less ({\rm line})$ to $E(\mu_*,\mu_*+\eps)$, and then completes $M$ at its positive end as well.  However, how one treats the positive end makes no difference to the index calculations or stabilization arguments.}

 It therefore suffices to analyze curves in $(\ov M,\om_t)$ for $J\in \Jj(T)$, where $\Jj(T)$ is the obvious analog of $\Jj'(T)$.
 Because  $\im(\Tilde{\Phi}_0)$ and the symplectic form $\om_0$ on $\ov{M}$ are invariant with respect to a suitable $\T^k$-action that rotates the second factor  $\R^{2k}$,
  there is an $\om_0$-compatible element $J_M\in \Jj(T)$ which is also $\T^k$ invariant and restricts to an almost-complex structure $J_X$ on $\ov X$.
Thus, the initial curve $C$ in $\ov X$ whose existence we assume may be considered as an element of
 the moduli space $\Mm_{J_M}(\ov M, dL, \be_1^{p})$,  of somewhere injective
 genus zero curves of degree $d$
 and with one negative end on the short orbit $\be_1$ of multiplicity
 $p$.
 The index formulas in \cite[Prop.~13]{CGH} show that  $C$ has index zero for any $k$; cf. \cite[Lemma~14]{CGH} and \eqref{eqn:highdimindexformula} below.  The proof then consists of the following steps.
 \begin{itemize}
\item We  show that
the moduli space
$\bigcup_{t\in [0,1]}\Mm_{J_t}(\ov M, dL, \be_1^{p})$
is compact, where   $J_t\in \Jj, t\in [0,1],$
is  a generic path of almost complex structures.
\item
We show that the count of curves in
$\Mm_{J_M}(\ov M, dL,  \be_1^{ p})$
is positive.
\end{itemize}
The proposition then follows  immediately: for more details see \cite[\S5]{CGH}.

\MS

\NI {\bf The compactness argument}.

Consider a generic  $1$-parameter family $J_t\in \Jj(T),t\in [0,1],$ of almost complex structures on the completed $(2k+4)$-dimensional manifold $\ov M$.
 There are two possibilities for loss of compactness.

 The first is convergence to a multiply covered curve in $\overline{M}$.
  This can be excluded using the fact that $p$ and $d$ are coprime by hypothesis.  The other possibility is that there is a limiting building $C_{\infty}$.
Any such limiting building has a top component in $\ov M$ and lower levels in the symplectization $\p E(1,x,S,\dots,S)\times \R$.  To deal with this case we argue much as in \cite{CGH}, the crucial point being that the limiting building must have index zero.  It is convenient to think of $C_\infty$ as a union of matched components (see Definition~\ref{def:conn}) as follows.   We first match all possible curves in symplectization levels.  This
defines  a single component\footnote
{
This is called $S_0$ in \cite{CGH}.
}
$C_*$ with a negative end on $\be_1^{p}$, and
 other matched components
 in the symplectization
  with no negative ends.
  We then group the curves in $C_\infty\less C_*$  into matched components;
  in other words, we match each symplectization component with any compatible curves $\im (u_i)$ in $\overline{M}$ forming larger components.
We
 define the {\bf index of a  matched component}
 to be the sum of the deformation indices of the constituent curves minus the dimension of the orbit spaces where matching occurs (for ends asymptotic to covers $\beta_1$ or $\beta_2$ this dimension is just $0$, otherwise it is $2(k-1)$, the dimension of the Morse-Bott family).  By the index formulas for curves in symplectizations, the Fredholm index of any component formed by joining curves in symplectizations can be calculated as if the component was a curve in a single level, since the contributions to the Fredholm index of pairs of matching ends in symplectizations cancel.
 Similarly, if the component contains curves in the top level, then we can calculate its index as though all its components were in the top level.

We begin with the following.

\begin{lemma}\label{lem:largeS}
If $k\ge 1$, $S\ge S(d,x)$ is sufficiently large and the path $J_t$ is generic, then any  curve  in the limiting building  $C_\infty$
that lies in $\ov M$
has the following properties.

\begin{itemize}
\item Any ends on $ \p E(1,x,S,\dots,S)$ must be asymptotic to covers of  $\be_1$ or $\be_2$;
\item Every end must have multiplicity less than $S/x$.
\end{itemize}
\end{lemma}

\begin{proof}
In addition to $\beta_1$ and $\beta_2$, we have a Morse-Bott orbit $\gamma$.
To prove the first bullet point, it suffices to consider irreducible somewhere injective curves. Since $J_t$ is generic,
any irreducible somewhere injective curve $C$ in $C_\infty$ asymptotic to $\gamma$ has index $\ge -1$, and hence (because all indices are even) index $\ge 0$.  If $C$ has degree $d'$, ends on $\beta_1^{r_i}$ for $1 \le i \le n_1$, on $\beta_2^{s_j}$ for $1 \le j \le n_2$, and on $\gamma^{t_{\ell}}$ for $1 \le \ell  \le n_3$,  we find using the discussion following \eqref{eqn:freddefnN}  that

\begin{align}
\label{eqn:highdimindexformula}
\op{ind}(C) &  = (k-1)(2 - n_1 - n_2 - n_3) + 6d'  -
 \sum^{n_1}_{i=1} \Bigl(2 r_i + (2 \lfloor r_i / x \rfloor + 1)  + k(2 \lfloor r_i / S \rfloor + 1)\Bigr) \\ \notag
 & \qquad  - \sum_{j=1}^{n_2} \Bigl(2 s_j + (2 \lfloor s_j x \rfloor + 1) + k(2 \lfloor (s_j x) / S \rfloor + 1)\Bigr) \\ \notag
& \qquad \qquad   - \sum_{\ell=1}^{n_3} \Bigl( 2t_\ell + (2 \lfloor t_\ell S \rfloor+1)
+ (2 \lfloor t_\ell S/x \rfloor +1) + (k-1)(2t_\ell-1) \Bigr).
\end{align}
We must have $d' \le d$, and, if there are any ends on $\ga$, the term $ \lfloor t_\ell S \rfloor\ge \lfloor  S \rfloor$ is large while the very first term combined with the final term in the sum over $\ell$ together give a nonpositive contribution.

Therefore if $S\ge 3d$  the right hand side of \eqref{eqn:highdimindexformula} must be negative if $C$ has any ends on $\gamma$.  This gives a contradiction.

Similarly, from \eqref{eqn:highdimindexformula} no irreducible somewhere injective curve of degree $d'$  in $C_\infty$ can have ends of multiplicity more than $3d'$.
Since an arbitrary degree $d'$ curve is a sum of $m_i$-fold covers of irreducible somewhere injective degree $d_i$ curves, where $\sum m_i d_i = d'$,
the maximum multiplicity of an end of such a curve is $\sum_i m_i 3d_i = 3d' \le 3d$ which is $< \frac Sx$ if $S\ge S(d,x) : = \max(3 x d^2, 3d) = 3x d^2$.
Therefore we may take $S(d,x) = 3x d^2$.
\end{proof}

The first point in the next lemma is a slight generalization of \cite[Lemma~18]{CGH}.  Note that the proof uses the fact that we consider curves with bottom multiplicity $p$
rather than a more general $m$.

\begin{lemma}\label{negcurves}
As above, let $C_*$ be the unique matched component of $C_\infty$ in the symplectization
with a negative end on $\beta_1^{p}$,
and suppose that $S\ge S(x,d)$
Then $\mathrm{index}(C_*) \ge 0$
with equality if and only if $C_*$ is
an unbranched  cover of a trivial cylinder over $\beta_1$, and hence has one positive end of multiplicity $p$.
\end{lemma}
\begin{proof}
Every positive end of $C_*$ is matched by the negative end of some curve in $C_\infty$ that lies in $\ov M$.  Hence,
because $S\ge S(d,x)$
Lemma~\ref{lem:largeS} shows that $C_*$ has no  ends on the Morse--Bott orbit $\ga$.

Suppose that the  positive ends of $C_*$ are asymptotic to  $\beta_1^{r_i}, 1\le i\le n_1,$  and $\beta_2^{s_j}, 1\le j\le n_2$.
Then
by \cite[Prop.~17]{CGH} and  \cite[Lemma~18]{CGH} (see also the discussion after \eqref{eqn:freddefnN}), we have
\begin{align}\label{eq:Fredkind}
\tfrac{1}{2}\mathrm{index}(C_*) \ & = -1 + n_1 + n_2 + \sum_{i=1}^{n_1}(r_i + \lfloor \frac{r_i}{x} \rfloor) + \sum_{j=1}^{n_2}(s_j + \lfloor s_j (x) \rfloor) -  p - \lfloor \frac{ p}{x} \rfloor \\ \notag
&\  =   \sum(r_i + \lceil \frac{r_i}{x} \rceil) + \sum(s_j + \lceil s_j  x \rceil) -  p -    \lceil \frac { p} x \rceil.
\end{align}

 As $C_*$ has nonnegative area, we have
 $\sum r_i + \sum x s_j \ge  p$.  Hence
\begin{align}\label{eq:Fredkind1}
& \sum r_i + \sum \lceil x s_j \rceil\; \ge\;  p,
\end{align}
 and
\begin{align}\label{eq:Fredkind2}
&  \sum \lceil \frac{r_i}{x} \rceil + \sum s_j \ge \lceil \sum \frac{r_i}{x} \rceil + \sum s_j\; \ge\; \lceil \frac{ p}x \rceil.
\end{align}
These estimates establish the inequality.

Since $x$ is irrational, $\sum \lceil x s_j \rceil> xs_j$ whenever $s_j\ne 0$.  Therefore, because
 $\sum r_i + \sum x s_j \ge  p$,
there is equality in \eqref{eq:Fredkind1}  only if  $s_j=0$ for all $j$ and $\sum_i r_i = p$.
As for the second estimate,  first observe that because $x = \frac pq + \eps$, we have
$ \lceil \frac{ p}x \rceil = q$, so that $ \lceil \frac{ p}x \rceil$  is just larger that $\frac{ p}x $.  On the other hand
if $r<p$, then $\frac {r}b \notin \Z$ so that $ \lceil \frac{ r}x \rceil$ is significantly larger than $\frac {r}x$.
 Hence there is equality in  \eqref{eq:Fredkind2} only if $n_1=1$ and $r_1 = p$.
\end{proof}

We now complete the compactness argument by dividing into cases.
\MS

\NI  {\bf Case 1:}   {\it $C_*$ has positive index.}  In this case,  Lemma \ref{negcurves} implies that at least one matched component of the building $C_\infty \less C_*$ has negative index.  But this is ruled out by the next lemma.

\begin{lemma}\label{lem:cpct} Let $C$ be a matched  component of the  limiting  building $C_\infty\less C_*$. Then $\mathrm{ind}(C) \ge 0$, with equality only if $C$ consists of a single curve (hence with a single negative end matched with $C_*$).
\end{lemma}

\begin{proof}
We first recall some index formulas.
It follows from \cite[Prop.~17]{CGH}
that if $S\ge S(d,x)$  and
 $v$  is  a connected curve in the symplectization $\partial E(1,b+\eps,S, \dots ,S)\times \R$ with $n_1 + n_2$ positive ends on $\beta_1^{r_i}, 1\le i\le n_1,$
 and  on $\beta_2^{s_j}, 1\le j\le n_2$,  and no negative ends,  then
\begin{equation}
\label{cpct1}
\tfrac{1}{2} \mathrm{ind}(v) = k-1+ n_1 + n_2 + \sum_{i=1}^{n_1} \bigl(r_i +  \lfloor \frac{r_i}x \rfloor\bigr) + \sum_{j=1}^{n_2}\bigl( s_j + \lfloor s_jx \rfloor\bigr),
\end{equation}
where $k\ge 0$ is the
stabilization dimension\footnote
{
Note that $k = N-2$ in the notation of \cite{CGH}.}.  In particular, this index is always positive.
(This index formula can  be obtained using the observations after \eqref{eqn:freddefnN}; notice that  $
\lfloor\frac{r_j}{S}\rfloor =\lfloor\frac{s_j x}{S}\rfloor = 0$ by the second claim in  Lemma~\ref{lem:largeS}.)

Further, if  $\im(u)$ is a curve in $\overline{M}$ with degree $d'$ and negative ends on $\beta_1^{r_i}$ for $1 \le i \le n_3$ and negative ends on $\beta_2^{s_j}$ for $1 \le j \le n_4$, then similarly by\footnote
{
Note that this index is the same as the index of a curve with $\ell$ positive ends, each of multiplicity one, on the  larger orbit of
$\p B^4(\mu_*,\mu_*+\eps')\times \C P^k(2T)$.}
\cite[Prop.~13]{CGH},
\begin{equation}
\label{cpct3}
\tfrac{1}{2} \mathrm{ind}(u) = k-1+3d' - k(n_3+n_4) - \sum_{i=1}^{n_3} (r_i+\lfloor \frac{r_i}{x} \rfloor) - \sum_{j=1}^{n_4} (s_j+\lfloor s_j x \rfloor).
\end{equation}
If $u$ is  an $m$-fold multiple cover of a curve $\widetilde{u}$ with degree $\widetilde{d}$ and negative ends on $\beta_1^{\widetilde{r}_i}$ for $1 \le i \le \widetilde{n}_3$ and negative ends on $\beta_2^{\widetilde{s}_j}$ for $1 \le j \le \widetilde{n}_4$, then we have
 $d' = m\widetilde{d}$, $\sum r_i = m\sum \widetilde{r}_i$ and $\sum s_j = m \sum \widetilde{s}_j$. Therefore using equation \eqref{cpct3} we obtain
\begin{align}
\label{cpct4}
\tfrac{1}{2} \mathrm{ind}(u) - \tfrac{m}{2} \mathrm{ind}(\widetilde{u}) & = (k-1)(1-m) - k(n_3+n_4-m\widetilde{n}_3-m\widetilde{n}_4) \\
&\quad  - \sum_{i=1}^{n_3} \lfloor \frac{r_i}{x} \rfloor - \sum_{j=1}^{n_4} \lfloor s_j x \rfloor
  + m\sum_{i=1}^{\widetilde{n}_3} \lfloor \frac{\widetilde{r}_i}{x} \rfloor + m\sum_{j=1}^{\widetilde{n}_4} \lfloor \widetilde{s}_jx \rfloor. \nonumber
\end{align}

We now divide $C$ into the unique curve $u$ with a negative end matching $C_*$ and a collection of connected planar components. Such planar components necessarily have strictly positive index. (Indeed, the matched index is given by \eqref{cpct1} in the case where there is just one positive end, perhaps with an additional positive contribution if the matched  component
intersects $\ov M$ and hence
has positive degree.)

Suppose that $u$ is a multiple cover $m\widetilde{u}$ for some $m>1$.   Assume first that $u$ is  attached to $C_*$ along $\be_1^{r_1}$.
 Since all the other ends of $u$ are matched by planar components in the symplectization, we may combine
 the indices of planar components calculated using  \eqref{cpct1}
 with the  formula \eqref{cpct4} for the multiply covered curve and use the fact that
 $\ind(\widetilde{u})\ge 0$
 to obtain
\begin{align}
\tfrac{1}{2} \mathrm{ind}(C) & \ge
(k-1)(1-m) - k(n_3+n_4-m\widetilde{n}_3-m\widetilde{n}_4) \\
& \qquad  - \sum_{i=1}^{n_3} \lfloor \frac{r_i}{x} \rfloor - \sum_{j=1}^{n_4} \lfloor s_jx \rfloor
 + m\sum_{i=1}^{\widetilde{n}_3} \lfloor \frac{\widetilde{r}_i}{x} \rfloor + m\sum_{j=1}^{\widetilde{n}_4} \lfloor \widetilde{s}_jx \rfloor \nonumber\\
&\qquad\quad    + (n_3-1)k + \sum_{i=2}^{n_3}(r_i+\lfloor \frac{r_i}{x} \rfloor) + kn_4 + \sum_{j=1}^{n_4} (s_j +\lfloor s_jx \rfloor) \nonumber \\
& =(k-1)(1-m) - k + km(\widetilde{n}_3 +\widetilde{n}_4) \nonumber \\
&\qquad + \sum_{i=2}^{n_3} r_i - \lfloor \frac{r_1}{x} \rfloor   + \sum_{j=1}^{n_4} s_j + m\sum_{i=1}^{\widetilde{n}_3} \lfloor \frac{\widetilde{r}_i}{x} \rfloor + m\sum_{j=1}^{\widetilde{n}_4} \lfloor \widetilde{s}_j(x) \rfloor. \nonumber
\end{align}
Since $\widetilde{n}_3 +\widetilde{n}_4 \ge 1$, we see that
$$
\tfrac{1}{2} \mathrm{index}(C)\; \ge\; (m-1) +  \sum_{i=2}^{n_3} r_i +  m\sum_{i=1}^{\widetilde{n}_3} \lfloor \frac{\widetilde{r}_i}{x} \rfloor - \lfloor \frac{r_1}{x} \rfloor \; \ge  \;  \sum_{i=2}^{n_3} r_i  +
m \sum_{i=1}^{\widetilde{n}_3} \lceil \frac{\widetilde{r}_i}{x} \rceil  -  \lceil \frac{r_1}{x} \rceil.
$$
Now the end, say  $\beta_1^{\widetilde{r}_1}$, of $\widetilde{u}$ that is covered by the end of $u$ asymptotic to $\beta_1^{r_1}$ satisfies
$m \widetilde{r}_1 \ge r_1$.
Therefore
$$
m \sum_{i=1}^{\widetilde{n}_3} \lceil \frac{\widetilde{r}_i}{x} \rceil \ge  \lceil m \sum_{i=1}^{\widetilde{n}_3} \frac{\widetilde{r}_i}{x} \rceil \ge \lceil \frac{r_1}{x} \rceil
$$
which implies that $ \mathrm{index}(C)\ge 0$.
Moreover, $ \mathrm{ind}(C)>0$ unless $\widetilde{n}_3 = n_3 = 1, \widetilde{n}_4=0$ and $r_1 = m \widetilde{r}_1$, i.e. unless  $u$ has just one negative end of multiplicity $r_1$, in which case it is the unique curve in $C$.

If $u$ is attached to $C_*$ along $\beta_2$, then the same argument shows that $\op{ind}(C) \ge 0$, with equality only if $C$ has one negative end.
\end{proof}

This completes the proof of compactness in Case 1.
\MS

\NI  {\bf Case 2:}   {\it  $C_*$ has index $0$.}
In this case, Lemma~\ref{negcurves} shows that $C_*$ is a multiple cover of $\be_1$ with just one positive end.  Then,  if the limit $C_{\infty}$ is nontrivial, there must be other nontrivial  curves in the symplectization, which implies that the top of  end of $C_*$ must attach to a curve
 $\im(u)$ with more than one negative end.   But by Lemma~\ref{lem:cpct} a component $C$ containing such $\im(u)$  must have positive index.
Since again no component has negative index,  this scenario is impossible.

This completes the proof of the compactness argument.
 \MS

 \NI

\NI {\bf The counting argument.}
This is proved much as in \cite[\S3]{CGH}.
There are two steps here.
We identify $\ov M$ with the completion of the complement of $\im (\widetilde{\Phi}_0)$ so that it supports a $\T^k$ action, and consider
the space  $\Jj^{\T^k}_{reg}$   of all  $\T^k$-invariant and admissible almost complex structures
on $\ov M$ for which all  somewhere finite action curves in both $\ov X$ and $\ov M$ are regular.  One shows first that
\begin{itemize}\item  $\Jj^{\T^k}_{reg}$ is nonempty; and second that
\item when $J\in \Jj^{\T^k}_{reg}$
the count of $J$-holomorphic curves
in
$\Mm_{J}(\ov M, dL, \be_1^{ p})$   is nonzero.
\end{itemize}
The second step (Proposition~10 in \cite{CGH}) is proved as in \cite[\S3.1]{CGH}.  The idea is this:  since $J$ is
$\T^k$-invariant  and the elements of
$\Mm_{J}(\ov M, dL,  \be_1^{p})$
have index zero, the elements of this moduli space must lie in the $4$-dimensional manifold $\ov X$, so that one can appeal to Wendl's automatic transversality results.  In our case this argument is slightly easier than in \cite{CGH} since our curves
have no positive ends.

To establish the first step (Proposition~11 in \cite{CGH}), one first notes that it is immediate provided that there is $J\in \Jj^{\T^k}_{reg}$
such all elements in $\Mm_{J}(\ov M, dL, \be_1^p)$ that do not lie entirely  in $\ov X$ are {\it orbitally simple} (i.e. intersect at least one $\T^k$ orbit exactly once transversally), since then standard methods allow one to find an $\T^k$-invariant and regular perturbation of $J$.    To show that there is a suitable $J$
one considers a second neck stretching as in \cite[\S3.2]{CGH}, this time  along a product surface
\[\Si = \Phi\bigl(\p (1+\de) E(1,x)\bigr)  \times \C P^k(2T)\]
in $\ov M$. (Thus, one extends the initial embedding $\Phi:E(1,x)\to \C P^2(\mu_*)$ to a slightly larger  ellipsoid, extends it  trivially to the product, and then
 stretches
  by an amount $K$
 along the corresponding product boundary $\Si$.)  We consider $\T^k$-invariant
 almost complex structures $J^K$ on $\ov M$
 that are products both near and outside the region bounded by $\Si$, so that when one stretches the neck the top level is a product that we denote ${\ov X}\,\!' \times \C P^k(2T)$ while the
 rest  is a (possibly multi-level) cobordism
 from $\Si$ to the ellipsoid  $\p \Ee': = \widetilde\Phi(\p E(1,x,S,\dots,S))$.

If for some $K$ all $J^K$-holomorphic curves  $C^K$ in $\Mm_{J^K}(\ov M, dL,  \be_1^p)$ are orbitally simple, then we are done.
 Hence we only need to consider the case when there is a sequence of non-orbitally simple curves $C^K$ for $K\to \infty$.
In this case  there is a limiting building $C_\infty$,  whose top
level $C^{top}$ lies in  ${\ov X}\,\!'  \times \C P^k(2T)$.
Consider the projection $C'$  of $C^{top}$  to the $4$-dimensional space
 ${\ov X}\,\!' $.
By construction, $C'$ is the limit of the projection to  ${\ov X}\,\!' $ of (pieces of) non-orbitally simple curves, and so has at least one multiply covered component.
We must show that this is impossible.  The argument used to prove this in  \cite[Prop.~12]{CGH}  does not generalize since it exploits the fact that in their case
$C'$ has essentially zero action.   However,  it is possible to prove this in our more general situation by using Lemma~\ref{lem:cpct}.

 The projected curve $C'$ cannot consist of a single component with end on the orbit set $\lbrace (\beta_1, p) \rbrace$
which is an $m$-fold cover for $m > 1$, because then it would have to have degree $d$ and we are assuming that $\gcd(p, d) = 1$.

Therefore the limiting building $C_\infty$ is nontrivial, i.e. it cannot consist
 just of an index zero cylindrical cover of $\be_1$ (in the cobordism from $\Si$ to the negative end of $\ov M$) together with a single component in   ${\ov X}\,\!' \times \C P^k(2T)$.
We now argue much as
in the compactness argument.

 Note first that the index arguments used above can
 be adapted    essentially without change.
  Indeed, although $\Si \cong \p \bigl((1+\de) E(1,x)\bigr)  \times \C P^k(2T)$ is different from $\p E(1, x,S,\dots,S)$, the index formulas for curves positively asymptotic to the orbits $\be_1,\be_2$ on $\Si$ are the same as they are for  $\p E(1, x,S,\dots,S)$, as one sees by comparing the formulas in Propositions~13 and 17 in \cite{CGH}.  Similarly, the contribution to the index of negative ends on $\Si$ is just as in \eqref{cpct3}, except that
 the index has an additional positive contribution of $2k(n_3 + n_4)$ (see Proposition~15 in \cite{CGH}) which takes into account the fact that the negative ends lie on the product $\Si$ so that each end lies in a $2k$-dimensional family.  However, this additional contribution is
 cancelled out
  by the fact that when we match ends in this Morse-Bott situation we must subtract $2k$.  Therefore we may calculate the indices of matched components just as before.  In particular, if we define  $C_*$ to be the component of $C^{lower}$ with bottom end on   $\be_1^{p}$, then Lemma~\ref{negcurves} and   Lemma~\ref{lem:cpct} both hold.
   Therefore all the  components of $C_\infty\less C_*$ have nonnegative index, and have positive index if they have more
   than one negative end.
In particular, if the building is nontrivial,  either $C_*$ or
  some component of $C_\infty\less C_*$  has positive index.  But this is impossible.

 Together, these two steps complete the proof of Proposition~\ref{prop:goodC1}.
 \end{proof}

\begin{rmk}\label{rmk:counterex} (i)  One might try to generalize  Proposition~\ref{prop:goodC1} by considering curves of genus zero $C$
with one negative end on $\be_1^m$ where $m$ is chosen so that the index is zero.  Thus, if $C$ has degree $d$ we assume
$\gcd(d,m) = 1$ and $3d = m + \lceil mx\rceil$; see \eqref{eq:indC}.    However, in this case Lemma~\ref{negcurves} might fail, so that  compactness does not hold.
For example, if one can decompose $d = d_1 + d_2$ and  $m = m_1 + m_2$ in such a way that
$m_i + \lceil {m_i}x\rceil = 3d_i$ for $i= 1,2$, then a curve in $\ov M$  of degree $d$ and one negative end on $\be_1^{m}$
might split into a nontrivial building whose top has two components of degrees $d_1,d_2$ and bottom (in the symplectization) has two positive ends of multiplicities $m_1,m_2$.
In such a case, we would have  $ \lceil \frac{m_1}x\rceil +  \lceil \frac{m_2}x\rceil  =  \lceil \frac{m_1+m_2}x\rceil $ so that  the curve in the symplectization as well as the two components in $\ov M$ all have index zero.
As an example, take $x = \frac {55}8$, $m_1 = 5, m_2 = 18$  and $d_1= 2$ and $d_2 = 7$.

\MS
\NI (ii)  We claim that the statement of Proposition~\ref{prop:goodC1} also holds  if there is a  curve  in the moduli space
$\Mm(\ov X_{\mu_*,x}, dL, \be_1^{3p})$ where  $x = b+\eps = \frac pq + \eps$, provided we assume that each of  the pairs $(p,q)$, $(d,3p)$ and $(3,p+q)$ is  mutually prime.    By  \eqref{eq:indC},
the index condition then implies that  $d = p+q$ so that we again get the sharp bound $c_k(b)\ge \frac {3d}p = \frac{3b}{b+1}$.

To prove the claim we must first establish compactness. First of all, any curve in this moduli space cannot be multiply covered, since $d$ and $3p$ are coprime, and similarly convergence to a multiply covered curve in $\overline{M}$ can be excluded.  For the rest of the argument, the key ingredients Lemmas \ref{lem:largeS} and \ref{lem:cpct} hold as before, but in the case of equality in Lemma \ref{negcurves} we can only conclude that $C_*$ is a branched cover of the trivial cylinder over $\beta_1$, with a single negative end of multiplicity $p$ and positive ends which have multiplicities $r_1p,\dots, r_{n_1}p$ where $\sum_{i=1}^{n_1} r_i = 3$. But in the case when $C_*$ is branched, that is when $n_1>1$, the remainder of the limiting building consists of $n_1>1$ matched planar components of degrees $d_1,\dots d_{n_1}$, each attached to one of the top ends of $C_*$.  By Lemma~\ref{lem:cpct}, each such component has nonnegative index, and hence all must have index $0$.  But if $n_1>1$ then at least one $r_i = 1$, and the corresponding degree $d_i$  must then satisfy  $3d_i = p+q$, contradicting our assumption that $\gcd(3,p+q)= 1$.  The compactness then follows as before. A similar argument applies to the building in the counting argument and this proves the claim.

More generally, this argument proves the following.

\begin{prop}    Suppose that for some triple $(d,x,m)$ there is a genus zero curve in $\ov X_{\mu_*,x}$ of degree $d$, index zero, and with one negative end on $\be_1^m.$   Suppose
 further that there are no decompositions
$d = \sum_{i=1}^n d_i, m = \sum_{1=1}^n m_i$, with $n>1$ 
and $d_i>0$ for all $i$, and such that
$$
3d_i = m_i + \lceil\frac {m_i}x\rceil \qquad \forall i.
$$
 Then $c_k(x)\ge \frac dm$.
\end{prop}

To prove compactness here, we first observe that the hypothesis excludes
the possibility  that the moduli space $\Mm(\ov X_{\mu_*,x}, dL, \be_1^m)$ contains an $n$-cover of an index $0$ curve for some $n$,
because such a cover would give rise to a decomposition with $m_i = \frac mn$ and $d_i = \frac{d}{n}$ for all $i$. 
Further, in this general situation,  equality in Lemma \ref{negcurves} only implies that $C_*$ is a branched cover of the trivial cylinder, with no restrictions on the positive ends except that their multiplicities $m_i$ must satisfy the condition
$3d = \sum _i(m_i + \lceil\frac {m_i}x\rceil)$.  However Lemma~\ref{lem:cpct} again implies that the remaining planar components have index $0$, and our hypotheses precisely exclude this if $C_*$ has multiple positive ends.

Notice that 
in such a situation 
the  inequality  $c_k(x)\ge \frac dm$ is in general  no longer sharp.\MS

\NI (iii)  Proposition~\ref{prop:goodC1} implies that we can prove that  $c_k(m) = \frac{3m}{m+1}$ for an integer of the form $3d-1$ by finding a  curve of degree $d$ and single end on  $\be_1^{3d-1}$ on the ellipsoid $\p E(1,3d-1)$.  However, to find suitable obstructions at the other integers  we need to use  the generalizations  in (ii) above.    For example, to show that
$c_k(7) = \frac{21}{8}$ it would suffice to find a degree $8$ curve in the completion of $\C P^2(\mu) \less \Phi(E(1,7+\eps))$ with one end on $\be_1^{21}$.  Constructing such curves is the subject of ongoing work.
\end{rmk}

\section{The case
when $C_L$ has  connectors.} \label{sec:heart}

We now assume that 
the limiting building $C_\infty$ has connectors (see Definition~\ref{def:conn}), and denote by $D_{12}$ the union of all the matched components of $C_L$ that contain a connector.
This section is devoted to showing that there is at most one representative of $B$ that is close to breaking into a building with $D_{12}$ nonempty.   This result is stated  in \S\ref{ss:rest}  as Proposition~\ref{prop:up}; its proof is at the end of  \S\ref{ss:rest}.

We recall that if $C_\infty$ has a connector then $C_U$ has negative end $\lbrace (\beta_1, \ell_{n}), (\beta_2, \ell_{n}) \rbrace$, and moreover, by Proposition \ref{prop:lowenergy}, any curve in the upper level has negative end on the orbit set $\lbrace(\beta_1,m),(\beta_2,m)\rbrace$ for some positive integer $m$.  It is relatively easy to find a model candidate $C_M$  for  
the connector (see Lemma~\ref{clm:homologyclass}), and to show that there is at most one $B$-curve that could  limit on a building with connector $C_M$ (see \S\ref{ss:rest}). 
What is difficult is to show that there is no other possible breaking with a connector.

 To this end, we begin by formulating in Proposition~\ref{prop:convertingtheproblemtoa} a fundamental inequality that translates the information on  $I$, $\ind$ and writhe (which are all integers)  that is contained in the index inequality~\eqref{eqn:indineq}   into numerical terms that are easier to calculate and manipulate.  For example, the quantity $A(C)$
in ~\eqref{eqn:indineq} that is a count of lattice points is converted via Pick's formula into a sum of terms  $A(\theta, s)$ that are areas.   This formula uses the fact the constraints on $C_L$ as well as the symplectic form come from related weight expansions and so  have simple numerical descriptions.
In \S\ref{ss:area}, we use the special arithmetic properties of the numbers $b_n$ to gain 
information on  the relevant areas $A(\theta, s)$.

The next step is to use the fact that, because the total action of $C_L$ is (approximately) $\frac 1{Q_n}$, all but one of its constituent curves have low action and so are \lq light'; the connected curve with nontrivial action is called {\bf heavy}.  The heavy curve cannot be multiply covered,
since  all connected curves in $C_L$ have action that is (approximately equal to) some multiple of $\frac 1{Q_n}$, so that its properties  can be analyzed using the machinery developed in \S\ref{ss:fund} and \S\ref{ss:area}.
 In \S\ref{ss:light}  we prove basic facts about light and heavy curves, including the  fact that the heavy curve is in the lowest level and must go through all the constraints on the last block.   Proposition~\ref{prop:improva} then shows that the heavy curve must
 be a connector with  
 the same constraints and asymptotics as the model $C_M$.
 The rather elaborate proof  compares the ECH and Fredholm indices, again using
 Proposition~\ref{prop:convertingtheproblemtoa} and the arithmetic properties of the numbers $b_n$.  The argument is completed in \S\ref{ss:rest}.

\subsection{The fundamental estimate}\label{ss:fund}

The basic estimate needed for the proof of Proposition~\ref{prop:up}  
is given by Proposition~\ref{prop:convertingtheproblemtoa} below.  To state it, we need to introduce some notation.

\begin{defn}\label{def:Atheta}
{\it Given a pair $(\theta,t)$, where $\theta$ is an irrational number and $t$ is a positive integer, we define $A(\theta,t)$ to be {\bf the area of the region} in the first quadrant formed by the line $y = \theta x$, 
the vertical line from $(t,\lfloor t \theta \rfloor)$ to $(t, t \theta)$, and
the maximal concave lattice path starting at the origin, ending at $(t, \lfloor t \theta \rfloor)$, and staying below the  line $y = \theta x$.  }
\end{defn}

The significance of $A(\theta,t)$ to our problem is given by the following.

\begin{lemma}\label{lem:fundest}
Let $\gamma$ be an elliptic orbit with monodromy angle $\theta$, and let $r$ be the length of $p^+_{\theta}(t)$.  Then
\begin{equation}
\label{eqn:agrading}
2A(\theta,t) = \theta t^2 - gr(\gamma^t) + t + \lfloor t \theta \rfloor + r.
\end{equation}
\end{lemma}

\begin{proof}
Let $M(\theta,t)$ be the area underneath the maximal concave lattice path.  Then
\begin{equation}
\label{eqn:aequation}
A(\theta,t) = \tfrac{1}{2} \theta t^2 - M(\theta,t).
\end{equation}
We can compute $M(\theta,t)$ by using Pick's theorem.  It is the area of a region $R$ with $\frac{gr(\gamma^t)}{2} + 1$ total lattice points and $r + t + \lfloor t \theta \rfloor$ lattice points on the boundary.  Hence by Pick's theorem (cf. \eqref{eqn:picks}), we have
\begin{equation}
\label{eqn:grM}
2M(\theta,t) = gr(\gamma^t) - t - r - \lfloor t \theta \rfloor.
\end{equation}
The lemma follows from this together  with \eqref{eqn:aequation}.
\end{proof}

Given a preglued holomorphic building $C$, we now define a vector $\op{diff}_C$ that will be important in our estimates.  To motivate its definition, recall that for any vectors $z, w$
the quantity $z \cdot z$ is minimized subject to the constraint
\begin{equation}
\label{eqn:energyconstraint}
z \cdot w = \ka
\end{equation}
if $z = \lambda w$, where
\begin{equation*}
\lambda = \frac{\ka}{w \cdot w}.
\end{equation*}
Now let $[C]$ be any homology class in $H_2(\widehat{\Ee}_n,\alpha,\emptyset)$ corresponding to a preglued holomorphic building $C$.
As explained in \S\ref{sec:noconn},\footnote{Our notation is such that this vector does {\em not} include the asymptotics of $[C]$; it just records the coefficients along the exceptional classes.}
we can identify  the class $[C]$ with a vector $z$.  Then if $w =  w(b_n)$, the weight vector of $b_n$, the symplectic area of the constraints is given by $z \cdot w$, modulo an arbitrarily small error caused by the fact that we cannot completely fill the ellipsoid by balls.
Now consider the vector
\begin{equation}
\label{eqn:diffdefn}
\op{diff}_C \eqdef \lambda w - z,\qquad\mbox{where }  \lambda: = \frac{\ka}{w \cdot w}, \quad \ka: = z \cdot w.
\end{equation}
Then
$ w \cdot \op{diff}_C = 0$,
so that
\begin{equation}
\label{eqn:diffcontribution}
z\cdot z = (\lambda w - \op{diff}_C)\cdot (\lambda w - \op{diff}_C)  
 = \lambda^2 w \cdot w + \op{diff}_C  \cdot \op{diff}_C.
\end{equation}

Since some of the quantities in our estimates are exact, while others are only approximate it will be convenient to introduce the following notation.

\begin{defn}\label{def:approx}   {\it If $A_1(\eps), A_2(\eps)$ are quantities that depend on a finite number of arbitrarily small constants $\eps_i>0$,  
then we write
$$
A_1(\eps)\le _\de  A_2(\eps)
$$
if for all $\de>0$ we have $A_1(\eps)- A_2(\eps)< \de$ for all sufficiently small $\eps, \eps'$.   Further, we write
$$
A_1(\eps)>_\de  A_2(\eps)
$$
if there is $\de_0>0$ so that $A_1(\eps) - A_2(\eps)>  \de_0$  for all sufficiently small $\eps_i.$
    Further we write $A_1(\eps)=_\eps  A_2(\eps)$  if $A_1, A_2$ are continuous functions of the small parameters $\eps_i$ that are equal  when all $\eps_i = 0$.   }
\end{defn}
Note that the properties  $\le _\de$ and $>\de$ are mutually exclusive; that is, it is impossible that $A_1(\eps)\le _\de  A_2(\eps)$ and also $A_1(\eps)>_\de  A_2(\eps)$.

We can now state the crucial estimates.  To simplify the notation for what will follow, define
\begin{align}\label{eq:theta}  \theta_n \eqdef b_n + \epsilon_n, &\quad \quad \widetilde{\theta}_n \eqdef 1/\theta_n,
\end{align}
where $\epsilon_n$ is small and irrational.
We will apply this result when $s,t\le \ell_n$, so that the quantity
$\frac s{P_n} + \frac t{Q_n}$ on the RHS of
\eqref{eqn:weakerestimate} is at most $\ell_n(\frac 1{P_n} + \frac 1{Q_n})$, which by \eqref{eqn:lnpnlimit}  is a decreasing sequence
that converges to $1$.  Thus this RHS is approximately  $3$.

\begin{proposition}
\label{prop:convertingtheproblemtoa}
Let $C$ be a connected somewhere injective curve in $\widehat{\Ee}_n$, asymptotic to $\lbrace (\beta_1,s), (\beta_2,t) \rbrace$.  Then:
\begin{itemize}
\item If $C$ has low action, we have
\begin{equation}
\label{eqn:strongestimate}
2A(\widetilde{\theta}_n,s) + 2A(\theta_n,t) + \op{diff}_C \cdot \op{diff}_C \le_\de 1.
\end{equation}
\item Otherwise,
\begin{equation}
\label{eqn:weakerestimate}
2 A(\widetilde{\theta}_n,s) + 2A(\theta_n,t) + \op{diff}_C \cdot \op{diff}_C \le_\de 1 +  2(\frac s{P_n} + \frac t{Q_n}).
\end{equation}
\end{itemize}
\end{proposition}

\begin{proof}

Let $z=[C]$.  We prove Proposition~\ref{prop:convertingtheproblemtoa} in several steps.

\vspace{3 mm}

\NI
{\bf Step 1}: {\it Applying the (improved) index inequality}

\vspace{3 mm}

By Proposition~\ref{prop:indineq}, we have
\begin{equation}
\label{eqn:bound}
I(C) - \tfrac12 \op{ind}(C) \ge I(C) - \op{ind}(C) \ge 2A(C)
\end{equation}
where $A(C)$ is a certain count of lattice points.
We also have
\begin{equation}
\label{eqn:indexequation}
\tfrac 12 \op{ind}(C) = -1 + n_1 + n_2 + \sum_{i = 1}^{n_1} \lfloor s_i \widetilde{\theta}_n  \rfloor + \sum_{j = 1}^{n_2} \lfloor t_j \theta_n \rfloor + s + t - z \cdot 1,
\end{equation}
where $(s_1,\ldots,s_{n_1})$ is the partition of $s$ given by the ends of $C$, and $(t_1,\ldots,t_{n_2})$ is the partition of $t$.  In addition, by \eqref{eq:gr10} and \eqref{eqn:Eei} we have
\begin{equation}
\label{eqn:Iequation}
I(C) = 2st + gr(\beta_1^s) + gr(\beta_2^t) - z \cdot z - z \cdot 1.
\end{equation}
We can substitute for $gr(\beta_1^s)$ and $gr(\beta_2^t)$ in \eqref{eqn:Iequation} using \eqref{eqn:agrading} to get
\begin{equation}
\label{eqn:improvedIequation}
I(C) = -2A(\widetilde{\theta}_n,s) -2 A(\theta_n,t) + s +r_1 + \lfloor s \widetilde{\theta}_n \rfloor + t + r_2 + \lfloor t \theta_n \rfloor +  (\theta_n t^2 + \widetilde{\theta}_n  s^2 + 2st - z \cdot z) - z \cdot 1,
\end{equation}
where $r_1, r_2$ are the number of ends in the ECH partitions.
\vspace{3 mm}

\NI
{\bf Step 2: } {\it  Estimates}

\vspace{3 mm}

Let us first suppose that $C$ has  symplectic area (approximately)  $\frac 1{Q_n}$.  Then because the orbit $\be_1$ has action $1$, while $\be_2$ has action $\theta_n =_\eps \frac{P_n}{Q_n}$,
we have
$$
\om(C)\; =_\eps  \; s + t  \theta_n- z\cdot w \; = _\eps\; \frac 1 {Q_n}.
$$
so that by \eqref{eqn:diffdefn} we have
$$
\ka = z\cdot w   \;=_\eps\;  s + t  \theta_n - \frac 1 {Q_n}.
$$
Since $w\cdot w = \frac{P_n}{Q_n} =_\eps \theta_n$ we obtain from \eqref{eqn:diffcontribution} that
\begin{align*}
\la = & \frac {\ka}{w\cdot w} =_\eps\;  s\frac{Q_n}{P_n}  + t   - \frac 1 {P_n},\vspace{.2in} \\ \notag
z \cdot z  & \ge_\de
\theta_n   t^2 + \widetilde{\theta}_n  s^2 +
  2st - \frac{2}{Q_n}(s \widetilde{\theta}_n+ t) + \op{diff}_C \cdot \op{diff}_C.
\end{align*}
Moreover, we can improve this to
\begin{equation*}
z \cdot z\; \ge_\de\;  \theta_n t^2 + \widetilde{\theta}_n s^2 + 2st + \op{diff}_C \cdot \op{diff}_C
\end{equation*}
in the case where $C$ has low action.

Substitute this into \eqref{eqn:improvedIequation} to get
\begin{equation}
\label{eqn:equationwithbound}
I(C)\; \ge_\de\; 2A(\widetilde{\theta}_n,s) +2 A(\theta_n,t) + s +r_1 + \lfloor s \widetilde{\theta}_n \rfloor + t + r_2 + \lfloor t \theta_n \rfloor - z \cdot 1 + \frac{2}{Q_n}(s\widetilde{\theta}_n + t) - \op{diff}_C \cdot \op{diff}_C.
\end{equation}
with the improvement to
\begin{equation*}
I(C) \ge_\de 2A(\widetilde{\theta}_n,s) + 2A(\theta_n,t) + s +r_1 + \lfloor s \widetilde{\theta}_n \rfloor + t + r_2 + \lfloor t \theta_n \rfloor - z \cdot 1 - \op{diff}_C \cdot \op{diff}_C.
\end{equation*}
in the low action case.  Now substitute \eqref{eqn:equationwithbound} and \eqref{eqn:indexequation} into \eqref{eqn:bound}.  This gives
\begin{align}
& 1 - 2A(\widetilde{\theta}_n,s) -2 A(\theta_n,t) - \op{diff}_C \cdot \op{diff}_C + \frac{2}{Q_n}(s\widetilde{\theta}_n + t)  + (r_1 - n_1) + (r_2 - n_2)\nonumber \\
&\qquad\qquad \qquad  + ( \lfloor s\widetilde{\theta}_n \rfloor - \sum_{i=1}^{n_1} \lfloor s_i \widetilde{\theta}_n \rfloor) + ( \lfloor t \theta_n \rfloor - \sum_{i=1}^{n_2} \lfloor t_j \theta_n\rfloor) \;\; \ge_\de \;\; 2A(C), \label{eqn:bound2}
\end{align}
with the improvement to
\begin{align}
& 1 - 2A(\widetilde{\theta}_n,s)  -2A(\theta_n,t) - \op{diff}_C \cdot \op{diff}_C  + (r_1 - n_1) + (r_2 - n_2)  \nonumber \\ &
\qquad\qquad \qquad + ( \lfloor s \widetilde{\theta}_n \rfloor - \sum_{i=1}^{n_1} \lfloor s_i \widetilde{\theta}_n \rfloor) + ( \lfloor t \theta_n \rfloor - \sum_{i=1}^{n_2} \lfloor t_j \theta_n\rfloor) \;\; \ge_\de \;\; 2A(C), \label{eqn:bound3}
\end{align}
in the low action case.

\vspace{3 mm}

\NI
{\bf Step 3:} {\em  We prove}
\begin{align}
\label{eqn:ACbound}
2A(C) & \ge \max(r_1 - n_1,0) + \max(r_2 - n_2,0) + \\ \notag
&\qquad \qquad ( \lfloor s \widetilde{\theta}_n \rfloor - \sum_{i=1}^{n_1} \lfloor s_i \widetilde{\theta}_n \rfloor) + ( \lfloor t \theta_n \rfloor - \sum_{j=1}^{n_2} \lfloor t_j \theta_n\rfloor).
\end{align}
\begin{proof}
Recall from \eqref{eq:Adef1} that
$$
A(C) = A_C(\beta_1,s) + A_C(\beta_2,t) = 2\Ll(\La_C) + b(\La_C)
$$
 is a certain count of lattice points.
Consider  $A_C(\beta_1,s)$.  Denote the concave path determined by the ends of $C$ at $\beta_1$ by $\Lambda_C$, and let $\Lambda$ be the path determined by the partition conditions.  Counting in the vertical line $x = s$ gives $(\lfloor s \widetilde{\theta}_n \rfloor - \sum_{i=1}^{n_1} \lfloor s_i \widetilde{\theta}_n \rfloor )$ lattice points that contribute to $\mathcal{L}(\Lambda_C)$.

 Further, if any part of the paths $\La$ and $\La_C$ are geometrically the same (though perhaps with different subdivisions), then the maximality of the ECH path implies that it has at least as many vertices as $\La_C$, so that any extra  vertices on this part of $\La$ are interior lattice points that contribute to the term $\frac 12 b(\La_C)$ in \eqref{eq:Adef}.
On the other hand any vertex in $\La$ that does not lie on $\La_C$ contributes to the term $\Ll(\La_C)$.
 Therefore  the vertices of $\La$ that do not lie on  the  line $x = s$ and are not vertices of $\La_C$ contribute at least
  $\max(r_1 - n_1,0)$ to $2A(C)$.
  Combining this with the analogous analysis for $\beta_2$ gives \eqref{eqn:ACbound}.
\end{proof}

\vspace{3 mm}

\NI
{\em Step 4:  Completing the proof}

\vspace{3 mm}

Combine \eqref{eqn:ACbound} with \eqref{eqn:bound2} to get
\[
 1 - 2A(\widetilde{\theta}_n,s)  -2 A(\theta_n,t) + \frac{2}{Q_n}(s \widetilde{\theta}_n + t) - \op{diff}_C \cdot \op{diff}_C \ge_\de 0.\]
In the low action case, combine \eqref{eqn:ACbound} with \eqref{eqn:bound3}.  This gives
\[ 1 - 2A(\widetilde{\theta}_n,s) -2A(\theta_n,t) - \op{diff}_C \cdot \op{diff}_C \ge_\de 0.
\]
Since $\frac{1}{Q_n}(s \widetilde{\theta}_n + t)  =_\eps \frac s{P_n} + \frac t{Q_n}$, this
 proves Proposition~\ref{prop:convertingtheproblemtoa}.
 \end{proof}

\begin{remark}  Although we will not use this in the current paper, we note here that
the bounds in Proposition~\ref{prop:convertingtheproblemtoa} can  be improved if any segment of the concave paths $\La_C$ at the ends of $C$
does not lie on a maximal concave path.  Specifically, we can define quantities
$A_C(\tilde{\theta}_n,s), A_C(\theta_n,t)$ analogously to $A(\tilde{\theta}_n,s), A(\theta_n,s)$, but using the concave path $\Lambda_C$ formed from the ends of $C$ instead. By maximality of the partition path, we always have $A_C(\tilde{\theta}_n,s) \ge A(\tilde{\theta}_n,s), A_C(\theta_n,t) \ge A(\theta_n,t)$; moreover, the proof of Proposition~\ref{prop:convertingtheproblemtoa} shows that as in \eqref{eqn:strongestimate}
\[ 2A_C(\tilde{\theta}_n,s) + 2 A_C (\theta_n,t) + \op{diff}_C \cdot \op{diff}_C \le_{\delta} 1\]
for low action curves, and similarly for \eqref{eqn:weakerestimate}.  To see this,
assume first that at the end $\be_2$ the paths $\La$ and $\La_C$ have no common segments, and let $i$ be the number of lattice points lying strictly between them.
Then
$$
2  A_C (\be_2,t) = 2 r_2 - 2 + 2\bigl(\lfloor t \theta \rfloor - \sum_j \lfloor t_j \theta \rfloor \bigr) + 2i + b_\La,
$$
while Pick's Theorem gives
\begin{align*}
2  A_C (\theta_n,t) - 2 A(\theta_n,t) & = r_2 + n_2 + b + \bigl(\lfloor t \theta \rfloor - \sum_j \lfloor t_j \theta \rfloor \bigr) + 2 i - 2\\ \notag
& = 2A_C (\be_2,t) - r_2 + n_2 - \bigl(\lfloor t \theta \rfloor - \sum_j \lfloor t_j \theta \rfloor \bigr).
\end{align*}
This equality still holds if  the paths $\La$ and $\La_C$ do have  common segments.  Indeed, in this case
 the convexity condition implies that these occur at the beginning of the paths.  By additivity, it therefore suffices to consider the case when the two paths are geometrically the same.   But in this case the left hand side is clearly zero, while  the right hand side
 also vanishes because $2A_C (\be_2,t) = b_\La = r_2 - n_2$
by the maximality condition on ECH partitions.
Now substitute  this, together with the analogous identity for $\tilde{\theta}_n$, in the inequalities \eqref{eqn:bound3}  and  \eqref{eqn:bound2} to obtain the strengthened  versions of \eqref{eqn:strongestimate} and \eqref{eqn:weakerestimate}.
\end{remark}

\subsection{Area estimates}\label{ss:area}

In order to understand the asymptotics of the connector, we now
establish the following estimates for the area $A(\theta,t)$ defined in \eqref{eqn:aequation}.  The connector has top on the orbit set
$\{(\be_1,s)\}, \{(\be_2,t)\}$ where $0< s,t < \ell_n$.  Proposition~\ref{prop:convertingtheproblemtoa} shows that the areas $A(\theta,\cdot)$ at these ends must be rather small.  As we explain in more detail in Lemma~\ref{lem:partcond} below, these areas are closely related to the partition conditions and hence to best lower approximations to $\theta$.
We saw in Example~\ref{ex:bigmono} that the lower convergents to $\theta_n$ have denominators  $t_k, 0\le k\le n-1$.  Further
the best approximations for $\theta_n$ from below whose denominator  $t$ satisfies $\ell_{n-1} < t < \ell_n$
are given by the semiconvergents $c_{2n-2}\oplus {\bf r} c_{2n-1} = [6,(1,5)^{n-1}, 1,r]$ for $1\le r< 6$.  These have
denominators $t_{n-1} + r\ell_{n-1}$.    In particular $t_n = t_{n-1} + 5\ell_{n-1}$, while $\ell_n= t_{n-1} + 6\ell_{n-1}$.

  The following proposition summarizes the
results we shall need when $n>1$.   (For the case $n=1$, see Example~\ref{ex:n=1}.)  Recall the notation $\le_\de, <_\de$ from Definition~\ref{def:approx}.

\begin{proposition}
\label{prop:estimates}  Let $n>1$.
\begin{itemize}
\item[{\rm (i)}] 
We have $2A(\theta_n,t) >_\de \frac{5}{ \tau^4}> 0.729$ for all $t < \ell_n$.
\item[{\rm (ii)}] If $2A(\theta_n,t) < 2.67$ and $\ell_n > t > t_n$, then $t = t_n + t_k$ for some $k < n$.  In this case, $2A(\theta_n,t) \ge_\de 1.39 + 
\frac 5{\tau^4} > 2.11$.
\item[{\rm (iii)}] $2A(\theta_n,t_n) \ge 1.39$.
\item[{\rm (iv)}] $2A(\theta_n,t_n + t_{n-1}) \ge_\de 2.52$  \item[{\rm (v)}] If  $t_n - 2 \ell_{n-1} <  t < t_n$ and $t \ne t_n - \ell_{n-1}$, then $2A(\theta_n,t) \ge_\de 1.39 + \frac{5}{ \tau^4}> 2.11.$
\item[{\rm (vi)}] $2A(\theta_n,t_n - \ell_{n-1}) \ge_\de 1.39$.
\item[{\rm (vii)}] 
$2A(\widetilde{\theta}_n, s)  >_\de \frac 7{48}$ for all $1 \le s < \ell_n$.
\item[{\rm (viii)}]  $2A(\widetilde{\theta}_n, \ell_{n-1})  = \frac{8\ell_{n-1}}{P_n} <_\de \frac{8\si}{\tau^4}$.
\item[{\rm (xi)}] If $2A(\widetilde{\theta}_n, s) < \frac {7}{24}$ for some $1 \le s < \ell_n$ then $s = \ell_{k}$ for some $1 \le k < n.$

\end{itemize}
\end{proposition}

We prove the proposition in several steps.
As a first step, we investigate  the relationship of the area $A(\theta, t)$ with the partition conditions.

\begin{lemma}
\label{lem:partcond}
Let $\theta \in (0,1)$ be any irrational number, let $m \ge 1$ be an integer, and let $\Lambda$ be the path corresponding to the partition conditions for $p^+_{\theta}(m)$; see Remark~\ref{rmk:p+}.
\begin{itemize}
\item If $m$ is the denominator of a best approximation $m'/m$ to $\theta$ from below, then $p^+_{\theta}(m) = m$ and $\Lambda$ is a straight line from the origin to $(m,m')$.
\item Otherwise, let $k < m$ be the largest possible denominator of a best approximation $k'/k$ from below.  Then $\Lambda$ is given by concatenating the straight line from the origin to $(k,k')$ with the maximal concave path for $(m-k)$, and
\[p^+_{\theta}(m)=p^+_{\theta}(m-k) \sqcup (k).\]
\end{itemize}
\end{lemma}

\begin{proof}
The first bullet point follows from the fact that this straight line is a lattice path, and it is maximal by the definition of a best approximation from below.

To prove the second bullet point, we have to show that the claimed path $\Lambda$ is concave and maximal.  This path is the concatenation of a line with a concave path, so to see that it is concave, we just have to check that the second segment of $\Lambda$ does not have strictly greater slope than the first.  Assume that it does, and translate this second segment to be at the origin.  This translated segment is part of the concave path giving $p^+_{\theta}(m-k)$, so in particular it must be below the line $y = \theta x$.  It cannot have $x$-coordinate less than or equal to $k$, since $k$ is assumed the denominator of a best approximation.  And it cannot have $x$-coordinate more than $k$, since $k$ was assumed the largest denominator of a best approximation.  This is a contradiction.

To see that $\Lambda$ is maximal, assume otherwise, and consider the actual maximal concave path.  As in the previous paragraph, the first segment of this path must agree with the first segment of $\Lambda$.  Now let $(k+\delta, \lfloor (k + \delta)\theta \rfloor)$ be the lattice point that is the endpoint of the first segment of this path that does not agree with $\Lambda$.  Then $(\delta, \lfloor (k + \delta) \theta \rfloor - \lfloor k \theta \rfloor)$ is above the maximal concave path for $(m-k)$, and by \eqref{eqn:partitionequation}, it is below the line $y = \theta x$.  This is a contradiction.
\end{proof}

\begin{example}\label{ex:partit}\rm
Recall from Examples~\ref{ex:bigmono} and~\ref{ex:smallmono} that the even convergents of $\theta_n$ have denominators $t_k, k<n$, while those of
$\widetilde{\theta_n}$ have denominators $\ell_k, k\le n$.  Further,  the best approximation to $\theta_n$ with denominator $< Q_n$ is $\ell_n$.  Therefore $p^+_{\theta}(m) = (m)$ when $m = \ell_n, t_k, k<n$, while $p^+_{\widetilde{\theta_n}}(m) = (m)$  when $m = \ell_k, k\le n$.  \hfill$\er$
\end{example}

This has the following consequences for estimating $A(\theta,s)$.  To simplify the notation, given $\theta$ and a positive integer $s$, define
\begin{align}\label{eq:ka}
\ka(\theta,s) = s \cdot ( s \theta - \lfloor s \theta \rfloor).
\end{align}
  Then the discussion above implies:

\begin{lemma}
\label{lem:simpleestimates}
\begin{itemize}
\item If $p^+_{\theta}(m) = (m)$, then $2 A(\theta,m) = \ka(\theta,m)$.
\item If $p^+_{\theta}(m) = (a_1,\ldots,a_n)$, then
\begin{equation}
\label{eqn:additivity}
2A(\theta,m) \ge \sum_{i=1}^n \ka(\theta,a_i)
\end{equation}
with strict inequality unless $n=1$.
\item If $p^+_{\theta}(m)=(a,b)$ with $a \ge b$ then
\begin{equation}
\label{eqn:concatenation}
2A(\theta,m) = \ka(\theta,a) +\ka(\theta,b) + 2 \tfrac ba \ka(\theta,a).
\end{equation}
\end{itemize}
\end{lemma}

\begin{proof}
The first two bullet points follow from Lemma~\ref{lem:partcond} and the definition~\ref{eq:ka}. The third follows by observing that $2A(\theta,m) - \ka(\theta,a) - \ka(\theta,b)$ is given by twice the area of the parallelogram determined by the vectors $(0,a \theta - \lfloor a \theta \rfloor)$ and $(b, \lfloor m \theta \rfloor - \lfloor a \theta \rfloor),$  and computing this area with the two-dimensional cross product.
\end{proof}

\begin{example}\label{ex:n=1}
When $n=1$, we have $b_1 = \frac {55}8, \ell_1 = 7$ and $t_1 = 6, t_0=1$.    The integers $m \in \{2,\dots, 5\}$ are all lower semiconvergents to $\theta_1$, and  $p^+(\theta_1)(m) = (m)$ for $1\le m\le 7$, while
$p^+(\widetilde{\theta}_1)(m) = (1^{\times m})$  for $1\le m\le 6$.  Hence because $b_1 = 7-\frac 18$ we find that
\begin{align}\label{eq:n=1ka}
2A(\theta_1,m) & = \ka(\theta_1,m) = \frac {m(8-m)}8, \qquad 1\le m\le 7,\vspace{.1in} \\ \notag
 2A(\widetilde{\theta}_1,m) & =\ka(\widetilde{\theta}_1,m)  =  \frac{8m^2}{55}, \qquad 1\le m\le 6.
\end{align}
We use these calculations instead of Proposition~\ref{prop:estimates} in the case $n=1$. \hfill$\er$
\end{example}

The next lemma   estimates $\ka$ for $n>1$ by using some basic facts about continued fractions.

\begin{lemma}  Let $n>1$.
\label{lem:factsaboutc}
\begin{itemize}
\item[{\rm (a)}] If $m< \ell_n$ is the denominator of an even convergent of $\theta_n: = b_n+\eps$, then
\begin{equation}
\label{eqn:convergentestimate}
2A(\theta_n,m) = \ka(\theta_n,m) > 5/\tau^4 > 0.729.
\end{equation}
\item[{\rm (b)}]  If $t_{n-1} < m < \ell_n$ is the denominator of a lower semiconvergent of $\theta_n$, then
\begin{equation}
\label{eqn:semiconvergentestimate}
2A(\theta_n,m)= \ka(\theta_n,m) > 1.39.
\end{equation}
\item[{\rm (c)}]
If $m < \ell_n$ is the denominator of any lower semiconvergent of $\theta_n$, then
\begin{equation}
\label{eqn:semiconvergentestimate2}
2A(\theta_n,m)= \ka(\theta_n,m) > 1.28.
\end{equation}
\item[{\rm (d)}]  If $m < \ell_n$ is the denominator of an even convergent of $\widetilde{\theta}_n$ then
\begin{equation}
\label{eqn:convergentestimate2}
2A(\widetilde{\theta}_n,m) = \ka(\widetilde{\theta}_n,m) > \frac 7{48}> 0.1458 \quad \mbox{ if } n>1.
\end{equation}
\item[{\rm (e)}]\;   $\ka(\widetilde{\theta}_n,\ell_{n-1})=
{\displaystyle \frac{8\ell_{n-1}}{P_n} < \frac {48}{5\tau^8}}$.
\end{itemize}
\end{lemma}

\begin{proof}
 To prove (a), first let $c_{2k} = \frac{t_k}{t_{k-1}} = \frac{p_{2k}}{q_{2k}}$ be an even convergent with $0 < k < n-1$. (Here we use the notation for $c_{2k}$ in \eqref{eq:convgt}.)  We want to estimate
$q_{2k}( q_{2k} \theta_n - p_{2k}).$  It suffices to estimate $q_{2k}( q_{2k} b_n - p_{2k}) = q_{2k}^2(b_n - c_{2k}).$
By \eqref{eqn:fundamentalrelation} we have
\[ c_{2k} \oplus {\bf 5} c_{2k+1} = c_{2k+2} < b_n.\]
Thus,
\[ b_n - c_{2k} > c_{2k} \oplus {\bf 5} c_{2k+1} - c_{2k} = \frac{p_{2k} + 5p_{2k+1}}{q_{2k}+5q_{2k+1}} - \frac{p_{2k}}{q_{2k}} = \frac{5}{{q_{2k}(q_{2k}+5q_{2k+1})}},\]
where in the last equation we have used \eqref{eqn:anotherrelation}.  However, by \eqref{eqn:fundamentalrelation}, we have
\[ q_{2k}+5q_{2k+1} = q_{2k+2}.\]
Hence, we have
\[ b_n - c_{2k} > \frac{5}{q_{2k}(q_{2k+2})}\]
so
\[ q_{2k}^2 (b_n - c_{2k}) > 5 \frac{q_{2k}}{q_{2k+2}}=5 \frac{t_k}{t_{k+1}}.\]
The fractions $\frac{t_k}{t_{k+1}}= \frac{1}{c_{2k+2}}$ are decreasing with $k$ by Lemma~\ref{lem:fibid0}~(ii) and limit to $\frac{1}{\tau^4}$.  This proves the first bullet point in the case where
$0 < k < n-1$.
The quantity that we want to estimate for $c_0 = 6$ is $\theta_n - 6 > \tau^4 - 6$; this is also bigger than $\frac{5}{\tau^4}.$
\MS

The case $k = n-1$ is similar.   As above, it suffices to estimate
$$
q_{2n-2}(q_{2n-2} b_n - p_{2n-2})=q_{2n-2}^2 ( b_n - c_{2n-2}).
$$
By \eqref{eqn:fundamentalrelation}, we have
\[ b_n = c_{2n-2} \oplus {\bf 7} c_{2n-1},\]
which implies
\begin{align*} b_n - c_{2n-2} = c_{2n-2} \oplus {\bf 7} c_{2n-1} - c_{2n-2} &= \frac{p_{2n-2} + 7p_{2n-1}}{q_{2n-2}+7q_{2n-1}} - \frac{p_{2n-2}}{q_{2n-2}} \\
 &= \frac{7}{{q_{2n-2}(q_{2n-2}+7q_{2n-1})}} =  \frac{7}{q_{2n-2}\ q_{2n}},
\end{align*}
where the third equality uses \eqref{eqn:anotherrelation}.
We know that $\frac{t_{n-1}}{Q_n}$ is decreasing by 
part (ii) of Lemma~\ref{lem:fibid0},
and  \eqref{eqn:lnpnlimit}
implies that  
\[\frac{t_{n-1}}{Q_n} = \frac{\ell_{n-1}}{P_{n-1}} - \frac {\ell_{n-2}}{P_{n-1}}\]
converges to $\sigma - \frac{\sigma}{\tau^4}.$  Since $7(\sigma - \frac{\sigma}{\tau^4}) > \frac{5}{\tau^4}$, this proves  (a).

\vspace{3 mm}

\NI
{\em Proof of {\rm (b)} and {\rm (c)}.}  Let $p'/q' = c_{2k} \oplus {\bf r} c_{2k+1}$ be a lower semiconvergent.  As above, it suffices to estimate $q'^2(b_n - p'/q')$.

Assume first that $k = n-1$; this is the case for $m>t_{n-1}$.  To simplify the notation, let
\[p: = P_n= p_{2n-2} + 7p_{2n-1}, \quad q = Q_n = q_{2n-2} + 7q_{2n-1}.\]
Since $b_n = c_{2n-2} \oplus {\bf 7} c_{2n-1}$, we have
\[
b_n - (c_{2n-2} \oplus {\bf r} c_{2n-1}) = (c_{2n-2} \oplus {\bf 7} c_{2n-1}) - (c_{2n-2} \oplus {\bf r} c_{2n-1}) = \frac{7 - r}{qq'}.
\]
We are interested in the case  $q' < \ell_n$.  Since $\ell_n = t_{n-1} + 6 \ell_{n-1}$, it follows that $r < 6$.
By \eqref{eqn:fundamentalrelation}, we have
\[ q_{2k-1} = q_{2k-2} + q_{2k-3}, \quad q_{2k-2} = q_{2k-4} + 5q_{2k-3} < 6 q_{2k-3},\qquad  2\le k\le n.
\]
It follows that
\begin{equation}
\label{eqn:coarsebound}
q_{2k-1} \ge \frac{7}{6} q_{2k-2},\qquad 1\le k \le n.
\end{equation}
Since we also have $q_{2n-2} + 7 q_{2n-1} <  8 q_{2n-1}$, we have
\[ \frac{q'}{q} = \frac{q_{2n-2} +  r q_{2n-1}}{q_{2n-2} +  7 q_{2n-1}} \ge \frac{6 + 7r}{56}.\]
Hence
\[ \frac{ (7-r)q'}{q} \ge \frac{1}{56} (7-r)(6+7r).\]
This is minimized over integers $1 \le r \le 5$ for $r = 1$, in which case its value is larger than $1.39$.  This proves  (b).

If  $0 < k < n-1$ we are in case (c), and the proof is similar.  By \eqref{eqn:fundamentalrelation}, we have
\[ b_n - (c_{2k} \oplus {\bf r} c_{2k+1}) = (b_n - c_{2k+2}) + (c_{2k} \oplus {\bf 5} c_{2k+1} - c_{2k} \oplus {\bf r} c_{2k+1}).\]
As above, let $p = p_{2k} + 5 p_{2k+1}$ and $q = q_{2k} + 5 q_{2k+1}$.  As in the proof of  (a) we can bound
\[ q^2(b_n - c_{2k+2}) > \frac{5}{\tau^4},\]
so
\begin{equation}
\label{eqn:firstbound}
q'^2(b_n - c_{2k+2}) > \frac{5}{\tau^4} \frac{q'^2}{q^2}.
\end{equation}
Similarily to above, we also have
\begin{equation}
\label{eqn:secondbound}
(c_{2k} \oplus {\bf 5} c_{2k+1} - c_{2k} \oplus {\bf r} c_{2k+1}) = \frac{5-r}{qq'}.
\end{equation}
Since we also have $q_{2k} + 5q_{2k+1} < 6 q_{2k+1}$, and since \eqref{eqn:coarsebound} still applies, we have
\[ \frac{q'}{q} = \frac{q_{2k} + r q_{2k+1}}{q_{2k} + 5 q_{2k+1}} > \frac{r}{6} + \frac{1}{7}.\]
Putting this all together, we therefore have
\[ q'^2(b_n - c_{2k} \oplus {\bf r} c_{2k+1}) > \frac{5}{\tau^4}(\frac{q'}{q})^2 + (5-r) \frac{q'}{q} > \frac{5}{\tau^4} (\frac{r}{6} + \frac{1}{7})^2 + (5-r) (\frac{r}{6} + \frac{1}{7}).\]
The quantity $\frac{5}{\tau^4} (\frac{r}{6} + \frac{1}{7})^2 + (5-r) (\frac{r}{6} + \frac{1}{7})$ is minimized for $r\in \{1,2,3,4\}$ when $r = 4$, in which case it is larger than $1.28$.  This completes the proof of (c) in all cases except $k = 0$.

We have to treat the case where $k = 0$ slightly differently, because the estimate \eqref{eqn:coarsebound} no longer applies.  The estimates \eqref{eqn:firstbound} and \eqref{eqn:secondbound} still hold, so we have
\[ q'^2(b_n - c_{2k} \oplus {\bf r} c_{2k+1}) > \frac{5}{\tau^4}(\frac{q'}{q})^2 + (5-r) \frac{q'}{q}.\]
Since $q_0 = q_1 = 1$, we have $\frac{q'}{q} = \frac{1+r}{6}$.  Thus, we have
\[\frac{5}{\tau^4}(\frac{q'}{q})^2 + (5-r) \frac{q'}{q} = \frac{5}{\tau^4}\bigl(\frac {r+1}6\bigr)^2 + (5-r)\frac{r+1}6.\]
This is minimized over $r\in \{1,2,3,4\}$ when $r = 4$, in which case it is greater than $1.33$.

This completes the proof of (b) and (c).

\vspace{3 mm}

\NI {\it Proof of {\rm (d)}:}  Now consider convergents  to $\widetilde{\theta}_n: = \frac 1{b_n+\eps}$, that we also denote by $c_i = \frac{p_i}{q_i}$.
 We want to estimate  $q_{2k}^2(b_n - c_{2k})$.  Because $m < \ell_n$, we know that $k < n$.
Assume first that $0 < k$.  Then
\[ c_{2k} \oplus c_{2k+1} = c_{2k+2} < \frac 1{b_n}.\]
Thus,
\[ 1/b_n - c_{2k} > c_{2k} \oplus c_{2k+1} - c_{2k} = \frac{p_{2k} + p_{2k+1}}{q_{2k}+q_{2k+1}} - \frac{p_{2k}}{q_{2k}} = \frac{1}{{q_{2k}(q_{2k}+q_{2k+1})}},\]
where in the last equation we have used \eqref{eqn:anotherrelation}.  However, by \eqref{eqn:fundamentalrelation}, we have
\[ q_{2k}+q_{2k+1} = q_{2k+2},\]
so
\[ q_{2k}^2 (b_n - c_{2k}) >  \frac{q_{2k}}{q_{2k+2}}= \frac{\ell_k}{\ell_{k+1}}.\]

The fractions $\frac{\ell_k}{\ell_{k+1}}$ increase with $k$ by Lemma~\ref{lem:fibid0}  and so are $\ge \frac 7{48} \ge 0.1458$ for $k\ge 1$.
 When $k=0$,  because $n\ge 2$ 
 we have
$$ \ka(\widetilde{\theta}_n,k) = (\widetilde{\theta}_n - \lfloor \widetilde{\theta}_n \rfloor) \ge_\eps \frac {55}{377} > \frac 7{48}.
$$
\vspace{3 mm}

\NI {\it Proof of {\rm (e)}:}
We want to compute
\[ \ell_{n-1} \bigl( \frac{\ell_{n-1} }{b_n} - \lfloor \frac{\ell_{n-1} }{ b_n}  \rfloor\bigr),\quad n\ge 1.\]
But  $\ell_{n-1}P_{n-1}  - \ell_{n-2}P_n = 8$ by Lemma~\ref{lem:fibid0}~(i)  which implies that
\begin{align}\label{eq:lnbn}
\lfloor \frac{\ell_{n-1} }{b_n}  \rfloor = \lfloor  \frac{\ell_{n-1}P_{n-1} }{P_n} \rfloor =  \ell_{n-2}
\end{align}
and also gives the equality in (e).
The estimate on $\frac{8 \ell_{n-1}}{P_n}$ follows by observing that this is an increasing function of $n$
by Lemma~\ref{lem:fibid0}~(i)
with limit
$\frac{8\sigma}{\tau^4}< \frac 6{5\tau^4}$.
This completes the proof of Lemma~\ref{lem:factsaboutc}.
\end{proof}

\begin{proof}{Proof of Proposition~\ref{prop:estimates}.}
Part (i)  follows from additivity (see \eqref{eqn:additivity}),  together with the fact that by Lemma~\ref{lem:factsaboutc}~(a), each of the convergents and semiconvergents of $\theta_n$ contribute at least $5/\tau^4$ to the function $\ka$.

To prove (ii), note first that by Lemma~\ref{lem:partcond}, since $t > t_n$, the partition for $t$ must start with $t_n$.
Since $t_n$ is the denominator of a lower semiconvergent to $\theta_n$ we have $\ka(\theta_n, t_n)> 1.39$ by
Lemma~\ref{lem:factsaboutc}~(b).
Assume first that $t-t_n$ is not a convergent.  If it is a semiconvergent, then $p^+_{\theta_n}(t)=(t_n,t-t_n)$, and, by part (c) of Lemma~\ref{lem:factsaboutc}, $\ka(\theta_n,t-t_n)$ is more than $1.28$, so that by additivity, $2A(\theta_n,t) > 2.67.$  If $t-t_n$ is the sum of at least two convergents or semiconvergents, then we still have $2A(\theta_n,t) > 2.67$, by Lemma~\ref{lem:factsaboutc}~(b,c) and additivity.    Thus, if $2A(\theta_n,t) < 2.67$, then $t - t_n$ must be a convergent, hence $t = t_n + t_k$ for some $k < n$.  In this case, we have $2A(\theta_n,t) > 1.39 + 5/\tau^4,$ since $\ka(\theta_n,t_n) > 1.39$, by Lemma~\ref{lem:factsaboutc}.  This proves (ii).  The bound (iii) also follows from this, again by Lemma~\ref{lem:factsaboutc}~(b).

To prove (iii), note that $p^+_{\theta_n}(t)=(t_n,t_{n-1})$ by Lemma~\ref{lem:partcond}.  Now by \eqref{eqn:concatenation} and the discussions above, we have
\[
2A(\theta_n,t_n + t_{n-1}) > 1.39 + 5/\tau^4 + 2 \frac{t_{n-1}}{t_n} 1.39.
\]
But the fraction $ \frac{t_{n-1}}{t_n}$ is decreasing  by Lemma~\ref{lem:fibid0}, and limits to $1/\tau^4$.  Since $$
\frac{(2)(1.39)+5}{\tau^4} + 1.39 > 2.52,
$$
 this proves (iii).

Parts (iv) and (v)  follows from the fact that if $t_n - 2\ell_{n-1} < t < t_n$, then $t$ is a semiconvergent larger than $t_{n-1}$; one now applies Lemma~\ref{lem:factsaboutc} and \eqref{eqn:additivity} as above.

Parts (vi), (vii), and (viii)  follow from Lemma~\ref{lem:factsaboutc}~(d), (e), together with \eqref{eqn:additivity}.
\end{proof}

\subsection{Facts about the curves in the lowest level }\label{ss:light}

Assume throughout this section that $D_{12}$ is nonempty, i.e. there is at least one connector component.  The main result we prove here is that the heavy curve must
lie  in  $C_{LL}$ and
have exactly one end on $\beta_2$, of multiplicity $t = t_n$.

We start by showing that all curves in any neck level are covers of trivial cylinders.   For this, it is helpful to keep in mind that, as reviewed in \S\ref{sec:stablesetup}, for any curve $C$ in a symplectization level, the symplectic form $d\lambda$ is pointwise nonnegative on $C$, with equality at a point $y \in C$ if and only if the tangent space to $C$ at $y$ is the span of the Reeb vector field and $\partial_s$.  Therefore, any curve with zero action  is
{\bf trivial}, i.e. a union of covers of $\mathbb{R}$-invariant cylinders.
 Further any low action curve in the neck must have top and bottom with almost the same action, and hence, if its  top is
 $\{(\be_1,s),(\be_2,t)\}$ with $s,t< \ell_n$,   must in fact
have zero action  by Lemma~\ref{lem:actconsid},  and so be trivial.

\begin{lemma}\label{lem:conn}  Any symplectization level of $C_L$ must be a union of covers of trivial cylinders;  in particular $C_{LL}$ has top
$\{(\be_1,\ell_n),(\be_2,\ell_n)\}$.
\end{lemma}

\begin{proof}
By equations \eqref{eq:gr},  \eqref{eqn:Eei}, and \eqref{eq:weight}, a curve $C$ in $\widehat{\Ee}$ with top asymptotic to $\{(\be_1, P_n)\}$ passing through the constraints $W(\frac{P_n}{Q_n})$   has $I(C) = 0$.  On the other hand, we know that $I(C_L) = 4$,
and
because $J$ is generic, we know from \cite[Prop. 3.7]{hlecture} that the ECH index
$I$ of every curve in the neck
 is nonnegative.  Hence, because
 $$
 \gr(\{(\be_1,\ell_n),(\be_2,\ell_n)\})= \gr\{(\be_2,Q_n)\}+2 = \gr\{(\be_1,P_n)\}+4
 $$
  by  \eqref{eq:gr}, the top of $C_{LL}$ must be asymptotic to one of the orbit sets
  $\{(\be_1,\ell_n),(\be_2,\ell_n)\}$, $\{(\be_2,Q_n)\}$ or $\{(\be_1,P_n)\}$.  If the top level of $C_{LL}$ is $\{(\be_1,\ell_n),(\be_2,\ell_n)\}$ then by the action considerations explained above we are done. (Recall that, since we have a connector, $C_U$ necessarily has negative ends $\lbrace (\beta_1, \ell_{n}), (\beta_2, \ell_{n}) \rbrace$.) So, we can assume that the top level of $C_{LL}$ is either $\{(\be_2,Q_n)\}$ or $\{(\be_1,P_n)\}$.  In fact, our arguments for both of these cases will only use the fact that the asymptotics for this top level 
  are supported by only one of the $\beta_i$.
Assume this, and without loss of generality assume that $i=2$.

Since  there is a connector by assumption, $C_U$ must consist of at least $2$ irreducible curves.  Because we are stretching curves of genus $0$, it then follows that any irreducible curve in the neck has upper asymptotics at an orbit set
$\{(\be_1,s), (\be_2,t)\}$ with at least one of $s,t < \ell_n$.  Hence, by Lemma~\ref{lem:actconsid}, its action $\Aa(\{(\be_1,s), (\be_2,t)\})$ is strictly less than $P_n = \Aa\bigl(\{(\be_1,P_n)\}\bigr) =_\de \Aa\bigl(\{(\be_2, Q_n)\}\bigr))$.  Therefore,  if its top and bottom have different asymptotics, part (iii) of Lemma~\ref{lem:actconsid} implies that it cannot have low action.
In other words, any nontrivial  irreducible curve in the neck is heavy, which implies that there can be only one nontrivial matched component in the neck.

Call this component $C_{neck}$.   The curve $C_{neck}$ must have positive ends on $\beta_1$ of the maximum multiplicity $\ell_n$, since otherwise there would have to be another component in the neck with ends on both $\beta_1$ and $\beta_2$, which is therefore nontrivial.
Given this, $C_{neck}$ cannot have any positive ends on $\beta_2$ at all since our building must have genus $0$.

These positive asymptotics for $C_{neck}$ are not possible, however.  This is because the lowest level of any connector would have to meet $C_{neck}$, and by definition the neck components for the connector would also have positive ends on $\beta_2$.  But this also contradicts the fact that our building has genus $0$.
\end{proof}

  \begin{cor} \label{cor:conn} {\it  If an irreducible curve $C$ in $C_{LL}$  has top on  $\{(\be_1,s),(\be_2,t)\}$ then
  $s,t< \ell_n$ and $s + t \le \ell_n$.}
  \end{cor}

  \begin{proof}  If $C$ has $s,t>0$ so that it is part of a connector then the components of $C_U$ that are attached via trivial cylinders to $C$ along covers of $\be_1$ must be different from the components  of $C_U$ that are attached to $C$ via $\be_2$.  Hence $s+t\le \ell_n$.  Thus if $C$ is part of a connector it must have $s,t< \ell_n$. But if $C$ is not part of a connector, its top has to be disjoint from the top of any part of a connector.  Hence its top must have multiplicity $<\ell_n$.
  \end{proof}

 We next  analyze curves in the lowest level, i.e. those in the completion $\widehat\Ee$ of the blown up ellipsoid.

\begin{lemma}
\label{lem:nolowenergy} If a curve $C$ in $\widehat\Ee$ has low action then it cannot go through any constraints on the last block.
\end{lemma}

\begin{proof}  By Corollary~\ref{cor:conn}, we may suppose that $C$ has top end on the orbit set $\{(\be_1,s),(\be_2,t)\}$, where $s,t<\ell_n$ and $s+t\le \ell_n$.  
As above, we denote its constraint vector by $z$.
  \MS

  \NI
{\bf Case 1:}  {\it $C$ has ends just on $\be_1$, i.e. $t=0$.}

If $C$ has ends on $\{(\be_1,s)\}$  then \eqref{eqn:strongestimate}  implies that
\begin{align}\label{eqn:bound1}
 2A(\widetilde{\theta}_n,s) + \op{diff}_{C} \cdot \op{diff}_{C} &\le 1,
 \end{align}
where $\op{diff}_{C_1} $ is as in \eqref{eqn:diffdefn}.
Since $s<\ell_n$, when $n>1$ we may apply Proposition~\ref{prop:estimates}~(vi) to obtain
 \[ \op{diff}_{C} \cdot \op{diff}_{C} \le 1- 8/55.\]
 The same inequality holds when $n=1$ by Example~\ref{ex:n=1}.
Since $z\cdot w \le s< \ell_n$,
the final block of $\la w$ has value at most 
\[ \al: = \widetilde{\theta}_n\frac{s}{Q_n}  \le \frac{\ell_n}{P_n} < 0. 128,\]
where we use the fact that  $\frac{\ell_n}{P_n}$ is an increasing function of $n$ by Lemma~\ref{lem:fibid0}
that converges to $\sigma < .128$ by \eqref{eqn:lnpnlimit}.
If $C$ passes through $r$ of the constraints corresponding to the  last block,
the contribution to $ \op{diff}_{C} \cdot \op{diff}_{C}$ from this block is $(7-r)\al^2 + r(1-\al)^2$.  Since $\al<0.128$ we have $2(1-\al)^2>1$, so that $C$ can pass through at most one of these constraints.    But the minimum of  $6\al^2 + (1-\al)^2$ is taken when $\al = \frac 17$ and is $\frac 67$.  Since $\frac 67 + \frac 8{55}>1$ this case cannot occur. 
\MS

\NI
 {\bf Case 2:}  {\it $C$ has ends just on $\be_2$.}

The argument is essentially the same.
 As above we must have $2A(\theta_n, t) <1$.     Suppose first that $n>1$.
 Since $t< \ell_n$, it follows from Proposition~\ref{prop:estimates}~(ii) that $t\le t_n$.
 Hence the entries $\al$ of $\la w$ on the final block are at most $\frac{t_n}{Q_n} = \frac{t_n}{P_{n-1}}$, which is
a decreasing function of $n$  by Lemma~\ref{lem:fibid0}.   Hence
 $$
 \al < \frac{t_n}{Q_n} \le\frac{t_1}{Q_1} = \frac 34.
 $$
 If $\al\ge \frac 12$ then $ \op{diff}_{C} \cdot \op{diff}_{C}$ is minimized if $x=1$ on the last block.
 But then, using  Proposition~\ref{prop:estimates}~(i), we have
\[ 2A(\theta_n,t) +  \op{diff}_{C} \cdot \op{diff}_{C} \ge \frac{5}{\tau^4} + 7\cdot (1-.75)^2 > 1,\]
which is impossible.  Hence we must have $\al<\frac 12$.
But now the argument in Case 1 shows that the error is too large unless $z=0$ on the last block.
When $n=1$ the argument is similar, since we may use the estimates in Example~\ref{ex:n=1}.
\MS

\NI
 {\bf Case 3:}  {\it $C$ has ends  on both  $\be_1$ and $\be_2$.}

In this case \eqref{eqn:strongestimate} implies that
$$
\op{diff}_{C} \cdot \op{diff}_{C} \le 1 - 2A(\widetilde{\theta}_n,s) - 2A({\theta}_n, t) \le 1 - \frac 8{55} - \frac {5}{\tau^4} \approx 0.125.
$$
As in Case 2, since $t< \ell_n$, we must in fact have $t\le t_n$, since otherwise   $2A({\theta}_n, t)>1$.
Further  the value of $\la w$ on the last block is $\al = \frac{1}{Q_n}(s \widetilde{\theta}_n + t)$, which
is a maximum when $t = t_n$ and $s = \ell_n - t_n$.  Hence
\[\al  \le_\de  \frac{1}{Q_n}\bigl((\ell_n - t_n)\widetilde{\theta}_n + t_{n}\bigr) = \frac {\ell_{n-1}}{P_n} + \frac {t_n}{Q_n}.\]
By Lemma~\ref{lem:fibid0}, $\frac{\ell_{n-1}}{P_n}$ increases with limit $\frac {\sigma}{\tau^4} < 0.02$ while $ \frac {t_n}{Q_n}$  decreases and is $\le  \frac {t_1}{Q_1} = \frac 34$.  Since the contribution of the new term $\frac{\ell_{n-1}}{P_n}$ is so small, we can complete the argument as in Case 2.

This completes the proof.  \end{proof}

\begin{corollary}   The heavy curve lies in $C_{LL}$ and goes through all the constraints on the last block.  Moreover, it has some ends on $\be_2$.
\end{corollary}
\begin{proof}  The first claim holds because there is only one heavy curve.  The second holds by a slight generalization of Case 1 in the previous proof.  Since the curve is heavy, we need to use \eqref{eqn:weakerestimate} rather than \eqref{eqn:strongestimate}, which means that we need to add the term $2\frac{\ell_n}{Q_n} \widetilde{\theta}_n = 2\frac {\ell_n}{P_n}$  to the RHS of \eqref{eqn:bound1}.  As $\frac {\ell_n}{P_n}$ is an increasing sequence with limit less than $.13$, this term is also less than $.13$.  On the other hand, if the curve passed through all $7$ of the smallest constraints, then by the computations in Case 1, we would have the contribution from $\op{diff}_C \cdot \op{diff}_C$ at least $7(1-.13)^2$, which is far too large.
\end{proof}

We next establish the asymptotics of the heavy curve along $\be_2$.   Notice that we make no claims about $s$, which could be zero.

\begin{proposition}
\label{prop:hcomponent}
The heavy curve has just one end on $\be_2$ of multiplicity $t_n = \ell_n - \ell_{n-1}$.
\end{proposition}

\begin{proof}
We suppose as before that the heavy curve  $C$  is asymptotic to $\lbrace (\beta_1,s), (\beta_2, t) \rbrace$.  By Corollary~\ref{cor:conn} we know that $s+t\le \ell_n$, with $s,t< t_n$ and $t>0$.
By
 Proposition~\ref{prop:convertingtheproblemtoa}~(ii) and the bounds from Proposition~\ref{prop:estimates}, we have for $n>1$ that
\begin{align}\label{eqn:conn1}
2A(\theta_n,t)  & \le 1 + \frac{2}{Q_n}(s \widetilde{\theta}_n + t) - \op{diff}_C \cdot \op{diff}_C \quad\,\mbox{ if }\; s=0\\
\label{eqn:conn2}
2A(\theta_n,t)  & \le 1 + \frac{2}{Q_n}(s \widetilde{\theta}_n + t) - \op{diff}_C \cdot \op{diff}_C - \frac 7{48}  \quad\,\mbox{ if }\; s>0.
\end{align}

We now consider various cases.  We will first show that $t = t_n$ and then will discuss why it has  just one end on $\be_2$.

\vspace{3 mm}

\NI
{\bf Case 1: $t > t_n$:}\;\;
First observe that when $n = 1$ we have $t_1 = 6$ and $\ell_1 = 7$.  Therefore we cannot have $t_1 < t< \ell_1$.  Hence we may assume $n\ge 2$.

Next, notice that because $t<\ell_n$,
the maximum possible value of $\frac{1}{Q_n}(s \widetilde{\theta}_n + t)$ occurs when $t = \ell_n -1$ and $s = 1$, so that
$$
\al : = \frac{1}{Q_n}(s \widetilde{\theta}_n + t) <  \frac{ \ell_n}{Q_n}.
$$
The right hand side of this inequality is a decreasing function of $n$ by Lemma~\ref{lem:fibid0},
so it is no more than $ \frac{ \ell_1}{Q_1} = \frac {7}{8}$.  Since $C$ goes through all the constraints on the last block
we find that $\op{diff}_C \cdot \op{diff}_C$ is at least $\frac 7{64}> 0.1.$  Thus, the RHS of \eqref{eqn:conn1} is at most $2.65$ so that by
Proposition~\ref{prop:estimates}~(ii), we must have $t = t_n + t_k$ for some $k < n$.
Moreover, if  $s>0$ then
RHS of \eqref{eqn:conn2} is at most $2.51$ so that
Proposition~\ref{prop:estimates}~(iii)   implies that $t = t_n + t_k$ for some $0< k < n-1$.
In other words, if $s>0$ we must have
$t \le t_n + t_{n-2}$.

\MS

\NI {\bf Case 1(A):}  {\it $t> t_n$ and $s>0$. }
We saw above that we must have $t \le t_n + t_{n-2}$.
The maximum value of $\al$ then occurs when $t = t_n + t_{n-2}$ and $s = \ell_n-t_n - t_{n-2}< \ell_n-t_n  = \ell_{n-1}$.
By Lemma~\ref{lem:fibid0} and \eqref{eqn:lnpnlimit}, $\frac{\ell_{n-1}}{P_n}$ increases with limit $\frac{\si}{\tau^4}\le 0.02$.
On the other hand, both $\frac {t_{n-2}}{Q_{n}}$ and  $\frac {t_{n}}{Q_{n}}$, decrease with $n$ by Lemma~\ref{lem:fibid0}.  Hence because $n\ge 2$ we have
 $$
 \al \le  \frac{ t_n + t_{n-2}}{Q_n} + \frac{\ell_{n-1}}{P_n} < \frac {42}{55} + 0.02 < 0.784,
 $$
 so that the contribution to $ \op{diff}_C \cdot \op{diff}_C$ from   the last block  is at least  $ 7(1-\al)^2 > 0.326$.
 Therefore \eqref{eqn:conn2} implies that
 $$
  2A(\theta_n,t) \le 1 + 2\al - 0.326 - 0.146 < 2.1,
  $$
 which is impossible for $t>t_n$ by Proposition~\ref{prop:estimates}~(ii).
\MS

\NI {\bf Case 1(B):}  {\it $t_n + t_{n-1} \le t< \ell_n$ and $s=0$. }

We saw above that we must have  $t = t_n + t_{n-1}$, so that  $\al = \frac{t_n + t_{n-1}}{Q_n}$.
Now $\frac{t_n+ t_{n-1}}{Q_n}$ decreases by Lemma~\ref{lem:fibid0}.  Moreover we saw in  \eqref{eqn:lnpnlimit} that $\lim \frac{t_n}{Q_n} = 1-2\si$,
so that
$
\lim_{n\to \infty} \frac{t_n+ t_{n-1}}{Q_n}\; =\;  (1-\si) - \frac \si{ \tau^4}.
$
  Therefore
   $$
0.84< 1-0.13 \bigl(1 + \frac 16\bigr)<   \frac{t_n+ t_{n-1}}{Q_n}=\al \le \frac{t_2+t_1}{Q_2} = \frac{47}{55}< 0.86.
$$
The values of $\la w$  on the third to last block are $8\al$ and so lie in the interval $[6.72, 6.84]$, while those on the penultimate block
are $7\al\approx 5.88$.
Therefore the contribution to $\op{diff}_C \cdot \op{diff}_C$ from the last three blocks is at least
$$
7(1-\al)^2  + (6-7\al)^2 + 5(7-8\al)^2 = 288 - 658\al + 376\al^2.
$$
Therefore the RHS of \eqref{eqn:conn1} is
$$
1 + 2\al  -   \op{diff}_C \cdot \op{diff}_C\le - 287 +660\al - 376 \al^2 \le 2.43
$$
for $\al$ in the given interval,
where the last inequality was obtained by evaluating the quadratic expression at $\frac {47}{55}$ since it increases over this interval.
Hence this is impossible by Proposition~\ref{prop:estimates}~(iv).

 \MS

 \NI {\bf Case 1(C):}  {\it $t_n < t< t_n + t_{n-1}$ and $s=0$. }

 We saw above that in this case  $t = t_n + t_k \le  t_n + t_{n-2}$.
As  above, Lemma~\ref{lem:fibid0} implies that  $
 \al =  \frac{ t_n + t_{n-2}}{Q_n} $ decreases with $n$  and so is $\le \frac{42}{55} \approx 0.7636$.
 Moreover, since $
 \lim  \frac{ t_n}{Q_n} =1-2\si >0.745,
 $ by \eqref{eqn:lnpnlimit} the quantity $\al$ lies in the range
 $0.745 < \al < \frac {42}{55}$,
 Hence $4.5<7\al <5.5$, so that by looking at the last two blocks we find that
\begin{align}\label{eq:estim}
1 + 2\al  -   \op{diff}_C \cdot \op{diff}_C & \le 1 + 2\al - 7(1-\al)^2 - (5-7\al)^2 \\ \notag
&   = -31 +86\al - 56\al^2  < 2.02   \;\;\mbox { if } \al \in [\tfrac 9{14},\tfrac {11}{14}].
\end{align}
Hence again this scenario is impossible by Proposition~\ref{prop:estimates}~(ii).
 \MS

 \NI {\bf Case 2:}  \; {\it  $0<t<t_n$}.

 We begin with some general remarks.

 We may use the estimate in \eqref{eq:estim} for
$ 1 + 2\al  -   \op{diff}_C \cdot \op{diff}_C $ if  $\al\in [\frac9{14},\frac{11}{14}]$.
 On the other hand, if $\al < \frac 9{14}$ we get even better estimates.  Indeed, if $3.5 \le 7\al \le 4.5$  we can use the estimate
 \begin{align}\label{eq:estim2}
1 + 2\al  -   \op{diff}_C \cdot \op{diff}_C & \le 1 + 2\al - 7(1-\al)^2 - (4-7\al)^2 \\ \notag
&   = -22 +72\al - 56\al^2  \le   1.143, \;\;\mbox { if } \al \in [\tfrac 12,\tfrac 9{14}].
\end{align}
where the maximum is taken precisely at $\al = \frac 9{14}$.
 Further if $\al \le 0.5$ we have
  \begin{align}\label{eq:estim3}
1 + 2\al  -   \op{diff}_C \cdot \op{diff}_C & \le 1 + 2\al - 7(1-\al)^2 \le \tfrac 14 \;\;\mbox { if }  \al\le \tfrac 12.
\end{align}
Since $2A(\theta_n,t) > \frac 14$ in all cases with $t>0$, the case $\al< \frac 12$ never occurs.

The following information will also be useful.
  \begin{align}\label{eq:estim4}
&\mbox{For  $1\le r \le 3$ the sequence $\frac{t_n - r\ell_{n-1}}{Q_n} = \frac{t_n}{Q_n}-r\frac{\ell_{n-1}}{Q_n}$ decreases,}\\ \notag
&\mbox{since,  by Lemma~\ref{lem:fibid0}, $ \frac{t_n}{Q_n}$ decreases and $\frac{\ell_{n-1}}{Q_n}=\frac{\ell_{n-1}}{P_{n-1}}$ increases.}
\end{align}
   With these preliminaries in place, we can now analyze various cases.
We begin with the case $n=1$.
\MS

\NI {\bf Case 2(A):}\, {\it $n=1$ and $t <  t_1 = 6$.  }

As we saw in Example~\ref{ex:n=1}, in this case we have exact formulas for the terms $2A(\theta_1,m)$ and
$2A(\widetilde{\theta}_1,m)$.    The maximum value for $\al = \frac{1}{8}(s \widetilde{\theta}_1 + t)$
occurs when $t=5$ and $s=1$, in which case it is approximately $\frac 1{55} + \frac 58\approx 0.644 > \frac 9{14}$.
We may estimate the value of $\alpha$ at  $\frac 1{55} + \frac 58$ by calculating its value at $0.65$, which is  $1.24$.  Therefore, because $2A(\theta_1, 5) = \frac {15}8$, this case does not occur.
Similarly, the case $t_1 = 5$  and $s=0$ is impossible, because now $\al < \frac 9{14}$ so that  \eqref{eq:estim2} implies we must have
$2A(\theta_1, 5) < 1.12$.
But if $t< 5$ then $\al < \frac 9{14}$ for all $s$, while $2A(\theta_1, t)\ge 2A(\theta_1, 5)$ except if $t=1, 2$.    But in this case $\al<\frac12$ for all $s$, which is also impossible as we explained above.
\MS

As we now see, the argument for $n>1$ is similar, but more elaborate.
\MS

\NI {\bf Case 2(B):} {\it  $n>1, t_n-2\ell_{n-1} < t < t_n$ and any $s$.}

The maximum  value of $\frac{1}{Q_n}(s \widetilde{\theta}_n + t)$ occurs when $t = t_n -1$ and $s = \ell_n - t_n + 1$.
Because $\frac{\ell_{n-1}}{P_n}$ increases with $n$, and $\frac{t_n}{Q_n}$ decreases by Lemma~\ref{lem:fibid0},  we have
\begin{align}
\label{eqn:ttnbound}
& \al \le \frac{1}{Q_n}\Bigl( (\ell_n - t_n +1)\widetilde{\theta}_n + t_n - 1\Bigr) 
<  \frac{\ell_{n-1}}{P_n} +  \frac{t_n}{Q_n} < \frac{\sigma}{\tau^4} +  \frac{t_1}{Q_1} < 0.77 < \frac{11}{14}.
\end{align}
 Therefore, by equations \eqref{eq:estim}, \eqref{eq:estim2}  and \eqref{eq:estim3}, we must have
$$
2A(\widetilde{\theta}_n,s) + 2A({\theta}_n,t) \le 2.02.
 $$
 Hence,  by Proposition~\ref{prop:estimates}~(v) we must have $t = t_n - \ell_{n-1}$.
However, if  $t = t_n - \ell_{n-1}$ then $s\le \ell_n - t_n + \ell_{n-1} = 2\ell_{n-1}$, and \eqref{eq:estim4}, \eqref{eqn:ttnbound} imply
\begin{align*}
 \frac{t_n - \ell_{n-1}}{Q_n} & \;\le\;  \al \le \frac{t_n - \ell_{n-1}}{Q_n} + 2  \frac{\ell_{n-1}}{P_n} \\
&\;\le\;
\frac{t_2 - \ell_1}{Q_2} + 2  \frac{\sigma}{\tau^4}  \;=\; \frac{34}{55} + \frac {12}{377} < 0.651.
\end{align*}
Since $\al< \frac{43}{56} < \frac{11}{14}$,  we may estimate $1 + 2\al  -   \op{diff}_C \cdot \op{diff}_C$  by evaluating
 the quadratic expression in \eqref{eq:estim} (which increases with $\al$ for $\al < \frac{43}{56}$) at $\al = 0.656$. This gives the upper bound
 $1.37$, which is smaller than the allowed bound from Proposition~\ref{prop:estimates}~(vi).  Hence this case does not occur.

 \MS

 \NI {\bf Case 2(C):} {\it  $n>1,  t\le t_n-2\ell_{n-1}$ and any $s$.}

As in Case 2(B), it follows from \eqref{eq:estim4}  that
\begin{align*}
& \al \le  \frac{3\ell_{n-1}}{P_n} +  \frac{t_n - 2\ell_{n-1}}{Q_n} < \frac{3\sigma}{\tau^4} +  \frac{t_2-2\ell_1}{Q_2} < 0.55 < \frac{9}{14}.
\end{align*}
Therefore, by evaluating \eqref{eq:estim2} at $\alpha = .55$ we have
$$1 + 2\al  -   \op{diff}_C \cdot \op{diff}_C\le 0.67 < \frac 5{\tau^4}.
$$
Hence this case cannot occur by Proposition~\ref{prop:estimates}~(i).
 \MS

This completes the proof that $t = t_n$.  It remains to show that this component $C$ has just one end on $\be_2$.
  To see this, first note that, because  $b_n$ is a semiconvergent to $\theta_n$, Lemma~\ref{lem:simpleestimates}
  implies that $p^+_{\theta_n}(t_n) = (t_n)$, namely the length of the partition conditions for this $t$ is $1$.
 If  $C$ has
  two or more ends, then the $(r_2-n_2)$ term on the left hand side of \eqref{eqn:bound2} is strictly negative,
and hence is less  than the corresponding term $\max(r_2-n_2,0)$ in \eqref{eqn:ACbound}.
 Therefore,
    we can improve \eqref{eqn:weakerestimate} by subtracting $1$ from the right hand side, and so can improve all of the estimates in Case $2$ above by at least $1$ as well.

However, there are no values of $s$ for which $(s,t_n)$ satisfies these new estimates.  This is because by \eqref{eqn:ttnbound}, regardless of the value of $s$, we have that a strengthened version of either \eqref{eq:estim}, or a strengthened version of one of the stronger estimates \eqref{eq:estim2}, \eqref{eq:estim3} holds, and this is impossible by Proposition~\ref{prop:estimates}~(iii). 
 This completes the proof of Proposition~\ref{prop:hcomponent}.
 \end{proof}

\subsection{The asymptotics of the heavy curve}\label{ss:heavy}

By Proposition~\ref{prop:hcomponent}, the heavy curve must
pass through all the smallest constraints, and have a single end on $\be_2$ of multiplicity $t_n$.
In this subsection we improve this result as follows.

\begin{prop}
\label{prop:improva}
The heavy curve is a connector with exactly two ends asymptotic to $\lbrace (\beta_1,\ell_{n-1}), (\beta_2,\ell_n - \ell_{n-1}) \rbrace$. It has homology class $z_M: = z_M(n)$ given by taking $6 \cdot W\left(\frac{\ell_n}{\ell_{n-1}}\right)$, and appending the last block of $1$s to the end.
\end{prop}

We begin the argument by showing that there is an ECH index zero candidate $C_M$ for $C$ with the above properties, that we call the {\em model curve}.  Thus the curve $C_M$ has
$s = \ell_{n-1}, t = \ell_n - \ell_{n-1},$ and homology class (i.e. constraint vector)
\begin{align}\label{eq:c}
z_M: = \bigl(6W(\frac{\ell_n}{\ell_{n-1}}), 1^{\times 7}\bigr).
\end{align}
It follows from Lemma~\ref{lem:WcM} that $C_M$ has action precisely $\frac{1}{Q_n}$, so that it is a candidate for the heavy curve.

The calculations in Lemma~\ref{clm:homologyclass} show that it is consistent to require that $C_M$ have genus zero and ECH partitions at its ends.
 In the second step, we show that the connector $C$ must have the same numerics as $C_M$, i.e. the same homology class, genus, and multiplicities of ends.

\begin{lemma}
\label{clm:homologyclass} $I(C_M)=\ind(C_M)=0.$
\end{lemma}

\begin{proof}
By \eqref{eq:gr10}, we have
\[ gr(\beta_1^{\ell_{n-1}}, \beta_2^{\ell_n - \ell_{n-1}}) = gr(\beta_1^{\ell_n-1}) + gr(\beta_2^{\ell_n - \ell_{n-1}}) + 2 \ell_{n-1}(\ell_n - \ell_{n-1}).\]
Since $\ell_{n-1}$ is a lower semiconvergent of $\widetilde{\theta}_n$, and $\ell_n - \ell_{n-1}$ is a lower semiconvergent of $\theta_n$, 
Lemma~\ref{lem:partcond}  shows that
both ends have ECH partitions of length $1$.  Thus, by \eqref{eqn:grM},
in both cases $M(\theta, t)$ is the area of a triangle, and we have
\[ gr(\beta_1^{\ell_{n-1}}) = (\ell_{n-1} + 1) \lfloor \ell_{n-1} / b_n \rfloor + \ell_{n-1} + 1,\]
and
\[ gr(\beta_2^{t_n}) = (t_n + 1) \lfloor b_n (t_n) \rfloor + t_n + 1.\]
We  
saw in \eqref{eq:lnbn} that $\lfloor \ell_{n-1} / b_n \rfloor = \ell_{n-2}$, and we have $\lfloor b_n t_n \rfloor = t_{n+1}$, because of the identity
\begin{align}\label{eq:Pntn}
P_n t_n - P_{n-1}t_{n+1} = P_1 t_1 - P_{0}t_{2} = 2,
\end{align}
see Lemma~\ref{lem:fibid0}.  Thus, we have
\begin{align}\notag
\label{eqn:gradingequation}
gr(\beta_1^{\ell_{n-1}}, \beta_2^{\ell_n - \ell_{n-1}}) &= \ell_{n-2}(\ell_{n-1}+1) + \ell_{n-1} + 1 +  t_{n+1}(t_n+1) + t_n + 1 + 2\ell_{n-1}t_n\\
&= (7\ell_{n-1} - \ell_n) (\ell_{n-1}+1) + \ell_{n-1} + 1 + (6\ell_n - \ell_{n-1})(\ell_n - \ell_{n-1} + 1)\\ \notag
&\hspace{2in} + \ell_n - \ell_{n-1} + 1 + 2\ell_{n-1}(\ell_n - \ell_{n-1})\\ \notag
& =  6\ell_n^2  -6 \ell_{n}\ell_{n-1} + 6\ell_{n-1}^2 + 6\ell_n + 6\ell_{n-1} + 2,
\end{align}
where in the second line we substituted for $\ell_{n-2}, t_{n+1} , t_n$ in terms of $\ell_n, \ell_{n-1}$ using \eqref{eqn:basicrecursion}.
Finally, the identities~\eqref{eq:weight} satisfied by weight expansions imply that
\begin{equation}
\label{eqn:homologycorrection}
z_M \cdot z_M + z_M \cdot 1 = 36 \ell_n\ell_{n-1} + 6(\ell_n - 1 + \ell_{n-1} ) + 14.
\end{equation}
We now claim that the right hand sides of \eqref{eqn:homologycorrection} and \eqref{eqn:gradingequation} are equal. To see this, subtract the right hand side of  \eqref{eqn:homologycorrection} from  \eqref{eqn:gradingequation}  to obtain
$$
6(\ell_n^2 - 7\ell_n\ell_{n-1} + \ell_{n-1}^2 -1) = 6(\ell_n^2 - \ell_{n+1}\ell_{n-1} -1) = 0
$$
where the last step uses \eqref{eq:recurrel}.
Thus $I(C_M) = 0$ by \eqref{eqn:Eei}.

Since $C_M$ has ECH partitions by assumption, it  has two ends.
By \eqref{eq:weight},  $C_M$ goes through $c_H\cdot 1 = 6( \ell_n + \ell_{n-1}) + 1$ constraints.
We then use  \eqref{eqn:freddefn}, \eqref{eq:Frind1} and \eqref{eqn:blowupcalc} to obtain
\begin{align}
\label{eqn:indcm}
\tfrac 12 \ind(C_M) & = -1 + 2 + \ell_{n-1} + (\ell_n - \ell_{n-1})  + \lfloor \widetilde{\theta}_n \ell_{n-1}\rfloor + \lfloor {\theta_n} t_n\rfloor - 6( \ell_n + \ell_{n-1} ) -1 \nonumber \\
& =  \ell_n +  \lfloor\ell_{n-1}\frac{Q_n}{P_n}\rfloor  + \lfloor t_n \frac{P_n}{Q_n}\rfloor   - 6( \ell_n + \ell_{n-1}).
\end{align}
By \eqref{eq:lnbn} and \eqref{eq:Pntn} we have
\begin{align*}
\ell_{n-1} \frac{Q_n}{Q_{n+1}} = \ell_{n-2} + \frac 8{Q_{n+1}},
\qquad
 t_n \frac{P_n}{Q_n} = t_{n+1}+\frac{2}{Q_n}.
\end{align*}
Hence, the final line in \eqref{eqn:indcm} simplifies to
\[ \ell_n + \ell_{n-2} + t_{n+1} - 6(\ell_n + \ell_{n-1}) = \ell_{n+1} - 6\ell_n - 6\ell_{n-1} + \ell_{n-2} = 0\]
by \eqref{eqn:basicrecursion}.
\end{proof}

\begin{rmk}\rm
\label{rmk:model}
(i)
Notice that if $C_M$ were represented by a $J$-holomorphic curve, then it could not be multiply covered since it goes through some constraints with multiplicity one, and hence as proved by Hutchings the condition $I(C_M)=0$ would force the curve to have ECH partitions and $\ind(C_M) = 0$, see the index inequality Proposition~\ref{prop:indineq}.  Since we have not shown that $C_M$ must exist, however, we must verify some of this by direct computation.
\MS

\NI (ii) If one wanted to show that $C_M$  has a $J$-holomorphic representative, then one could probably prove this as follows.
Suppose for simplicity that $n=1$.  Then the connector has top ends on $\be_1, \be_2^6$ and  has constraint vector $x=(6^{\times 6}, 6, 1^{\times 7}) $; see Example~\ref{ex:558}.  It can be built by starting with a curve $C_1$ with top $\be_1$ through $E_7$ and a curve $C_2$ with top $\be_2^6$ through $6(E_1 + \dots + E_6) + 5E_7 + E_8 + \dots + E_{13}$.  These curves must intersect once (since a plane in $\Ee$ asymptotic to $\be_1$ intersects a plane
   asymptotic to $\be_2$ exactly once). One can check that these curves have $I(C) = \ind(C) = 0$; and can probably construct them by stretching suitable classes
as    outlined in Remark~\ref{rmk:8}.   Resolving  the point of intersection $C_1\cdot C_2$ gives a $2$-parameter family of curves with two positive ends, so that we can recover an index $0$ curve by imposing the constraint that it go through one more constraint, namely $E_{14}$. \hfill$\er$
\end{rmk}

We next investigate the homology class of $C$, that we write as $z_M - \Delta$, where $z_M$ is the \lq\lq model" set of constraints as in \eqref{eq:c}.
We know from  Proposition~\ref{prop:hcomponent}  that $\Delta$ must be zero on the last block.  Let
\begin{align}\label{eq:kw} k = w \cdot \Delta,
\end{align}
so that $C$ is asymptotic to $\lbrace (\beta_1,\ell_{n-1}-k) ,  (\beta_2,\ell_n - \ell_{n-1}) \rbrace.$  We interpret the case $k = \ell_{n-1}$ as corresponding to $C$ having no ends on $\beta_1$ at all.

Our main tool is the following dot product calculation, whose (rather technical) proof is deferred to the end of this subsection.

\begin{lemma}
\label{le:dotpr}
$\Delta \cdot z_M = 6k \ell_{n-1}$.
\end{lemma}

Granted this, we can now prove our main result.

\begin{proof}[Proof of Proposition~\ref{prop:improva}]  We first show that $C $ has the same numerics as  $C_M$, i.e. that
\begin{equation}
\label{eqn:desiredeq}
k=\De=0.
\end{equation}
To begin, we estimate the ECH index of $C$ as follows.
Since $I(C_M) = 0$, we  have:
\begin{align*}
I(C) & = I(C_M) - ( I(C_M) - I(C) ) \\
       & = 0 - \Bigl( \gr(\beta_1^{\ell_{n-1}}) - \gr(\beta_1^{\ell_{n-1} - k}) + 2t_nk + ( (z_M - \Delta) \cdot (z_M-\Delta) - z_M \cdot z_M) - \Delta \cdot 1\Bigr) \\
       &  = \Delta \cdot 1 -
\Bigl(                                                                                                                                                                                                                                                                                                                                                                                                                                                                                                        \gr(\beta_1^{\ell_{n-1}}) - \gr(\beta_1^{\ell_{n-1} - k}) + 2t_nk - 2 z_M \cdot \Delta + \Delta \cdot \Delta  \Bigr),
\end{align*}
where the term $2t_nk$ comes from the term $2m_1m_2$  in \eqref{eq:gr10}.
We compute the difference $(\gr(\beta_1^{\ell_{n-1}}) - \gr(\beta_1^{\ell_{n-1} - k}))$ by applying \eqref{eqn:agrading}, obtaining
\begin{align*}
\gr(\beta_1^{\ell_{n-1}}) - \gr(\beta_1^{\ell_{n-1} - k}) & = \widetilde{\theta}_n \bigl(\ell_{n-1}^2 - (\ell_{n-1} - k)^2\bigr) + k +  \lfloor \ell_{n-1} \widetilde{\theta}_n \rfloor - \lfloor (\ell_{n-1} - k) \widetilde{\theta}_n \rfloor + \\
 & \qquad\qquad (1-r) + 2A(\widetilde{\theta}_n,\ell_{n-1}-k) - 2A(\widetilde{\theta}_n,\ell_{n-1}),
\end{align*}
where $r$ is the length of the ECH partition $p^+_{ \widetilde{\theta}_n}(\ell_{n-1}-k)$.
Since $\ind(C_M) = 0$, we can also write
\begin{align*}
- \tfrac 12  \ind(C)& = \tfrac{1}{2}\bigl(\ind(C_M) - \ind(C)\bigr) \\
& =
(1-r_C) +k + \lfloor \ell_{n-1} \widetilde{\theta}_n \rfloor -  \sum_{i=1}^{r_C}\lfloor s_i \widetilde{\theta}_n,\rfloor
-\Delta \cdot 1
\end{align*}
where $r_C$ is the number of ends of $C$ on $\be_1$; we interpret $r_C=0$ if $C$ has no ends on $\beta_1$, and any sum with indices from $1$ to $r_C$ as equal to $0$ as well.
Furthermore, because
$$
k= \De\cdot w \le \sqrt{\De\cdot \De} \sqrt{w\cdot w } =  \sqrt{\De\cdot \De} \sqrt{b_n},
$$
we have
\begin{equation}
\label{eqn:deltadeltaestimate}
\Delta \cdot \Delta \ge \widetilde{\theta}_n k^2.
\end{equation}
By \eqref{eqn:deltadeltaestimate}, we know that $\Delta \ne 0$ if $k \ge 1$, thus if $k \ge 1$ the inequality \eqref{eqn:deltadeltaestimate} is strict, since $\Delta$ and $w$ cannot be parallel because the last block of $\Delta$ is identically $0$. If $k=0$, then the inequality \eqref{eqn:deltadeltaestimate} is also strict as long as $\Delta \ne 0$.  Thus, if \eqref{eqn:desiredeq} does not hold, and we set $z_M\cdot \De = 6k \ell_n$ by Lemma~\ref{le:dotpr},  we obtain
\begin{align}\label{eqn:Iineq} \notag
I(C) - \tfrac{1}{2} \ind(C) & < - \Bigl[2k(t_n - 6 \ell_{n-1} + \widetilde{\theta}_n\ell_{n-1}) +  2A(\widetilde{\theta}_n,\ell_{n-1}-k) - 2A(\widetilde{\theta}_n,\ell_{n-1})\Bigr]\\
 & \qquad\quad -\Bigl[r_C - r + \sum_{i=1}^{r_C}\lfloor (s_i \widetilde{\theta}_n) \rfloor - \lfloor (\ell_{n-1} - k) \widetilde{\theta}_n \rfloor \Bigl].
\end{align}

\NI
{\bf Claim.}  {\it The first term in square brackets above is nonnegative.  }
\begin{proof}[Proof of Claim.]

This is immediate if $k=0$.  So assume that $k>0$.
Notice first that
\begin{align}\label{eqn:tnk}
t_n - 6 \ell_{n-1} + \widetilde{\theta}_n\ell_{n-1} =_{\epsilon} \frac{8}{P_n} > 0,
\end{align}
since $t_n = \ell_n - \ell_{n-1}$,  $\widetilde{\theta}_n =_{\epsilon} \frac{Q_n}{P_n}$ and Lemma~\ref{lem:fibid0} implies that
\[ P_n(\ell_n - 7 \ell_{n-1}) + Q_n \ell_{n-1} = Q_n \ell_{n-1} - P_n \ell_{n-2} =  Q_1 \ell_{0} - P_1\ell_{-1}= 8.\]
 If $0< k < \ell_{n-1} - \ell_{n-2}$, then $\ell_{n-1}-k> \ell_{n-2}$ so that $2A(\widetilde{\theta}_n,\ell_{n-1}-k)\ge \frac{16}{55}$   by
 Proposition~\ref{prop:estimates}~(viii).

 Since $ \frac{16}{55}> \frac {8\si}{\tau^4}$ because $\si< 0.2$,
 part~(vii)  of the same proposition shows that $2A(\widetilde{\theta}_n,\ell_{n-1}-k) - 2A(\widetilde{\theta}_n,\ell_{n-1}) > 0,$ hence the claim.
If  $\ell_{n-1} > k \ge \ell_{n-1} - \ell_{n-2}$,
then  \eqref{eqn:tnk} implies that
\[ k(t_n - 6 \ell_{n-1} + \widetilde{\theta}_n\ell_{n-1}) \ge_{\delta} \frac{8 t_{n-1}}{P_n}, \]
so that
\begin{align*} & 2k(t_n - 6 \ell_{n-1} + \widetilde{\theta}_n\ell_{n-1}) + 2A(\widetilde{\theta}_n,\ell_{n-1}-k) - 2A(\widetilde{\theta}_n,\ell_{n-1})\\
&\qquad\qquad\qquad >
2k(t_n - 6 \ell_{n-1} + \widetilde{\theta}_n\ell_{n-1}) + \frac{8}{55} - 2A(\widetilde{\theta}_n,\ell_{n-1})\\
&\qquad\qquad\qquad  \ge_{\delta} \left(\frac{8 \ell_{n-1}}{P_n} - 2A(\widetilde{\theta}_n,\ell_{n-1})\right) + (\frac{8}{55} - \frac{8 \ell_{n-2}}{P_n}) > 0,
\end{align*}
where the last step uses  Proposition~\ref{prop:estimates}~(vii), and the fact that $\frac{\ell_{n-2}}{P_n}$ is an increasing sequence with limit $\frac{\sigma}{\tau^8}$.  If $k = \ell_{n-1}$, then the first term in square brackets is $0$.
 Thus in all cases, the claim holds.
\end{proof}

Thus in all cases, if \eqref{eqn:desiredeq} does not hold, we have
\[ I(C) - \tfrac{1}{2} \ind(C) <  (r-r_C) + \lfloor (\ell_{n-1} - k) \widetilde{\theta}_n \rfloor - \sum_{i=1}^{r_C}\lfloor s_i \widetilde{\theta}_n \rfloor ) .\]
However,
by \eqref{eqn:bound} and  \eqref{eqn:ACbound} in
 Proposition~\ref{prop:convertingtheproblemtoa},
we have
\[ I(C) - \tfrac{1}{2} \ind(C) \ge r-r_C +  \lfloor (\ell_{n-1} - k) \widetilde{\theta}_n \rfloor - \sum_{i=1}^{r_C}\lfloor s_i \widetilde{\theta}_n \rfloor.\]
This is a contradiction.  Hence we must have \eqref{eqn:desiredeq}.

It remains to show that the connector has just one end on $\be_1$.  This holds because of our initial assumption that there is a breaking with a connector, i.e.  by assumption the connector $C$ does exist as a holomorphic curve.  Since it has the same asymptotics as $C_M$, it has $I(C) = 0$, and because is simple (because it goes through some constraints with multiplicity one) it must therefore have
ECH partitions; see Remark~\ref{rmk:writhexact}~(iii).  The result now holds because $p^+_{\be_1}(\ell_{n-1}) = (\ell_{n-1})$.
\end{proof}

It remains to prove :
\begin{proof}[Proof of  Lemma~\ref{le:dotpr}]  We must show $\De\cdot z_M = 6k\ell_{n-1}$.  We do this in several steps.
\MS

 \NI
 {\em Step 1:  Applying the recursion for $z_M$ and $w$.}

Let $\widetilde{z}_M$ denote the homology class of $z_M$, with the last block removed, and define  $\widetilde{w}$
 analogously.
 Because $\Delta$ is supported away from the last block, we have
\begin{equation}
\label{eqn:trunok}
\Delta \cdot z_M = \Delta \cdot \widetilde{z}_M.
\end{equation}
Here and below, to simplify the notation we truncate the vector $\Delta$ without further comment by removing the last block of zeroes, so that expressions like $\Delta \cdot \tilde{z}_M$ are defined.  This is justified in view of \eqref{eqn:trunok}.

To simplify the discussion, we suppose for the moment that the entries of ${\Delta}$ are constant on the blocks of $W(b_n)$.
Thus,
\begin{equation}
\label{eqn:tildedeltaexp}
{\Delta} = \bigl(x_0^{\times 6}, x_1, x_2^{\times 5}, \dots, x_2^{\times 5}, \dots, x_{2n-2}^{\times 5}, x_{2n-1}\bigr).
\end{equation}
This assumption does require a slight loss of generality,  but below we will see that this is justified.
By the discussion after \eqref{eq:c} the vector $\widetilde{z}_M = 6W(\frac{\ell_n}{\ell_{n-1}})$ has the same block decomposition,
and in the notation of Lemma~\ref{lem:R} we may write
\[
\widetilde{z}_M
= 6( \ell_{n-1}^{\times 6}, t_{n-1}, \dots)
= 6\ell_{n-1} \cdot R(1,0) + 6t_{n-1} \cdot R(0,1)
\]
by Lemma~\ref{lem:R}(i).  Hence
$$
\Delta \cdot \widetilde{z}_M = 6\ell_{n-1} \De\cdot R(1,0) + 6 t_{n-1} \De\cdot R(0,1).
$$
Similarly, because the weight vector $w = w(b_n) = (1^{\times 6}, b_n-6,\dots) $  satisfies the same recursion on all but the last block (on which $\De=0$), we may invoke Lemma~\ref{lem:R} to write
$$
k = {\Delta} \cdot w =  \De\cdot R(1,0) +  (b_n-6)  \De\cdot R(0,1).
$$
Therefore,
\begin{align}\label{eqn:dot}
\Delta \cdot \widetilde{z}_M &= 6\ell_{n-1} k + 6\bigl(t_{n-1} - \ell_{n-1}(b_n-6)\bigr)\De\cdot R(0,1).
\end{align}
It remains to show that $\De\cdot R(0,1)=0$.
\MS

\NI
{\em Step 2:  We prove $|\De\cdot R(0,1)|< Q_n$ when $\De$ satisfies \eqref{eqn:tildedeltaexp}.}

Assume now that the entries $x_0, \ldots, x_{2n-1}$ of $\De$ satisfy the recursion in \eqref{eqn:R}.
Then Lemma~\ref{lem:R} implies that
\begin{align}\label{eq:RDe}
\De\cdot R(0,1) =
\ell_{n-1}x_{2n-1}.
\end{align}
It is possible that the $x_i$ do not satisfy \eqref{eqn:R}.  However, they differ from a sequence $\widetilde{x}_i$ that does by a small amount.  Namely, recall that we may write
\[ \widetilde{z}_M- {\Delta} = \lambda \widetilde{w} - \widetilde{\op{diff}}_C,\qquad \lambda: = \frac{(\widetilde{z}_M - {\De})\cdot w}{w\cdot w}. \]
So,
\[ {\Delta} = (\widetilde{z}_M - \lambda \widetilde{w}) +\widetilde{\op{diff}}_C.
\]
Let the $\widetilde{x}_i$ be the entries of $\widetilde{z}_M- \la \widetilde{w}$.
Then the $\widetilde{x}_i$ for $i \ge 2$ satisfy the recursion \eqref{eqn:R}.  Further,
we claim that the entries $z_i$ of $\widetilde{\op{diff}}_C$  satisfy
\begin{align}\label{eq:zi}
|z_i|<1,\qquad \forall i.
\end{align}
To see this, note that the sum of the two area terms $2A(\cdot,\cdot)$  on the left hand side of \eqref{eqn:weakerestimate} must be at least $1.39$ by Proposition~\ref{prop:estimates}~(iii);
the right hand side of \eqref{eqn:weakerestimate}  is 
 no more than $2.54$ by \eqref{eqn:ttnbound}; and the contribution to ${\op{diff}}_C \cdot \op{diff}_C$ from the last block must be at least $7(1-0.77)^2$, again by \eqref{eqn:ttnbound}.
Thus we have
$$
\sum_i z_i^2 \le   \op{diff}_C \cdot \op{diff}_C  - 7(1-0.77)^2 \le 2.54 -1.39 - 7(1-0.77)^2 < 1,
$$
which proves \eqref{eq:zi}.
Hence if ${\bf{I}}$ denotes the vector all of whose entries are $\pm 1$ with signs the same as those of the entries of $R(0,1)$, we find by replacing $\De$ in
 \eqref{eq:RDe} with $\widetilde{z}_M  - \la \widetilde {w}$ and using $|z_i|<1$, that
\begin{align*}
\De\cdot R(0,1)  & < \ell_{n-1}\widetilde{x}_{2n-1} +  {\bf {I}} \cdot R(0,1) \\
& = \ell_{n-1}\widetilde{x}_{2n-1} + \bigl(t_0 + t_1 + \ldots + t_{n-1} + 5(\ell_0 + \ell_1+ \ldots + \ell_{n-2})\bigr) \\
 & = \ell_{n-1} \widetilde{x}_{2n-1} + \ell_{n-1} +t_{n-1} - 1,
 \end{align*}
 where the first equality uses Lemma~\ref{lem:R}~(ii) and the second  the identities $t_k= \ell_k - \ell_{k-1}$  and  $5\ell_k = t_{k+1}-t_{k}$ from \eqref{eqn:basicrecursion}.

We next claim that  $|\widetilde{x}_{2n-1}| \le 1.$  To see this, note that
\[ \widetilde{x}_{2n-1} = 6 - \lambda \frac{7}{Q_n},\quad  \mbox{ where } t_n \le \lambda  \le \ell_{n-1}\widetilde{\theta}_n + t_n.\]
 Further,
\[ 7\frac{t_n}{Q_n} >  5,\]
since the fractions $\frac{t_n}{Q_n}$ are decreasing  by Lemma~\ref{lem:fibid0}, with limit $\sigma (\tau^4 - 1)$ by \eqref{eqn:lnpnlimit}.
We also claim that
 \[7 \left(\frac{\ell_{n-1}}{P_n} + \frac{t_n}{Q_n} \right)  < 6,\]
because $\frac{\ell_{n-1}}{P_n}$ are increasing  with limit $\frac{\sigma}{\tau^4}$, while $\frac{t_n}{Q_n}$ decreases and
\[7\left( \frac{\sigma}{\tau^4} + \frac{6}{8}\right) < 6.
\]
Thus, $0\le \widetilde{x}_{2n-1} \le 1$, so that
\begin{equation}
\label{eqn:tildegammabound}
|\De\cdot R(0,1)   | < \ell_{n-1} \widetilde{x}_{2n-1} + \ell_{n-1} + t_{n-1} - 1 <  2 \ell_{n-1} + t_{n-1} - 1<  Q_n,
\end{equation}
as claimed.
\MS

  \NI
{\em Step 3: Divisibility considerations:}\;\;
We now claim that $\De\cdot R(0,1)$ is divisible by $Q_{n}$.  To see this, note that ${\Delta} \cdot \widetilde{w} = \De\cdot w = k$ is an integer,
which implies that ${\Delta} \cdot \widetilde{W}$ is divisible by $Q_{n}$.  Therefore by Lemma~\ref{lem:R}~(iii)
$$
0\equiv {\Delta} \cdot \widetilde{W} \equiv -Q_{n-1}\, \De\cdot R(0,1) \pmod {Q_n}.
$$
  Since $Q_n, Q_{n-1}$ are relatively prime, this implies that $\De\cdot R(0,1)$ is a multiple of $Q_n$.    By
  \eqref{eqn:tildegammabound}, this implies  $\De\cdot R(0,1) = 0$.  Thus $\De\cdot \widetilde{z}_M = 6\ell_{n-1}k$ by \eqref{eqn:dot}.

 \vspace{3 mm}

 \NI
{\em Step 4:  Justifying the special form for ${\Delta}$:}\;\;
We have therefore proved  Lemma~\ref{le:dotpr}, except for the assumption that ${\Delta}$ can be written in the form \eqref{eqn:tildedeltaexp}, i.e. that its entries are constant on each block.
However, the above arguments only used information about the two dot products, ${\Delta} \cdot \widetilde{w}$, ${\Delta} \cdot \widetilde{z}_M$, and the fact that if $V=R(A,B) = (V_0^{\times 6}, V_1,V_2^{\times 5},\dots) $ for some integers $A,B$ then
\begin{align}\label{deV}
\De\cdot V & = \sum_i k_i V_i \;\;\mbox{ for some } k_i\in \Z.
\end{align}
 If $\Delta$ does not have the form in \eqref{eqn:tildedeltaexp}, rewrite it, taking the average value on each block of size $m$, where $ m= 5,6$.  This does not  change either of these dot products.  Moreover, because the new values of $x_i$ have the form $\frac {n_i}m$ (i.e. their denominator
 equals the length of the relevant block) the dot product $\De\cdot V$ 
 still satisfies \eqref{deV}.
 Hence the argument goes through even if $\De$ does not  have this special form.
\MS

   This completes the proof of Lemma~\ref{le:dotpr} and hence of Proposition~\ref{prop:improva}.  
\end{proof}

\vspace{3 mm}

\subsection{The rest of the proof}\label{ss:rest}

The next result completes  the proof of Proposition~\ref{prop:boundcurv}, and hence the proof of Theorem~\ref{thm:main}.

\begin{proposition}
\label{prop:up}  There is at most one breaking with a connector.
\end{proposition}

 To prove this, we locate the unique double point in the limiting building and then use the fact that $B\cdot B = 1$ to argue, as in the proof of Proposition~\ref{prop:bound} at the end of \S\ref{sec:noconn},
 that there can be only one breaking of this kind; see Corollary~\ref{cor:mainthm}.
 What is important is to show that this double point is a point  of intersection of two {\it different} connected components of the building:  if it were internal to one component, then there would be no obvious way to control the  number of nearly $J$-holomorphic representatives.

Consider a limiting building  $C_\infty$ that has a connector.  We saw in Lemma~\ref{lem:conn}  that the curves in the neck are all multiple covers of trivial cylinders.  Hence we can divide $C_{\infty}$ in a slightly different way than before (cf. Definition~\ref{def:CU}), taking the top level  to consist of the curves in $C_U$ together with those in the neck
(which are covers of trivial cylinders by Lemma~\ref{lem:conn}), and then dividing the curves in $C_{LL}$ into three groups according to the asymptotics at their top end.  Thus, we now consider the top level  to be a union of matched components as follows.

\begin{definition}\label{def:U}  {\it In this section,  $D_j$, for $j=1,2$,   denotes  the union of the curves in $\widehat\Ee$ with top end on $\be_j$, while the connector
$D_{12}$ is also a curve (rather than a matched component).  Further, we define $U_1$ to be the matched component in the upper  levels whose negative end connects to the positive end of $D_{12}$ at $\beta_1$; similarly,  $U_2$  connects to the positive end of $D_{12}$ at $\beta_2$.}
\end{definition}

 Then $U_1$ is at least an $\ell_{n-1}$-fold cover of the low action curve $C_0$, and $U_2$ is at least an $(\ell_{n}-\ell_{n-1})$-fold cover of $C_0$.  Hence, because $U_1, U_2$ contain different curves since the building has genus zero, it follows that
 \begin{itemize}\item
 the upper level is $U_1\cup U_2$,
 \item
 $U_1$ is precisely an $\ell_{n-1}$-fold cover of $C_0$ (extended by  trivial components in the neck), and
 \item $U_2$ is precisely a $(\ell_{n}-\ell_{n-1})$-fold cover of $C_0$.
 \end{itemize}
In particular, $U_1$ is in class $3 \ell_{n-1} \, L$ and $U_2$ is in class $3 t_n \, L$.

 We now argue much as in the proof of Proposition~\ref{prop:bound} at the end of \S\ref{sec:noconn}, except that we chop below rather than above the neck.
Consider a curve close to breaking along a building $C_\infty$ with a connector as above,
and chop the
nearly broken curve building close to the top of the lowest level, i.e. at the bottom of the neck region.
This gives compact curves with boundary  $U_{1,i}, U_{2,i}, D_{1,i}, D_{12,i}$ and $D_{2,i}$
defined from $U_1, U_2, D_1, D_{12}$ and $D_{2}$ and holomorphic with respect to a sequence of almost-complex structures $J^{R_i}$ with $R_i \to \infty$.  Note that because every $B$-curve has exactly one double point,
there can be at most one intersection point between these compact curves.
We now prove Proposition~\ref{prop:boundcurv} in several steps.

We begin with the following intersection alternative:

\begin{lemma}
\label{lem:intalt}
Either $U_{1,i}$ intersects $U_{2,i}$, or $(D_{1,i} \cup D_{2,i})$ intersects $D_{12,i}$.
\end{lemma}

\begin{proof}  Because $B$-curves are simple, the curves  $U_{1,i}\cup U_{2,i}$, $D_{1,i}\cup D_{2,i}$, and
$D_{12,i}$  are somewhere injective compact curves whose boundaries form links around the orbits $\be_1,\be_2$,
so that their intersection number can be calculated using the intersection formula \eqref{eqn:interstheq}.\footnote{In \S\ref{sec:prelim}, we stated this formula for punctured curves.  However, it is equally valid for curves with boundary obtained by truncating a completed cobordism by removing any $Y\times (R,\infty) $ or $Y\times (-\infty,-R)$ region for large $R$.  For a similar situation, see \cite[Lem 3.5]{HN}.}    Thus
\begin{equation}
\label{eqn:u1u2}
 U_{1,i}\cdot U_{2,i} = Q_{\tau}([U_{1,i}],[U_{2,i}]) + L_{\tau}(U_{1,i},U_{2,i}),
 \end{equation}
and similarly
\begin{equation}
\label{eqn:d1d2d12}
(D_{1,i}\cup D_{2,i}) \cdot D_{12,i}
= Q_{\tau}([D_{1,i} \cup D_{2,i}],[D_{12,i}]) + L_{\tau}(D_{1,i} \cup D_{2,i},D_{12,i}),
\end{equation}
where $L_{\tau}$ denotes the asymptotic linking number defined in \eqref{eq:linking}, and $Q_{\tau}$ is the relative intersection pairing whose formula is given in \eqref{eqn:cp2calc} and \eqref{eqn:blowupcalc}.
Because the negative ends of $U_{1,i}\cup U_{2,i}$ are the same as the positive ends of $D_{1,i} \cup D_{2,i}\cup D_{12,i}$ we have
\begin{equation}
\label{eqn:linkingequation}
L_{\tau}(U_{1,i},U_{2,i}) = - L_{\tau}(D_{1,i} \cup D_{2,i},D_{12,i}).
\end{equation}
Now, by \eqref{eqn:cp2calc}, we have
\begin{equation}
\label{eqn:qequation1}
Q_{\tau}([U_{1,i}],[U_{2,i}]) = 9 \ell_{n-1} t_n - 2 \ell_{n-1} t_n  = 7 \ell_n\ell_{n-1} - 7 \ell_{n-1}^2.
\end{equation}
To compute $Q_{\tau}([D_{1,i} \cup D_{1,i}],[D_{12,i}])$, note that by Proposition~\ref{prop:improva},
$[D_{12,i}] =  z_M$, and so $[D_{1,i}\cup D_{2,i}] = W - z_M$.  Thus by  \eqref{eqn:blowupcalc} and \eqref{eqn:cmdotW} we have:
\begin{align}
\label{eqn:qequation2}
Q_{\tau}([D_{1,i}\cup D_{2,i}],[D_{12,i}]) & = \ell_{n-1}^2 + t_n^2 - W \cdot z_M + z_M \cdot z_M\\
& =   \ell_{n-1}^2 + (\ell_n-\ell_{n-1})^2 - (\ell_n^2+41\ell_n\ell_{n-1} - 5 \ell_{n-1}^2 + 6) + \\ \notag
&\hspace{3in}(36 \ell_{n-1}\ell_n + 7) \nonumber  \\ \notag
& =  - (7 \ell_n\ell_{n-1} - \ell_{n-1}^2) + 1.
\end{align}

By combining \eqref{eqn:qequation1} and \eqref{eqn:qequation2}, it follows that
\begin{equation}
\label{eqn:homologycalculation}
Q_{\tau}([U_{1,i}],[U_{2,i}]) = - Q_{\tau}([D_{1,i} \cup D_{2,i}],[D_{12,i}]) + 1.
\end{equation}
Now combine \eqref{eqn:homologycalculation} with \eqref{eqn:u1u2}, \eqref{eqn:d1d2d12}.  This gives
\[  U_{1,i} \cdot U_{2,i} = 1 - (D_{1,i} \cup D_{2,i}) \cdot D_{12,i}.\]
Since $0 \le (D_{1,i} \cup D_{2,i}) \cdot D_{12,i} \le 1$, the claim now follows.
\end{proof}

\begin{corollary}
\label{cor:interscor}
$D_{1,i}$ and $D_{2,i}$ do not intersect.
\end{corollary}

\begin{proof}
Assume that they do intersect.  Then neither can intersect $D_{12,i}$, or else we would have too many intersection points.  Thus, by the previous step, $U_{1,i}$ and $U_{2,i}$ would have to intersect.  This also gives too many intersection points.
\end{proof}

The next step is the following uniqueness claim.

\begin{lemma}
\label{lem:homologicaluniqueness}
The constraint $z$ that is carried by the curve $D_{1,i}$ is  independent of the breaking.
\end{lemma}

Note that this implies the same statement for $D_{2,i}$, since the homology class of $D_{12,i}$ is fixed in view of
Proposition~\ref{prop:improva}.

\begin{proof}
Suppose given one breaking $D_{1,i}, D_{2,i}$ with constraints $z,y$ and another
 $D_{1,i}', D_{2,i}'$ with constraints $z',y'$.  Since $z+y = z'+y'$, we have $(z - z') \cdot (y' - y) = (z - z') \cdot (z - z')\ge 0$ (where we think of $z,y$ as vectors as in the previous section).
But because  $D_{1,i}, D_{1,i}'$ are asymptotic to $\be_1$ while $D_{2,i}, D_{2,i}'$ are asymptotic to $\be_2$, the contribution of the top ends to intersection numbers such as $D_{1,i}\cdot D_{2,i}'$ is fixed and equal to $T: = \ell_{n-1} \cdot (\ell_n - \ell_{n-1})$.
Hence
\begin{align*}
0 & \le  (z - z') \cdot (y' - y) = z \cdot y'+z'\cdot y - z' \cdot y' - z \cdot y    \\
 & = - \Big( (T - z \cdot y') + (T - z' \cdot y) - (T - z \cdot y ) - (T - z' \cdot t') \Big)\\
 & = - \Big( (T - z \cdot y') + (T - z' \cdot y)\Big)
\end{align*}
where the final equality above follows from Corollary~\ref{cor:interscor}.  But  $(T - z \cdot y')$ and $(T - z'\cdot y)$ both compute the number of intersections of $J$-holomorphic curves, and so are 
nonnegative.  Thus
$ (z - z') \cdot (y' - y) = 0,$
so that $z=z', y=y'$ as claimed.
\end{proof}

Now assume as above that $C_i$ and $C_i'$ are two different $B$-curves that are close to breaking into a building with a connector.  The final step is the following variant  of Lemma~\ref{lem:intalt}.

\begin{lemma}\label{lem:intalter}
\begin{itemize}
\item If $U_{1,i}$ does not intersect $U'_{2,i}$, then either $D_{12,i}$ intersects $D'_{1,i}$ or $D_{2,i}$ intersects $D'_{12,i}$.
\item If $U_{2,i}$ does not intersect $U'_{1,i}$, then either $D_{1,i}$ intersects $D'_{12,i}$, or $D_{12,i}$ intersects $D'_{2,i}$.
\end{itemize}
\end{lemma}

\begin{corollary}\label{cor:mainthm} Proposition~\ref{prop:up} holds.
\end{corollary}
\begin{proof} Since $C_i$ and $C_i'$ are $B$-curves, we have $C_i \cdot C_i' = 1$.  On the other hand, the intersection alternative in Lemma~\ref{lem:intalter} guarantees that there are at least two intersection points between $C_i$ and $C'_i$.
Hence this scenario cannot occur.
\end{proof}

The proof of  Lemma~\ref{lem:intalter} mimics that Lemma~\ref{lem:intalt}.  One has to be precise to get the relevant linking terms to cancel, however; hence the quite specific alternatives.

\begin{proof}[Proof of Lemma~\ref{lem:intalter}]
We begin by proving the first bullet point.

By \eqref{eqn:interstheq}, we have:
\begin{equation}
\label{eqn:u1u2'}
 U_{1,i}\cdot U'_{2,i} = Q_{\tau}([U_{1,i}],[U'_{2,i}]) + L_{\tau}(U_{1,i},U'_{2,i}),
 \end{equation}
where $L_{\tau}$ denotes the asymptotic linking number defined in \eqref{eq:linking}.
Further,
\begin{align}
 L_{\tau}(U_{1,i},U'_{2,i}) =   -  L_{\tau}(D_{12,i},D_{1,i}')- L_{\tau}(D_{2,i},D'_{12, i}),
 \end{align}
because, by Definition~\ref{def:U},  the negative ends of $U_{1,i}$ and $U'_{2,i}$ on $\beta_1$ are the same as the positive ends on $\beta_1$ of $D_{12,i}$ and $D'_{1,i}$ respectively,
while the negative ends of $U_{1,i}$ and $U'_{2,i}$ on $\beta_2$ are the same as the positive ends on $\beta_2$ of $D_{2,i}$ and $D'_{12,i}$.  Note also that
there is no linking between the end of $D_{12,i}$ at $\beta_2$ and the ends of $D'_{1,i}$,
nor any
between the end of $D'_{12,i}$ at $\beta_1$ and the ends of  $D_{2,i}$.
Hence
\begin{equation}
\label{eqn:d1d2d12'}
D_{12,i} \cdot D'_{1,i} + D_{2,i} \cdot D'_{12,i} = Q_{\tau}([D_{12,i}],[D'_{1,i}]) + Q_{\tau}([D_{2,i}],[D'_{12,i}]) - L_{\tau}(U_{1,i},U'_{2,i}).
\end{equation}
By Lemma~\ref{lem:homologicaluniqueness}, the constraints $z,z'$ on $D_{1,i}, D_{1,i}'$ are the same.  Hence
\[ Q_{\tau}([D_{12,i}],[D'_{1,i}]) = t_n^2 - z \cdot z_M, \quad Q_{\tau}([D_{2,i}],[D'_{12,i}]) = \ell_{n-1}^2 - y \cdot z_M,\]
so that
\begin{align}\notag
Q_{\tau}([D_{12,i}],[D'_{1,i}]) + Q_{\tau}([D_{2,i}],[D'_{12,i}])
 & = \ell_{n-1}^2 + t_n^2 - (z+y) \cdot z_M \\ \label{eqn:qequation4}
 & = \ell_{n-1}^2 + t_n^2 - (W - z_M) \cdot z_M,\\ \notag
 & = Q_{\tau}([D_{1,i} \cup D_{1,i}],[D_{12,i}]),
\end{align}
where the last equality holds by the first line of \eqref{eqn:qequation2}.
Equation \eqref{eqn:homologycalculation} now implies that
\[ Q_{\tau}([U_{1,i}],[U'_{2,i}]) = - (Q_{\tau}([D_{12,i}],[D'_{1,i}]) + Q_{\tau}([D_{2,i}],[D'_{12,i}]))+1,\]
since $Q_{\tau}([U_{1,i}],[U'_{2,i}])=Q_{\tau}([U_{1,i}],[U_{2,i}]).$  Combine this with \eqref{eqn:u1u2'} and \eqref{eqn:d1d2d12'} as in the proof of Lemma~\ref{lem:intalt}, we obtain
\[ U_{1,i}\cdot U'_{2,i}  = 1 - (D_{12,i} \cdot D'_{1,i} + D_{2,i} \cdot D'_{12,i}).\]
Since all intersection numbers are nonnegative, this completes the proof of the first bullet point.

 The proof of the second  is identical, modulo switching the roles of $C_i, C_i'$.
\end{proof}


\begin{thebibliography}{999999}
\bibitem[BH]{BH}  O. Buse and R. Hind, Ellipsoid embeddings and symplectic packing stability, {\it Compos. Math.} {\bf 149} (2013), no 5. 889--902.

\bibitem[CGHi]{CGH} D. Cristofaro-Gardiner  and R. Hind, Symplectic embeddings of products, arXiv:1508.02659.

\bibitem[CGHR]{CGHR} D. Cristofaro-Gardiner, M. Hutchings and V. Ramos,   The asymptotics of ECH capacities,
{\it Invent. Math.} {\bf 199} (2015), no. 1, 187--214.

\bibitem[CGLS]{CGLS} D. Cristofaro-Gardiner, T. Li, and R. Stanley,  New examples of period collapse, arXiv:1509.01887.

\bibitem[HW]{HW}  G.H. Hardy and E. W. Wright,  {\it An Introduction to the Theory of Numbers}, Sixth edition. Revised  by D.R. Heath-Brown and J.H. Silverman, Oxford University Press, Oxford, 2008.

\bibitem[Hi]{Hi}  R. Hind, Some optimal embeddings of symplectic ellipsoids, {\it J. Topol.} {\bf 8} (2015), no 3, 871--883.

\bibitem[HiK]{HK} R. Hind and E. Kerman, New obstructions to symplectic embeddings, {\it Invent. Math.} {\bf 196} (2014), no.2, 383--452.

\bibitem[HiK2]{HK2}   R. Hind and E. Kerman,  Erratum to New obstructions to symplectic embeddings, to appear.

\bibitem[H01]{HUT01} M. Hutchings,   An index inequality for embedded pseudoholomorphic curves in symplectizations,  (2001)

\bibitem[H]{hech} M. Hutchings, The embedded contact homology index revisited,  {\it New Perspectives and challenges in symplectic field theory}, 263--297, CRM Proc. Lecture Notes, {\bf 49}, Amer. Math Soc., Providence, R.I.

\bibitem[H2]{hlecture} M. Hutchings, Lecture notes on embedded contact homology, in {\it Contact and symplectic topology}, 389--484, Bolyai Soc. Math. Stud. {\bf 26}, {\it J\'anos Bolyai Math. Soc.}, Budapest, 2014.

\bibitem[H3]{qech} M. Hutchings, Quantitative embedded contact homology,  {\it J. Differential Geom.} {\bf 88} (2011) no 2, 231--266.

\bibitem[HN1]{HN} M. Hutchings  and J. Nelson, Cylindrical contact homology for dynamically convex contact forms in three dimensions,
arXiv:1407.2898, {\it J. Symp. Geom.}, to appear.

\bibitem[HT1]{HT} M. Hutchings and C. Taubes, Gluing pseudoholomorphic curves along branched covered cylinders, II.  {\it J. Symp. Geom.} {\bf 7} (2009), no 1. 29--133.

\bibitem[HT2]{HT2} M. Hutchings and C. Taubes, The Weinstein conjecture for stable Hamiltonian structures, {\it Geom. Topol.} {\bf 10} (2006), 169--266.

\bibitem[M1]{Mell}  D. McDuff,  Symplectic embeddings of $4$-dimensional ellipsoids, {\it J. Topol.} {\bf 2} (2009), no.1. 1-22,  Corrigendum: {\it J. Topol.} {\bf 8} (2015) no 4, 1119-1122.

\bibitem[M2]{M1}  D. McDuff, The Hofer conjecture on embedding symplectic ellipsoids, {\it J. Differential Geom.} {\bf 88} (2011), no 3, 519--532.

\bibitem[MS]{MS}   D. McDuff and F. Schlenk, The embedding capacity of four-dimensional symplectic ellipsoids, {\it Ann. of Math.} (2) {\bf  175} (2012), no 3, 1191--1282.

\bibitem[N]{N} J. Nelson, Automatic transversality  in contact homology I: regularity, arXiv:1407.3993.

\bibitem[PV]{pelayongoc} A. Pelayo and S. V\~{u} Ng\d{o}c, The Hofer question on intermediate symplectic capacities, {\it Proc. London Math. Soc.}, {\bf 110} (2015), 787--804.

\end{thebibliography}
\end{document}